\documentclass[11pt,reqno, a4paper]{amsart}
\usepackage{amsmath,amsthm,amssymb,amsfonts,amscd}
\usepackage{mathrsfs}
\usepackage{bbm}
\usepackage{bbding}
\usepackage{hyperref}
\usepackage{geometry}
\usepackage{color}
\usepackage{xcolor}

\usepackage{picture,epic}
\usepackage{tikz}

\numberwithin{equation}{section}

\theoremstyle{plain}
\newtheorem{theorem}{Theorem}[section]
\newtheorem{lemma}[theorem]{Lemma}
\newtheorem{corollary}[theorem]{Corollary}
\newtheorem{proposition}[theorem]{Proposition}

\theoremstyle{definition}

\theoremstyle{remark}
\newtheorem{remark}[theorem]{Remark}

\renewcommand{\Re}{\operatorname{Re}}
\renewcommand{\Im}{\operatorname{Im}}

\newcommand{\sgn}{\operatorname{sgn}}

\newcommand{\supp}{\operatorname{supp}}

\newcommand{\Sym}{\operatorname{Sym}}

\newcommand{\N}{\operatorname{N}}
\newcommand{\GL}{\operatorname{GL}}
\newcommand{\SL}{\operatorname{SL}}
\newcommand{\SO}{\operatorname{SO}}
\newcommand{\PSL}{\operatorname{PSL}}
\renewcommand{\mod}{\operatorname{mod}\ }

\newcommand{\dd}{\mathrm{d}}
\newcommand{\Res}{\operatorname{Res}}

\newcommand{\lcm}{\operatorname{lcm}}

\newcommand{\fixmehide}[1]{}
\newcommand{\fixmelater}[1]{}
\newcommand{\fixmedone}[1]{}
\newcommand{\fixmehidden}[1]{}
\newcommand{\tmop}[1]{\ensuremath{\operatorname{#1}}}

\makeatletter
\def\@tocline#1#2#3#4#5#6#7{\relax
  \ifnum #1>\c@tocdepth 
  \else
    \par \addpenalty\@secpenalty\addvspace{#2}%
    \begingroup \hyphenpenalty\@M
    \@ifempty{#4}{%
      \@tempdima\csname r@tocindent\number#1\endcsname\relax
    }{%
      \@tempdima#4\relax
    }%
    \parindent\z@ \leftskip#3\relax \advance\leftskip\@tempdima\relax
    \rightskip\@pnumwidth plus4em \parfillskip-\@pnumwidth
    #5\leavevmode\hskip-\@tempdima
      \ifcase #1
       \or\or \hskip 1em \or \hskip 2em \else \hskip 3em \fi%
      #6\nobreak\relax
    \hfill\hbox to\@pnumwidth{\@tocpagenum{#7}}\par
    \nobreak
    \endgroup
  \fi}
\makeatother

\begin{document}

\title[Quantum variance for dihedral Maass forms]
{Quantum variance for dihedral Maass forms} 
\author{Bingrong Huang and Stephen Lester}

\address{Data Science Institute and School of Mathematics \\ Shandong University \\ Jinan \\ Shandong 250100 \\China}
\email{brhuang@sdu.edu.cn}

\address{School of Mathematical Sciences, Queen Mary University of London, E1 4NS London, United Kingdom}
\email{s.lester@qmul.ac.uk}
\date{\today}

\begin{abstract}
  We establish
 an asymptotic formula for the weighted quantum variance of dihedral Maass forms on $\Gamma_0(D) \backslash \mathbb H$ in the large eigenvalue limit, for certain fixed $D$. As predicted in the physics literature, the resulting quadratic form is related to the classical variance of the geodesic flow on $\Gamma_0(D) \backslash \mathbb H$, but also includes factors that are sensitive to underlying arithmetic of the number field $\mathbb Q(\sqrt{D})$.

\end{abstract}

\keywords{Quantum variance, quantum covariance, dihedral Maass form, $L$-function, generalized Riemann Hypothesis, Generalized Ramanujan Conjecture} 


\maketitle
\setcounter{tocdepth}{1}
\tableofcontents

\section{Introduction} \label{sec:Intr}

In the theory of quantum chaos, an important problem is to understand the fluctuations of matrix coefficients of observables. In the generic chaotic case, it is conjectured in the physics literature \cite{FP1986,EFKAMM} that the leading order constant for the quantum
variance of such observables agrees with that of the classical variance of the dynamics of the geodesic flow.
Moreover, it is predicted that the limiting distribution of quantum fluctuations exists and is Gaussian.
These conjectures are wide open except for a handful of cases as described below.
This problem becomes even harder when only a sparse subsequence of all eigenfunctions is considered.
In this paper, we compute a new instance of the quantum fluctuations, for a distinguished sequence of eigenfunctions of the modular domain, called dihedral Maass forms.



Let $(M,g)$ be a Riemannian manifold (smooth, compact) with negative curvature and Riemmanian volume form $\tmop{dvol}_g$ normalized with $\tmop{vol}_g(M)=1$.
 Consider an orthonormal basis $\{\psi_j\}_j$ of $L^2(M, \tmop{dvol}_g)$ consisting of Laplace--Beltrami eigenfunctions and let $\lambda_j$ denote the Laplace--Beltrami eigenvalue of $\psi_j$. For an observable $\psi$ with $\int_M \psi(x) \tmop{dvol}_g(x)=0$ the quantum variance (in configuration space) is given by
\[
  Q(\psi;\Lambda)=\frac{1}{\# \{ \lambda_j \le \Lambda \}}\sum_{   \lambda_j \le \Lambda} \bigg| \int_{M} \psi(x) |\psi_j(x)|^2 \tmop{dvol}_g(x)\bigg|^2.
\]
Roughly, the quantum variance measures how far away the
$L^2$-mass of a typical eigenfunction is from being equidistributed on $M$ w.r.t. $\tmop{dvol}_g$.  Zelditch \cite{Zelditch1994} proved an upper bound for $Q(\psi;\Lambda)$, showing that $Q(\psi;\Lambda) =O((\log \Lambda)^{-1})$, which implies a quantitative form of the Quantum Ergodicity Theorem\footnote{For a negatively curved manifold, the Quantum Ergodicity Theorem (in configuration space) states  that along a density one subsequence of eigenfunctions $\{\psi_{j_{\ell}}\}_{\ell}$ that $\int_M \psi(x) |\psi_{j_{\ell}}(x)|^2 \tmop{dvol}_g(x) \rightarrow \int_M \psi(x)  \tmop{dvol}_g(x)$ as $\ell \rightarrow \infty$, for $\psi \in L^2(M,\tmop{dvol}_g)$.}.
Much stronger estimates are expected to hold and in \cite{EFKAMM} a fascinating asymptotic formula for $Q(\psi;\Lambda)$ is conjectured and it is predicted
for generic hyperbolic systems
that as $\Lambda \rightarrow \infty$
\[
Q(\psi;\Lambda) \sim g \frac{ V_M(\psi)}{\Lambda^{(d-1)/2}}
\]
where $V_M(\psi)$ is the classical variance of the geodesic flow $\mathcal G_t$ and is given by
\[
  V_M(\psi) =\int_{-\infty}^{\infty} \int_{S^*M} (\psi \circ \mathcal G_t)(x) \overline{\psi(x)} \dd \mu(x) \, \dd t
\]
where $\mu$ is Liouville measure on the unit cotangent bundle $S^*M$,
and $g=2$ if the system has time reversal symmetry, otherwise $g=1$.
 Moreover, it is predicted that the quantum fluctuations display Gaussian statistics.
These predictions beautifully mirror what is known about the fluctuations of observables under the geodesic flow. For $\psi$ a smooth observable on  $S^*M$ normalized with $\int_{S^*M} \psi(x) \dd \mu(x)=0$,
Sinai \cite{Sinai60} and Ratner \cite{Ratner1973} proved that $\frac{1}{T}\int_0^T (\psi \circ \mathcal G_t)(x) \dd t$, with $x \in S^*M$, has a Gaussian limiting distribution as $T \rightarrow \infty$ with mean zero and variance given by
\[
V_{M,T}(\psi)= \int_{S^*M} \left| \int_0^T (\psi \circ \mathcal G_t)(x) \dd t\right|^2 \, \dd\mu(x).
\]
 For further background and discussion see \cite[Section II]{BackerSchubertStifter1998} and \cite[Section 10]{Zelditch2010} and the references therein.

\medskip
%

There has been substantial prior work on computing the quantum variance in arithmetic settings.
 Luo--Sarnak \cite{Luo-Sarnak2004} computed a weighted analog of the quantum variance in the context of holomorphic modular forms for $\SL_2(\mathbb Z)$.
  Similar results for Hecke--Maass cusp forms were proved in Luo--Sarnak \cite{Luo-Sarnak2004} and extended by Zhao \cite{Zhao2010}.
  Recently, Sarnak--Zhao \cite{SarnakZhao2013} obtained the asymptotic formula of quantum variance for several phase space observables, that is for Hecke--Maass cusp forms on $\PSL_2(\mathbb{Z}) \backslash \PSL_2(\mathbb{R})$.
  Additionally, in the setting of quaternion algebras
  Nelson \cite{Nelson2016,Nelson2017,Nelson2019} has computed an asymptotic formula for the quantum variance. In each of
 the previously mentioned works
  the leading order constant of the quantum variance is given by classical variance of the geodesic flow $V(\psi)$ along with an additional arithmetic factor, which is related to the central value of an $L$-function.
  Notably, Kurlberg--Rudnick  \cite{KurlbergRudnick2005} computed the quantum variance in the setting of the quantized cat map and observed that the leading order constant for the quantum variance also deviates from the generic classical variance (see \cite[Section 6.1]{KurlbergRudnick2005}).

Let $\mathbb{H}=\{x+iy:y>0\}$  be the upper half plane.
Let $\Gamma_0(D)=\{\left(\begin{smallmatrix}
                   a & b \\
                   c & d
                 \end{smallmatrix}\right)\in\SL_2(\mathbb{Z}): c\equiv0 \ (\mod D)\}$
be the Hecke congruence group of level $D$.
Dihedral Maass forms comprise a sparse family of eigenfunctions of the Laplace operator on the surface $\mathbb{X}=\Gamma_0(D)\backslash\mathbb{H}$ equipped with hyperbolic measure $\dd x\dd y/y^2$.
The Weyl law, as established by Selberg, gives that the number of Laplace eigenvalues less than a large parameter $\Lambda$ (in the discrete part of the spectrum)\footnote{The spectrum of the Laplacian $\Delta = -y^2\left(\partial^2/\partial x^2 + \partial^2/\partial y^2 \right)$ on $\mathbb{X}$ has both  discrete and continuous components.
The discrete spectrum consists of the constants and the space $L_{\rm cusp}^2(\mathbb{X})$ of cusp forms,
for which we can take an orthonormal basis $\{u_j\}_{j\in\mathbb{N}}$ of Hecke--Maass forms.}
 is  $\sim c\Lambda$ where $c>0$ depends only on $D$, whereas the number of eigenvalues corresponding to dihedral Maass forms is of order of magnitude $\sqrt{\Lambda}$.  In this article, we investigate the behavior of fluctuations of $L^2$-mass of dihedral forms and compute the variance of these fluctuations. Our main result establishes an asymptotic formula for the weighted quantum variance of dihedral Maass forms, with leading order constant equal to $V(\psi)$ times an additional arithmetic factor.
 The arithmetic correction factor consists of a central value of an $L$-function, which is similar to the results mentioned earlier, as well as an additional arithmetic factor that reflects the underlying structure of the number field $\mathbb Q(\sqrt{D})$.


\subsection{Statement of the main results}

Let $F=\mathbb{Q}(\sqrt{D})$  be a fixed real quadratic field with discriminant $D>0$ squarefree and $D\equiv1$ (mod 4).
For simplicity, we assume that $F$ has narrow class
number 1 and $D$ is a product of two distinct primes congruent to 3 (mod 4).
For example, we may take $D = 21$.
Let $\omega_D=(1+\sqrt{D})/2$ and let $\epsilon_D>1$ be the fundamental unit of $F$.
Note that we have $\operatorname{N}(\epsilon_D)=\epsilon_D \tilde{\epsilon}_D=1$, where $\tilde{\alpha}$ is the conjugate of $\alpha \in F$ under the nontrivial automorphism of $F$.
The ring of integers of $F$ is $\mathcal{O}_F=\mathbb{Z}[\omega_D]$,
and the group of units $U_F$ in $\mathcal{O}_F$ is isomorphic to $\{\pm1\}\times \epsilon_D^{\mathbb{Z}}$.
For each integer $k\neq0$, we have the
Hecke Gr\"{o}ssencharacter $\Xi_k$ of $F$, which is defined by
\[
  \Xi_k((\alpha)) := \Big|\frac{\alpha}{\tilde{\alpha}}\Big|^{\frac{\pi i k}{\log \epsilon_D}} \quad
  \textrm{for an ideal $(\alpha)\subset \mathcal O_F$ with generator $\alpha$.}
\]
Let $\mathcal{B}_0^*(D,\chi_D)$ denote the set of $L^2$-normalized newforms of weight 0 for $\Gamma_0(D)$, with nebentypus character $\chi_D$ (the Kronecker symbol).
By \cite{maass1949uber}, we know that the theta-like series associated to $\Xi_k$ by 
\begin{equation} \label{eq:dihedraldef}
  \phi_k(z) := \rho_k(1) \; y^{1/2} \sum_{\substack{\mathfrak{a}\subset\mathcal{O}_F \\ \mathfrak{a}\neq \{0\}}}
  \Xi_k(\mathfrak{a}) K_{it_k}(2\pi \N(\mathfrak{a}) y) \big( e(\N(\mathfrak{a})x) + e(-\N(\mathfrak{a})x)\big) \in \mathcal{B}_0^*(D,\chi_D),
\end{equation}
where $z=x+iy\in\mathbb{H}$, $t_k := t_{\phi_k} = \pi k / \log \epsilon_D$ and
$\phi_k$ has Laplace eigenvalue $1/4+t_k^2$.
Here $\N(\mathfrak{a})=\# \mathcal{O}_F / \mathfrak{a}$ is the norm of a nonzero ideal $\mathfrak{a}\subset \mathcal{O}_F$, $K_\nu(z)$ is the modified Bessel function,  $e(x)=e^{2\pi ix}$, and
$\rho_k(1)$ is the positive real number such that $\phi_k$ is $L^2$-normalized, i.e.,
\[
  \| \phi_k \|_2^2 = \langle \phi_k,\phi_k \rangle_D =  \int_{\Gamma_0(D)\backslash\mathbb{H}} |\phi_k(z)|^2 \frac{\dd x\dd y}{y^2} = 1.
\]

This paper studies the distribution of $L^2$-mass for dihedral Maass forms.
For a test function $\psi:\mathbb{X}\rightarrow\mathbb{C}$, define
\begin{equation} \label{eq:measdef}
  \mu_k(\psi) := \langle \psi,|\phi_k|^2 \rangle = \int_{\mathbb{X}} \psi(z) |\phi_k(z)|^2 \frac{\dd x\dd y}{y^2}.
\end{equation}
Notably, Liu--Ye \cite{LiuYe2002} proved there exists $\delta>0$ such that
\begin{equation}\label{eq:QUE}
\mu_{k}(\psi)= \int_{\mathbb{X}} \psi(z)  \frac{\dd x\dd y}{y^2}+O(k^{-\delta}),
\end{equation}
that is, Quantum Unique Ergodicity (QUE) holds for such forms.

We are interested in the statistical fluctuations of the remainder term in \eqref{eq:QUE} as $k$ varies over integers $K < k \le 2K$, where $K$ is a large parameter.
To this end, let us first define the expected value of $\mu_k(\psi)$ as
\begin{equation}\label{eqn:EV}
    \mathbb{E}(\psi;K) := \frac{1}{K} \sum_{k\in \mathbb{Z}} \mu_k(\psi) \Phi\left( \frac{k}{K}\right),
\end{equation}
where $\Phi$ is a smooth function with compact support in $[\tfrac12,2]$.
It is not hard to prove that $\mathbb{E}(\psi;K)=o(K^{-1/2})$ under Generalized Riemman Hypothesis (GRH) (Remark \ref{rem:expectedvalue}), which is negligibly small. 
This motivates us to define the quantum covariance for dihedral Maass forms by
\[
  Q(\psi_1,\psi_2;K;\Phi)
  :=
  \sum_{k\in \mathbb{Z}}  \mu_k(\psi_1)
   \overline{\mu_k(\psi_2)}  \Phi\left( \frac{k}{K}\right)
\]
for $\psi_1,\psi_2\in L_{\rm cusp}^2(\mathbb{X})$.
We also define the harmonic weighted  quantum covariance by
\[
  Q^{\mathrm{h}}(\psi_1,\psi_2;K;\Phi)
  :=
  \sum_{k\in \mathbb{Z}} L(1,\phi_{2k})^2 \mu_k(\psi_1)
   \overline{\mu_k(\psi_2)}  \Phi\left( \frac{k}{K}\right),
\]
where $L(s,\phi_{2k})$ is the $L$-function of $\phi_{2k}$.
Note that if $\psi_1$ or $\psi_2$ is odd, then $Q^{\mathrm{h}}(\psi_1,\psi_2;K;\Phi) = Q(\psi_1,\psi_2;K;\Phi)=0$ and we will restrict to even Hecke--Maass forms in what follows.
In particular for $\psi\in \mathcal{B}_0^*(D)$, the set of $L^2$-normalized newforms of weight 0 for  $\Gamma_0(D)$  with trivial nebentypus,  the classical variance $V(\psi)$ has been explicitly calculated in
\cite[Appendix I]{Luo-Sarnak2004} where it was shown that
  \begin{equation}\label{eqn:V}
  \begin{split}
    V(\psi) &= \int_{-\infty}^{\infty} \int_{\Gamma_0(D)\backslash \PSL_2(\mathbb{R})} (\psi \circ \mathcal{G}_t)(g)
    \overline{\psi(g)} \dd \mu(g) \dd t \\
    &= \frac{\big|\Gamma(\frac{1}{4}+\frac{it_\psi}{2})\big|^4} {2\pi |\Gamma(\frac{1}{2}+it_\psi)|^2}
    \end{split}
  \end{equation}
  where $\tfrac14+t_{\psi}^2$ is the Laplace--Beltrami eigenvalue of $\psi$ and $\mu(g)$ is the hyperbolic measure on the unit cotangent space $\Gamma_0(D) \backslash \PSL_2(\mathbb R)$.

Also, let
 $L(s,\psi)$ and $L(s,\psi\times\chi_D)$ be the corresponding $L$-functions of $\psi$ and $\psi\times\chi_D$ where $\psi$ is a Hecke--Maass cusp form. Write $\lambda_{\psi}(n)$ for the $n$th Hecke eigenvalue of $\psi$. Let $\zeta(s)$ denote the Riemann zeta-function and $\zeta_D(s)=\zeta(s)\prod_{p|D}(1-p^{-s})$.
Define $\widetilde{\Phi}(s)=\int_{0}^{\infty} \Phi(x) x^{s-1} \dd x$ for the Mellin transform of $\Phi$.

The main result in this paper establishes the following asymptotic formula for the quantum variance.

\begin{theorem}\label{thm:QV}
  Let $\psi$ be an even Hecke--Maass cuspidal newform on $\Gamma_0(D)$. Then as $K \rightarrow \infty$ we have that
  \begin{equation} \label{eq:qv_formula}
    Q^{\mathrm{h}}(\psi,\psi;K;\Phi)
    =  \widetilde{\Phi}(0) A^{\mathrm{h}}(\psi) \, V(\psi)+o(1),
  \end{equation}
  where $V(\psi)$ is as given in \eqref{eqn:V} and
  \[
 A^{\mathrm{h}}(\psi) = L(\tfrac12,\psi) L(\tfrac12,\psi\times \chi_D) \cdot  \frac{\pi \log \epsilon_D}{2 D^{2}\zeta_D(2)  L(1,\chi_D)}
  \left(1+ \frac{\lambda_\psi(p_1)}{\sqrt{p_1}}+\frac{\lambda_\psi(p_2)}{\sqrt{p_2}} +\frac{\lambda_\psi(D)}{\sqrt{D}} \right).
  \]
Assume the Generalized Ramanujan Conjecture (GRC). Then as $K \rightarrow \infty$ we have that
  \[
    Q(\psi,\psi;K;\Phi)
   = \widetilde{\Phi}(0) A(\psi) \, V(\psi)+o(1),
  \]
  where $A(\psi) = A^{\mathrm{h}}(\psi) C_{D,\psi}'$, with $C_{D,\psi}'$ as in \eqref{eqn:C_Dpsi}.
\end{theorem}

  Notably, Theorem \ref{thm:QV} asserts that the quantum variance for dihedral Maass forms is equal to the classical variance only after inserting the ``arithmetic correction factor'' of
  $A(\psi)$.
  The arithmetic factor $A(\psi)$ inherits arithmetic information of the underlying real quadratic field (such as the regulator and prime factorization of the discriminant). It would be interesting to explore this relationship further.

\begin{remark}
  In the proof of Theorem \ref{thm:QV} we show that the error term in \eqref{eq:qv_formula} can be improved to $O(K^{-1/2+\vartheta+\varepsilon})$ where $\vartheta$ is the best known exponent towards GRC (currently it is known that $\vartheta \le 7/64$). Consequently,
  this implies that QUE holds for dihedral forms, thereby giving a new proof of the result of Liu--Ye \cite{LiuYe2002}.
\end{remark}

Now we consider the quantum covariance. We will prove the following conditional result.

\begin{theorem}\label{thm:cov}
  Assume GRH.
  Let $\psi_1,\psi_2$ be two orthogonal even Hecke--Maass cuspidal newforms. Then we have as $K \rightarrow \infty$ that
  \[
    Q(\psi_1,\psi_2;K;\Phi)\longrightarrow 0.
  \]
  In particular, the  quadratic form $Q = \lim_{K\rightarrow \infty} Q(\cdot,\cdot;K;\Phi)$ is diagonalized by the orthonormal basis of Hecke--Maass cuspidal newforms on $\mathcal{B}_0^*(D)$.
\end{theorem}

\begin{remark}
  It is instructive to compare our results for dihedral Maass forms with those for Eisenstein series.
  The key point here (from the automorphic representation theoretic point of view) is that if $F$ is a quadratic etale algebra over $\mathbb Q$ and $\chi$ is a Hecke character of $F$, then the principle of automorphic induction gives us an automorphic representation of ${\rm GL}_2(\mathbb Q)$. If $F$ is split, i.e., $F = \mathbb Q \times \mathbb Q$, then this leads to an Eisenstein series. On the other hand if $F$ is a field, then we end up getting a cusp form. The latter case is exactly what we consider in this paper.
  The quantum variance for Eisenstein series was considered by the first author in \cite{huang2018quantum}. The arithmetic correction factor in that case is $L(\tfrac12,\psi)^2$, and the covariance can be proved to be zero unconditionally.
\end{remark}

Additionally, applying our main estimate, Proposition \ref{prop:moment}, we can prove further results about the statistical fluctuations of $\mu_k(\psi)$. Using the method of Rudnick--Soundararajan \cite{rudnick-sound1,rudnick-sound2} for lower bounds for moments of $L$-functions we can show the moments of $\mu_k(\psi)$ blow up. More precisely,
under GRC we have  for any even integer $\ell \ge 2$ that
  \[
    K^{\ell-1} \sum_{K<k\leq 2K} |\mu_k(\psi)|^{2\ell} \gg (\log K)^{\ell(\ell-1)/2},
  \]
   provided $\Lambda(\tfrac12,\psi)\Lambda(\tfrac12,\psi\times \chi_D)\neq0$ and additional nonvanishing conditions hold for certain Dirichlet series related to $\psi$. Also, by following the argument of Siu \cite{Siu}, which  uses the method of Radziwi{\l\l}--Soundararajan \cite{radziwill-soundararajan}, it is possible to prove a one-sided central limit theorem for $\log |\mu_k(\psi)|$. Under GRC we have that
   \[
   \frac{1}{K} \# \bigg\{K < k \le 2K : \frac{\log |\mu_k(\psi)|+\frac12 \log K+\frac14 \log \log K}{\sqrt{\frac14 \log \log K}}  \ge V \bigg\} \le (1+o(1)) \frac{1}{\sqrt{2\pi}}\int_V^{\infty} e^{-t^2/2} \, \dd t,
   \]
   for any fixed $V \in \mathbb R$.


\subsection{Outline of the proofs}
The starting point of the proof of Theorem \ref{thm:QV} uses an extension of the Watson--Ichino \cite{Watson2008} formula due to Humphries--Khan \cite{Humphries-Khan}. This formula relates $|\mu_k(\psi)|^2$ to the central value of a certain triple product $L$-function. Since $\phi_k$ is a dihedral form this triple product $L$-function factors as follows $L(s,\psi \times \tmop{ad} \phi_k)=L(s,\psi \times \chi_D)L(s, \psi \times \phi_{2k})$, which is easy to see by comparing Euler products. This reduces the computation of the quantum variance to the calculation of a (twisted) first moment of the Rankin--Selberg $L$-function $L(\tfrac12, \psi \times \phi_{2k})$.  This moment estimate is related to previous work of Conrey--Snaith \cite{Conrey-Snaith}, who studied moments of $L$-functions attached to holomorphic CM modular forms $f_k$ and obtained an asymptotic formula for the first moment of $L(\tfrac12,f_k)$.
Our case should be compared with an asymptotic estimate of the second moment of $L(\tfrac12,f_k)$ which is not known yet. It would be interesting to extend our method to this problem.

To compute the twisted first moment of $L(\tfrac12,\psi \times \phi_{2k})$ we use Poisson summation and then separately analyze `diagonal' and `off-diagonal' terms. The diagonal terms arise from ideals $\mathfrak a \subset \mathcal O_{\mathbb Q(\sqrt{D})}$ such that for some  $\alpha$ with $\mathfrak a=(\alpha)$ either $|\frac{\alpha}{\widetilde \alpha}|=1$ or $|\frac{\alpha}{\widetilde \alpha}|=\varepsilon_D$. Since we restrict to real quadratic fields with totally positive fundamental unit both cases occur and account for separate diagonal terms.  These ideals correspond to lattice points lying on certain lines with rational slope, which motivates calling these diagonal terms.

We also need to bound the contribution of the ideals
whose angles lie nearby $0$ or $\log \epsilon_D$. These ideals are the
`off-diagonal' terms and correspond to lattice points  which lie in small neighborhoods of the lines mentioned above. We parameterize these ideals, thereby relating their contribution to sums of Fourier coefficients of Hecke--Maass forms over values of quadratic polynomials.
Such sums are known as `non-split sums' and have been studied in previous work of Hooley \cite{hooley}, Sarnak \cite{sarnak84}, Blomer \cite{blomer2008}, Templier \cite{templier} and Templier--Tsimerman \cite{templier-tsimerman}. Our approach to estimating these sums uses the spectral theory of half-integral weight automorphic forms. In particular we consider Dirichlet series whose $nth$ coefficient is roughly $\lambda_{\psi}(an^2+bn+c)$ and use Poincar\'e series together with theta functions to provide an analytic continuation of this Dirichlet series to $\tmop{Re}(s)>1/4+\vartheta/2$ except for a possible pole at $s=1/2$, which is related to the residual spectrum of the half-integral weight Laplace operator.  The possible pole at $s=1/2$ would contribute a term of the same size as our main term so we must show our Dirichlet series is analytic at $s=1/2$ (or equivalently its residue at $s=1/2$ is zero). Assuming GRC we can easily accomplish this by using well-known estimates for bounds of sums of multiplicative functions (see Nair--Tenenabum \cite{nair-tenenbaum} or Henriot \cite{henriot}).  Unconditionally, we proceed using a different argument, which relates this residue to the value of a twist of the symmetric square $L$-function of $\psi$ at $s=1$.
This is done by writing a theta series as a residue of a half-integral weight Eisenstein series, which we are able to accomplish in certain cases (see Appendix \ref{appendix:Eisenstein}).
However, in this argument we needed to assume that the coefficients of the quadratic polynomial satisfy certain hypotheses and this is reason we need to assume GRC to compute the unweighted quantum variance in Theorem \ref{thm:QV} (so that we can use the multiplicative function bounds).

  We emphasize that the condition that there is no unit of negative norm, i.e. $\N(\epsilon_D)=1$,  is important to our argument and the case $\N(\epsilon_D)=-1$ (e.g. $D=5$ or 13) is different. The reason is that when $\N(\epsilon_D)=-1$, there is no ideal with angle equal to $\log \epsilon_D$. This leads to the problem of bounding the contribution from ideals with angles close to $\log \epsilon_D$. While in our case with $\N(\epsilon_D)=1$, there are ideals with angles equal to $\log \epsilon_D$ and their contribution is not difficult to estimate. In addition, we have repulsion for the angles of the ideals, from which we find structure of the `off-diagonal' terms and then succeed to bound their contribution.


In contrast to prior work on the quantum covariance (such as \cite{SarnakZhao2013} and Nelson \cite{Nelson2016,Nelson2017, Nelson2019}) our  proof of Theorem \ref{thm:cov} uses the aforementioned Watson--Ichino formula together with estimates for moments of central values of Rankin--Selberg $L$-functions. This approach is more closely related to work of the first named author on the quantum variance of the Eisenstein series. However a substantial difference is for the covariance we instead compute the average of $(L(\tfrac12, \psi_1 \times \phi_{2k}) L(\tfrac12, \psi_2 \times \phi_{2k}))^{1/2}$ and under GRH show that this is $\ll \frac{1}{(\log k)^{1/4-o(1)}}$, so we only have a very weak bound on the covariance. Similar estimates appears in earlier works \cite[Theorem 1.5]{milinovich-TB} and \cite[Lemma 5.1]{lester-radziwill}. The basic approach of the proof follows Soundararajan's method \cite{Sound-moments} for bounds for moments of $L$-functions.



\subsection{Organization of the paper}

This paper is organized as follows.
In Section \ref{sec:prelim} we recall some facts on automorphic forms of weight 0 and of half-integral weights.
In Section \ref{sec:nonsplit} we establish bounds for sums of Fourier coefficients of Maass forms over non-split quadratic polynomials. These bounds are a key component in the estimate for the twisted first moment of $L(\tfrac12,\psi \times \phi_{2k})$, which is computed in Section \ref{sec:twisted1moment}. Using these results we complete the proof Theorem \ref{thm:QV} in Section \ref{sec:variance}. Theorem \ref{thm:cov} on quantum covariance is proved in Section \ref{sec:covariance}.
In Appendix \ref{appendix:Sarnak} we give an estimate for the triple product of automorphic forms, and in Appendix \ref{appendix:Eisenstein} we express our theta series in terms of the residue of certain half-integral weight Eisenstein series, both of which are needed in Section \ref{sec:nonsplit}.


\subsection{Acknowledgements} We would like to thank Ze\'ev Rudnick for helpful discussions and comments. B.H. is partially supported by the Young Taishan Scholars Program of Shandong Province (Grant No. tsqn201909046) and Qilu Young Scholar Program of Shandong University. S.L. is partially supported by EPSRC Standard Grant EP/T028343/1.

\section{Preliminaries} \label{sec:prelim}

\subsection{Automorphic forms of weight 0}\label{subsec:wt0}

Let $G=\SL_2(\mathbb{R})$, $K=\SO(2)$ and $\Gamma_0(N)$ be the Hecke congruence subgroup. Let $\mathbb H=G/K=\{x+iy:y>0\}$ be the upper-half plane.
We work on the space $L_{\mathrm{cusp}}^2(\Gamma_0(N)\backslash \mathbb{H})$ of cuspidal automorphic functions equipped with hyperbolic measure $\dd \mu(z)=\dd x \dd y/y^2$. A cuspidal function $\psi$ is one such that $\int_{0}^{1} \psi(x+iy) \dd x = 0$ for almost every $y$.

\medskip
Let $\psi\in L_{\mathrm{cusp}}^2(\Gamma_0(N)\backslash \mathbb{H})$ be a Hecke--Maass cusp form of weight zero.
The weight zero Laplacian reads
$\Delta=y^2(\frac{\partial^2}{\partial x^2}+\frac{\partial^2}{\partial x^2})$, and we have
$ \Delta \psi + \lambda_\psi \psi =0 $, where the eigenvalue  $\lambda_\psi = \tfrac14+ t_\psi^2$.
The spectral parameter $t_\psi$ belongs to $\mathbb R \cup [-\frac{i}{2},\frac{i}{2}]$.
Each such $\psi$ admits a Fourier expansion of the form
\begin{equation}
  \psi(z)= \varrho_\psi \sum_{n\neq 0} \frac{\lambda_\psi(n)}{\sqrt{|n|}} W_{0,it_{\psi}}(4\pi|n|y)e(nx),
 \quad
 z=x+iy\in\mathbb H,
\end{equation}
where $W_{\alpha,\beta}(y)$ is the Whittaker function (see e.g. \cite[\S 9.22-23]{GR}) and $\lambda_{\psi}(n)$ denotes the $n$th Hecke eigenvalue of $\psi$.
By convention, $\lambda_\psi(0)=0$.
The Hecke bound reads $\lambda_\psi(n)\ll |n|^{1/2}$. Under GRC $\lambda_\psi(n)\ll
|n|^\epsilon$ would hold. We have $\lambda_\psi(n)\ll |n|^{\vartheta+\epsilon}$, for all $\vartheta \geq 7/64$ is achieved in~\cite{KS}.

\medskip

Note that $\Psi(z) := \psi(\gamma_{4a}z)=\psi(4az) \in L^2_{\mathrm{cusp}}(\Gamma_{0}(4aN)\backslash \mathbb{H})$ with $\gamma_{4a}=\begin{pmatrix}4a & 0 \\ 0 & 1 \end{pmatrix}$ is a cusp form of level $4aN$ and we have
\begin{equation}\label{eqn:FE_Psi}
  \Psi(z)=  \varrho_\psi \sum_{n\neq 0} \frac{\lambda_\psi(n)}{\sqrt{|n|}} W_{0,it_\psi}(16a\pi|n|y) e(4an x).
\end{equation}

\subsection{Automorphic forms of half-integral weight}\label{sec:autom_half}

For odd $d$, define
\begin{equation}\label{eqn:epsidon_d}
  \epsilon_d := \left\{\begin{array}{ll}
                        1, & \textrm{if } d\equiv 1 \ (\mod 4), \\
                        i, & \textrm{if } d\equiv 3 \ (\mod 4).
                      \end{array} \right.
\end{equation}
Define
\begin{displaymath}
  J(\gamma, z) := \epsilon_d^{-1} \left(\frac{c}{d}\right) \left(\frac{|cz+d|}{cz+d}\right)^{-1/2} = \epsilon_d^{-1} \left(\frac{c}{d}\right) e^{i\frac{1}{2} \arg(cz+d)}
\end{displaymath}
for $\gamma = \left(\begin{smallmatrix}a & b\\ c & d\end{smallmatrix}\right) \in \Gamma_0(4)$ and $\Im z > 0$. Here $\left(\frac{c}{d}\right)$ is the extended Kronecker symbol as in \cite{Shimura}.
Note that
\[ J(\gamma,z) = \frac{\theta(\gamma z)}{\theta(z)}, \quad \gamma\in\Gamma_0(4), \]
where $\theta(z) = y^{1/4} \sum_{n\in\mathbb{Z}} e(n^2 z)$ is the standard theta series.

Let $M$ be a positive integer such that $4\mid M$ and $\chi$ be a character.
Also, let $\kappa\in \tfrac12+\mathbb{Z}$ be a half-integer
and $\Delta_{\kappa} :=  y^2(\partial_x^2 + \partial_y^2) -i\kappa y\partial_x$ be the Laplacian of weight $\kappa$.
Let $\textbf{H}_{\kappa}(M,\chi)$ be the Hilbert space of $L^2$-integrable functions $f$ satisfying
\begin{displaymath}
  f(\gamma z) = \chi(d) J(\gamma, z)^{2\kappa} f(z)
\end{displaymath}
for all $\gamma=\begin{pmatrix}
a& b \\
c & d
\end{pmatrix}
\in \Gamma_0(M)$.
For $t \in \mathbb{C}$, denote by $\textbf{H}_{\kappa}(M,\chi, t)$ the subspace of smooth functions $f \in \textbf{H}_{\kappa}(M,\chi)$ satisfying $(\Delta_{\kappa} + \lambda)f=0$ with $\lambda=\tfrac14 +  t^2$.
Without loss of generality we shall always assume $\Im t \geq 0$.

Let $\{f_{j,\kappa}\}_j$ with
\begin{equation}\label{eqn:f}
  f_{j,\kappa}(z) = \rho_{j,\kappa}(0, y) + \sum_{n \not=0} \rho_{j,\kappa}(n) W_{\sgn(n)\frac{\kappa}{2}, it_j}(4\pi |n|y)e(nx) \in \textbf{H}_{\kappa}(M,\chi),
\end{equation}
be a complete orthonormal system of $\textbf{H}_{\kappa}(M,\chi)$ where each $f_{j,k}$ is an eigenfunction of $\Delta_{\kappa}$ with eigenvalue $\lambda_j = \frac{1}{4} + t_j^2$, i.e. $(\Delta_{\kappa} + \lambda_j)f_{j,\kappa} = 0$.
We call $\lambda_j$ \emph{exceptional} if $t_j \not\in \mathbb{R}$, i.e.\ $\lambda_j < 1/4$.
The functions $f_{j,\kappa}$  may be cusp forms (in which case $\rho_{j,\kappa}(0, y) = 0$) or residues of possible poles of an Eisenstein series $E_{\mathfrak{a},\kappa}(z; s)$ defined as in \eqref{eqn:ES}.  In either case, $\lambda_j \geq 3/16$ (cf.\ the discussion preceding \eqref{eqn:KS} below).
If $u$ is any automorphic eigenfunction of $\Delta_{\kappa}$ with spectral parameter $t = \sqrt{\lambda-1/4}$, then its Shimura lift is an even weight Maass form with spectral parameter $2t$ (see e.g.\ \cite{Biro}). It is a cusp form  unless $f$ comes from theta functions, in which case $\lambda = 3/16$. In all other cases  the Kim--Sarnak bound \cite{KS} implies
\begin{equation}\label{eqn:KS}
   |\Im t | \leq \frac{\vartheta}{2} \leq \frac{7}{128}.
\end{equation}

%

\medskip

We next introduce the Eisenstein series.
For each equivalence class $\mathfrak{a}$ of cusps of $\Gamma_0(M)$, let $\Gamma_{\mathfrak{a}} := \{\gamma \in \Gamma_0(M) \mid \gamma\mathfrak{a} =\gamma\}$ be the stabilizer of $\mathfrak{a}$, $\sigma_{\mathfrak{a}} \in SL_2(\mathbb{R})$ be a scaling matrix (i.e. $\sigma_{\mathfrak{a}}\infty = \mathfrak{a}$ and $\sigma_{\mathfrak{a}}^{-1}\Gamma_{\mathfrak{a}} \sigma_{\mathfrak{a}} = \Gamma_{\infty}$) and $\gamma_{\mathfrak{a}} = \sigma_{\mathfrak{a}} \left(\begin{smallmatrix}1 & 1\\ & 1\end{smallmatrix}\right)\sigma_{\mathfrak{a}}^{-1} = \left(\begin{smallmatrix} \ast & \ast \\ c_{\mathfrak{a}} & d_{\mathfrak{a}} \end{smallmatrix} \right) \in \Gamma_0(N)$, say,  a generator of $\Gamma_{\mathfrak{a}}$.
A cusp $\mathfrak{a}$ is \emph{singular} for weight $\kappa$ and character $\chi$, if
\begin{displaymath}
  \chi(d_\mathfrak{a})\epsilon_{d_{\mathfrak{a}}}^{-2\kappa} \left(\frac{c_{\mathfrak{a}}}{d_{\mathfrak{a}}}\right) = 1.
\end{displaymath}
For a singular cusp $\mathfrak{a}$, let
\begin{equation}\label{eqn:ES}
  E_{\mathfrak{a},\kappa}(z; s) := \sum_{\gamma \in \Gamma_{\mathfrak{a}}\backslash \Gamma_0(M)}
  \bar\chi(\sigma_{\mathfrak{a}}^{-1}\gamma)  J(\sigma_{\mathfrak{a}}^{-1}\gamma, z)^{-2\kappa} \Im(\sigma_{\mathfrak{a}}^{-1}\gamma z)^s \qquad \tmop{Re}(s)>1
\end{equation}
be the \emph{Eisenstein series} attached to $\mathfrak{a}$. We write the Fourier expansion as
\begin{multline}\label{eqn:FE-ES}
  E_{\mathfrak{a},\kappa}(z; s)  = \delta_{\mathfrak{a} = \infty}y^s + \frac{\pi 4^{1-s} e\left(-\frac{\kappa}{4}\right) \Gamma(2s-1)}{\Gamma(s+\kappa/2)\Gamma(s-\kappa/2)} \phi_{\mathfrak{a}}(0, s)y^{1-s} \\
  + \sum_{n \not=0}
  \frac{\pi^s e\left(-\frac{\kappa}{4}\right) |n|^{s-1}}{\Gamma(s+\sgn(n)\frac{\kappa}{2})} \phi_{\mathfrak{a}}(n, s) W_{\sgn(n)\frac{\kappa}{2}, s-\frac{1}{2}}(4\pi|n|y) e(nx)
\end{multline}
where
\begin{displaymath}
  \phi_{\mathfrak{a}}(n, s) = \phi_{\mathfrak{a}}(n, s, \kappa)
  = \sum_{c>0} \sum_{\substack{1 \leq d \leq c\\ \left(\begin{smallmatrix}* & *\\ c & d\end{smallmatrix}\right)   \in  \sigma_{\mathfrak{a}}^{-1}\Gamma_0(M)}}
  \bar\chi(d)
    \left(\frac{c}{d}\right) \epsilon_d^{2\kappa} e\left(\frac{nd}{c}\right) c^{-2s},
\end{displaymath}
cf.\ \cite[p.\ 3876]{Proskurin} and \cite[\S 2]{Duke}.

\medskip

We will now further describe the exceptional spectrum of the half-integer weight Laplacian with $t=i/4$. There are Maass forms of weight $\kappa$ which have eigenvalue $3/16$; these can occur both as residues of Eisenstein series and as cusp forms.
Let $\Lambda_{\kappa} := \kappa/2 + y(i\partial_x - \partial_y)$ be the Maass lowering operator.
The space $\textbf{H}_{\kappa}(M,\chi, i/4)$ corresponding to the exceptional eigenvalue $3/16$ is the kernel of $\Lambda_{\kappa}$. Hence if $(\Delta_{\kappa} + 3/16)u = 0$, then $y^{-\kappa/2}u$ is holomorphic, and so $\textbf{H}_{1/2}(M,\chi, i/4) = \{y^{1/4}f \mid f \in M_{1/2}(M,\chi)\}$ and $\textbf{H}_{3/2}(M,\chi, i/4) = \{y^{3/4}f \mid f \in S_{3/2}(M,\chi)\}$ where $M_{\kappa}(M,\chi)$ is the space of holomorphic modular forms of weight $\kappa$, level $M$ and character $\chi$, and $S_{\kappa}(M,\chi)$  is the subspace of cusp forms.
Note that for $\kappa\in \tfrac32+2\mathbb{Z}_{\geq0}$, it follows from the presence of $\Gamma(s-\kappa/2)$ in the constant term of \eqref{eqn:FE-ES} that any form in $\textbf{H}_{3/2}(M,\chi, i/4)$ is cuspidal.
For $\kappa\in \tfrac12+2\mathbb{Z}_{\geq0}$, there exist an orthogonal basis of $\textbf{H}_{\kappa}(M,\chi, i/4)$ which are generated by theta series $\theta_{\omega,t}(z) := y^{1/4} \sum_{n\in\mathbb{Z}} \omega(n) e(t n^2 z)$ as described completely in \cite{Serre-Stark}.
In particular,  $u \in \textbf{H}_{\kappa}(M,\chi, i/4)$ has no  negative Fourier coefficients.

\medskip

Recall that we have the spectral decomposition of $\textbf{H}_{\kappa}(M,\chi)$ which consists of the following:
\begin{enumerate}
  \item[(i)] An orthonormal basis of cusp forms $f_{j,\kappa}$, where $\kappa$ is the weight and $t_j$ is the spectral parameter;
  \item[(ii)] an orthogonal basis of residual forms;
  \item[(iii)] a continuous spectrum provided by the Eisenstein series.
\end{enumerate}

\subsection{Poincar\'{e} series}

The Poincar\'e series of weight $\kappa$ and character $\chi$ is defined by
\begin{equation}\label{eqn:PS}
  P_{d,\kappa,\chi}(z;s):=\sum_{\gamma\in \Gamma_\infty \backslash \Gamma_0(M)} \bar\chi(\gamma) J(\gamma,z)^{-2\kappa}
  e^{-2\pi |d| \Im(\gamma z)} \Im(\gamma z)^s e(d \Re (\gamma z)),
\end{equation}
if $\Re(s)>1$.
Using the relation:
\begin{equation}
  [\Delta_{\kappa}+s(1-s)]P_{d,\kappa,\chi}(\cdot;s) =2\pi d(\kappa -2s \sgn(d))P_{d,\kappa,\chi}(\cdot;s+1),
\end{equation}
which follows from \cite[\S 2]{sarnak84}, the function $P_{d,\kappa,\chi}(\cdot;s)$ admits a meromorphic continuation to $\Re s >\frac{1+\vartheta}{2}$ with a simple pole at $s=\tfrac34$ (see also Fay \cite{Fay}).
We will later on use that the explicit value of the residue is given by theta series.

\section{Non-split sums of Fourier coefficients} \label{sec:nonsplit}
\subsection{Notation}
Let $\psi$ be a Hecke-Maass form of level $N$ and trivial nebentypus.
In this section we estimate
\[
  \mathcal{S} := \sum_{n\geq1} \lambda_\psi(an^2 + bn + c)
   W\left( \frac{an^2+ bn+c}{Y} \right)
\]
where $a,b,c\in\mathbb{Z}$\, $0<|a| \ll Q$, $b\ll QR$, $c\ll QR^2$ and $\Delta:=b^2-4ac>0$, $W(x)$ is a smooth function with compact support in $[1,2]$ and $W^{(j)}(x)\ll P^j$. Here the parameters $P,Q,R$ satisfy $P,Q,R \le Y^{\delta}$ for some $\delta>0$ sufficiently small. WLOG we may assume $a>0$. Recall that $\widetilde W(s):=\int_0^{\infty} W(x) x^{s-1} \, dx$ is the Mellin transform of $W$.
Also, recall $\vartheta \in [0,7/64]$ is the best known exponent towards GRC. By \cite{KS}, we can take $\vartheta=7/64$.

\subsection{Main result}  In the remainder of this section we will establish the following estimates for $\mathcal S$.

\begin{theorem}\label{prop:sum_quad}
Let $C_{\psi,a,b,c}$ be as defined in \eqref{eqn:C}. Each of the following holds:
  \begin{itemize}
    \item[(i)] There exists $B>0$ such that
        \[
        \mathcal{S} = C_{\psi,a,b,c} \widetilde{W}(1/2) Y^{1/2} + O\left(Y^{1/4+\vartheta/2+\varepsilon} (PQR)^{B}\right),
        \]
        where the implied constant depends on at most $\psi$ and $\varepsilon$.
    \item[(ii)] If $N=p_1p_2$ where $p_1\equiv p_2 \equiv 3\pmod{4}$ are distinct primes, $a\mid N$ and $a\mid b$ then we have that $C_{\psi,a,b,c}=0$.
    \item[(iii)] Assume GRC. Then we have that $C_{\psi,a,b,c}=0$.
  \end{itemize}
\end{theorem}

\begin{remark}
  Non-split sums of the divisor function or Fourier coefficients were considered in \cite{hooley,sarnak84,blomer2008,templier,templier-tsimerman}, to name a few.  Our approach uses Poincar\'e series along with a triple product estimate, since this approach gives us an explicit expression of the constant $C_{\psi,a,b,c}$. To make this method work, we need to  extend Sarnak's method to bounding the triple product of two cusp forms and a theta series (see Appendix \ref{appendix:Sarnak}).
\end{remark}

\begin{remark}
  To show $C_{\psi,a,b,c}=0$ unconditionally, the assumptions on $N$ and $a\mid N$ may not be essential, but our argument does need an assumption such as $a\mid b$, in which case we can write the theta function as the residue of certain Eisenstein series (see Appendix \ref{appendix:Eisenstein}) and then use the symmetric square lift $L$-function of $\psi$ to show that the constant $C_{\psi,a,b,c}=0$.
  For a Maass cusp form $\psi$, previous works only  consider the case $b=0$, so the assumption $a\mid b$ holds.
\end{remark}

\subsection{Reduction}\label{sec:reduction}

In this section we prove the following lemma, which reduces the problem of estimating $\mathcal{S}$ to analytic estimates of certain Dirichlet series. We define $\lambda_\psi(x)=0$ if $x\in\mathbb{R}\setminus\mathbb{Z}$.
\begin{lemma}\label{lemma:S=int(D)}
 Let $d=(2a,b)$, $a'=2a/d$, $b'=b/d$, and $\varphi$ denote the Euler totient function. Also, let $\delta_{n,0}=1$ if $n=0$, and $0$ otherwise.
We have that
  \[
    \mathcal{S} =  \frac{1}{4\pi i}
  \int_{(1)}  D_{\psi}(s,\Delta)  \widetilde{W}(s) (8aY)^{s} \dd s + O\Big(\frac{P\Delta}{Y^{1/2-\vartheta}}\Big),
  \]
  where
  \[
    D_{\psi}(s,\Delta) := \frac{1}{\varphi(a')} \sum_{\substack{\chi \; (a') }} \bar\chi(b') 2^{\nu_\chi/2} d^{\nu_\chi} D_{\psi,\chi,d^2}(s,\Delta)
  \]
  and
  \begin{equation}\label{eqn:Dchi}
    D_{\psi,\chi,t}(s,\Delta) := \mathop{\sum_{n\geq0}}
    \frac{\lambda_\psi\left(\frac{t n^2-\Delta}{4a}\right) (2-\delta_{n,0}) \chi(n) n^\nu }{(tn^2 +\Delta+|tn^2-\Delta|)^{s+\nu/2}}
  \frac{(2|tn^2-\Delta|)^{it_\psi}}
  {( tn^2 + \Delta+ |tn^2-\Delta|)^{it_\psi}},
  \end{equation}
  for $\Re(s)\geq 1/2+\vartheta+\varepsilon$.
  Here $\nu=\nu_\chi=0$ or $1$ such that $\chi(-1)=(-1)^\nu$.

  Furthermore, if $a\mid b$ and $(a,2)=1$ then we have
  \[
    \mathcal{S} = \frac{1}{4\pi i}
  \int_{(1)} D_{\psi,\chi_1,a^2}(s,\Delta)  \widetilde{W}(s) (8aY)^{s} \dd s + O\Big(\frac{P\Delta}{Y^{1/2-\vartheta}}\Big),
  \]
  where $\chi_1(n)=1$ for all $n\in\mathbb{Z}$ is the trivial character.
\end{lemma}

Recall that $W^{(j)}(x)\ll P^j$ for all $j\geq0$, so we have that $
\widetilde{W}(s) \ll (1+\frac{|s|}{P})^{-A}.$

Hence the contour integrals which appear in Lemma \ref{lemma:S=int(D)} are effectively restricted to $\Im(s) \ll P Y^\varepsilon$.
\begin{proof}

By completing the square we have
\begin{align}\label{eqn:S=square}
  \mathcal S & = \sum_{n\geq1} \lambda_\psi\Big(\frac{(2an+b)^2-\Delta}{4a}\Big) W\Big(\frac{(2an+b)^2-\Delta}{4aY}\Big) \nonumber \\
   & =  \sum_{\substack{n\geq1 \\ n\equiv b (2a)}} \lambda_\psi\Big(\frac{n^2-\Delta}{4a}\Big) W\Big(\frac{n^2-\Delta}{4aY}\Big),
\end{align}
where $\Delta=b^2-4ac>0$.
Here we use the fact that $\Delta\ll R^2 \ll Y^{1/2}$ and $\supp W \subset [1,2]$.
Then we have
\begin{align*}
  \mathcal S & =  \sum_{\substack{n\geq1 \\ n\equiv b'(a')}} \lambda_\psi\Big(\frac{ d^2n^2-\Delta}{4a}\Big) W\Big(\frac{d^2 n^2-\Delta}{4aY}\Big)
  = \frac{1}{\varphi(a')} \sum_{\chi\;(a')} \bar\chi(b') \mathcal{S}_\chi,
\end{align*}
where
\[
  \mathcal{S}_\chi
   := \sum_{n\geq1} \chi(n) \lambda_\psi\Big(\frac{ d^2n^2-\Delta}{4a}\Big) W\Big(\frac{d^2 n^2-\Delta}{4aY}\Big).
\]
Note that $\supp W \in [1,2]$. Hence the sum above is restricted to $n \ge 1$ with $d^2n^2-\Delta>0$ and we have that
\begin{align*}
  \mathcal S_\chi & =  \sum_{n\geq0} \frac{2^{\nu/2} d^\nu n^{\nu}}{(d^2 n^2+\Delta+ |d^2 n^2-\Delta|)^{\nu/2}} \chi(n) \frac{(2-\delta_{n,0})}{2} \lambda_\psi\Big(\frac{ d^2n^2-\Delta}{4a}\Big) \\
  & \hskip 5pt \cdot \frac{(2|d^2n^2-\Delta|)^{it_\psi}}
  {(d^2 n^2 + \Delta+ |d^2n^2-\Delta|)^{it_\psi}}
   W\Big(\frac{d^2 n^2+\Delta+ |d^2 n^2-\Delta|}{8aY}-\frac{\Delta}{4aY}\Big) \Big(1 + \frac
  {\Delta}{|d^2n^2-\Delta|} \Big)^{it_\psi}.
\end{align*}
Since $W'(x)\ll P$, we have that
\begin{align*}
  \mathcal{S}_\chi & = 2^{\nu/2-1} d^\nu \sum_{n\geq0}  \chi(n) n^\nu (2-\delta_{n,0}) \lambda_\psi\Big(\frac{ d^2n^2-\Delta}{4a}\Big) \frac{(2|d^2n^2-\Delta|)^{it_\psi}}
  {(d^2 n^2 + \Delta+ |d^2n^2-\Delta|)^{\nu/2+it_\psi}} \\
  & \hskip 90pt \cdot
   W\Big(\frac{d^2 n^2+\Delta+ |d^2 n^2-\Delta|}{8aY}\Big) \Big(1+O\Big(\frac{P\Delta}{aY}\Big)\Big).
\end{align*}
Hence we have that
\begin{equation}\label{eqn:S=S1}
  \mathcal{S}_\chi = \mathcal{S}_1 + O\Big(\frac{P\Delta}{Y^{1/2-\vartheta-\varepsilon}}\Big),
\end{equation}
where
\begin{align*}
  \mathcal{S}_1 & =  2^{\nu/2-1} d^\nu \sum_{n\geq0}  \chi(n) n^\nu (2-\delta_{n,0}) \lambda_\psi\Big(\frac{ d^2n^2-\Delta}{4a}\Big) \frac{(2|d^2n^2-\Delta|)^{it_\psi}}
  {(d^2 n^2 + \Delta+ |d^2n^2-\Delta|)^{\nu/2+it_\psi}} \\
  &  \hskip 90pt \cdot
   W\Big(\frac{d^2 n^2+\Delta+ |d^2 n^2-\Delta|}{8aY}\Big).
\end{align*}
By Mellin inversion, we have that
\begin{align*}
  \mathcal{S}_1 &  =  2^{\nu/2-1} d^\nu \sum_{n\geq0}  \chi(n) n^\nu (2-\delta_{n,0}) \lambda_\psi\Big(\frac{ d^2n^2-\Delta}{4a}\Big) \frac{(2|d^2n^2-\Delta|)^{it_\psi}}
  {(d^2 n^2 + \Delta+ |d^2n^2-\Delta|)^{\nu/2+it_\psi}}
   \\  & \hskip 120pt \cdot
  \frac{1}{2\pi i} \int_{(1)} \widetilde{W}(s) \Big(\frac{d^2 n^2+\Delta+ |d^2 n^2-\Delta|}{8aY}\Big)^{-s} \dd s \\
  & = 2^{\nu/2-1} d^\nu \frac{1}{2\pi i}
  \int_{(1)} D_{\psi,\chi,d^2}(s,\Delta)  \widetilde{W}(s) (8aY)^{s} \dd s,
\end{align*}
where $\widetilde{W}(s) = \int_{0}^{\infty} W(x) x^s \frac{\dd x}{x}$.
This proves the first claim.

\smallskip

Now we assume $(2,a)=1$ and $a\mid b$. By \eqref{eqn:S=square} we have that
\begin{align}\label{eqn:S=2}
  \mathcal S & =  \sum_{\substack{n\geq0 \\ n\equiv b(2)}} \lambda_\psi\Big(\frac{ a^2n^2-\Delta}{4a}\Big) W\Big(\frac{a^2 n^2-\Delta}{4aY}\Big) .
\end{align}
It follows that
\begin{align*}
  \mathcal S & =  \frac{1}{2} \sum_{\substack{n\geq0 \\ n\equiv b(2)}} \lambda_\psi\Big(\frac{ a^2n^2-\Delta}{4a}\Big) (2-\delta_{n,0}) \frac{(2|a^2n^2-\Delta|)^{it_\psi}}
  {(a^2 n^2 + \Delta+ |a^2n^2-\Delta|)^{it_\psi}} \\
  & \hskip 60pt \cdot
   W\Big(\frac{a^2 n^2+\Delta+ |a^2 n^2-\Delta|}{8aY}-\frac{\Delta}{4aY}\Big) \Big(1 + \frac
  {\Delta}{|a^2n^2-\Delta|} \Big)^{it_\psi} .
\end{align*}
Note that we define $\lambda_\psi(x)=0$ if $x\in\mathbb{R}\setminus\mathbb{Z}$, so the condition $n\equiv b(2)$ is redundant.  Repeating the argument above gives the second claim.
This completes the proof.
\end{proof}

\subsection{Bounding the Dirichlet series}

Let $\chi$ be a Dirichlet character mod $r$, and let $t\geq1$.
Let $\nu=\nu_{\chi}\in \{0,1\}$ be such that $\chi(-1)=(-1)^\nu$.
Recall from Section \ref{subsec:wt0} that $\Psi \in L^2_{\mathrm{cusp}}(\Gamma_{0}(4aN)\backslash \mathbb{H})$ and $\theta_{\chi,t}\in \textbf{H}_{\kappa}( 4r^2 t ,\chi_\nu)$ with $\kappa=\kappa_\chi=1/2+\nu$ and $\chi_\nu(n)=\chi(n) \left(\frac{-1}{n}\right)^\nu$ (see e.g. \cite[Theorem 10.10]{iwaniec1997topics}). Here $\theta_{\chi,t}(z)=y^{1/4+\nu/2}\sum_{n\in\mathbb{Z}} \chi(n) n^\nu e(n^2 z)$.
We have
$\overline{\Psi} \theta_{\chi,t}\in \textbf{H}_{\kappa} (M,\chi_\nu)$ where $M=\lcm[4aN,4r^2t]$.
We first prove a relation between $D_{\psi,\chi,t}(s,\Delta)$ and $\langle P_{\Delta,\kappa,\chi_\nu}(\cdot;s+1/4), \overline{\Psi} \theta_{\chi,t} \rangle$.
We will prove the following lemma.

\begin{lemma}\label{lemma:D2inprod}
Let $s \in \mathbb C$ with $\tmop{Re}(s)>3/4$.
  We have that
  \begin{multline*}
    D_{\psi,\chi,t}(s,\Delta) \\ = \varrho_\psi^{-1} (16\pi a)^{-1/2} \frac{(2\pi)^{s+\nu/2} \Gamma(s+\nu/2+1/2)}{\Gamma(s+\nu/2+it_\psi)\Gamma(s+\nu/2-it_\psi)}
    \langle P_{\Delta,\kappa,\chi_\nu}(\cdot;s+1/4), \overline{\Psi} \theta_{\chi,t} \rangle
    + R(s),
  \end{multline*}
where $R(s)$ is defined as in \eqref{eqn:R(s)}. Furthermore, $R(s)$
  is holomorphic if $\Re(s)>-1/2+\vartheta+\varepsilon$ in which case we have
  \[
    R(s) = O\Big( t^{-1/2-\nu/2} a^{-\vartheta} \Delta^{1/2-\tmop{Re}(s)+\vartheta} \Big).
  \]
\end{lemma}

\begin{proof}
Consider
\[
  I = \langle P_{\Delta,\kappa,\chi_\nu}(\cdot;s+1/4), \overline{\Psi} \theta_{\chi,t} \rangle
  = \int_{\Gamma_0(M)\backslash \mathbb{H}} P_{\Delta,\kappa,\chi_\nu}(z;s+1/4) \Psi(z) \overline{\theta_{\chi,t}(z)} \frac{\dd x\dd y}{y^2}.
\]
Unfolding the integral according to the definition of $P_{\Delta,\kappa,\chi_\nu}$ yields
\begin{align}
  I & = \int_{0}^{\infty} \int_{0}^{1} \Psi(z) \overline{\theta_{\chi,t}(z)} e^{-2\pi \Delta y} y^{s+1/4} e(\Delta x) \frac{\dd x\dd y}{y^2} \nonumber \\
  & = \int_{0}^{\infty} \int_{0}^{1}
   \varrho_\psi \sum_{m\in\mathbb{Z}} \frac{\lambda_\psi(m)}{\sqrt{|m|}} W_{0,it_\psi}(16\pi a|m|y) e(4am x) \nonumber \\
  & \hskip 60pt \cdot y^{1/4+\nu/2} \sum_{n\geq 0} (2-\delta_{n,0}) \chi(n) n^\nu e(-  t n^2 x) e^{-2\pi t n^2 y}  e^{-2\pi \Delta y} y^{s+1/4} e(\Delta x) \frac{\dd x\dd y}{y^2}  \nonumber \\
  & =    \varrho_\psi \mathop{\sum_{m\in\mathbb{Z}} \sum_{n\geq0}}_{4am- t n^2+\Delta=0} \frac{\lambda_\psi(m)}{\sqrt{|m|}} (2-\delta_{n,0}) \chi(n) n^\nu B(m,n), \label{eqn:I=sum}
\end{align}
where
\[
  B := B(m,n) = \int_{0}^{\infty} W_{0,it_\psi}(16\pi a|m|y)   e^{-(2\pi t n^2 +2\pi \Delta) y} y^{s-3/2+\nu/2}  \dd y.
\]
Note that
$W_{0,it_\psi}(16\pi a|m|y)= \sqrt{16 a|m| y} K_{it_\psi}(8\pi a|m|y)$ (cf. \cite[eq. (9.235.2)]{GR}).
By \cite[eq. (6.621.3)]{GR}, we have
\begin{multline}\label{eqn:K&exp}
  \int_{0}^{\infty} y^{u-1} e^{-\alpha y}K_v(\beta y) \dd y \\
  = \frac{\sqrt{\pi} (2\beta)^v}{(\alpha+\beta)^{u+v}} \frac{\Gamma(u+v)\Gamma(u-v)}{\Gamma(u+1/2)} F\left(u+v,v+1/2;u+1/2;\frac{\alpha-\beta}{\alpha+\beta}\right),
\end{multline}
if $\Re(u)>|\Re(v)|$ and $\Re(\alpha+\beta)>0$. Here $F(\alpha,\beta;\gamma;z)$ is the hypergeometric series (cf. \cite[\S 9.1]{GR}).
Hence for $\Re(s)>\vartheta$, we have
\begin{multline}\label{eqn:B=}
  B  = (16 a|m|)^{1/2} \int_{0}^{\infty} K_{it_\psi}(8\pi a|m|y)   e^{-(2\pi t n^2 +2\pi \Delta) y} y^{s-1+\nu/2}  \dd y \\
   = (16 a|m|)^{1/2}
  \frac{\sqrt{\pi} (16\pi a|m|)^{it_\psi}}
  {(2\pi t n^2 +2\pi \Delta+8\pi a|m|)^{s+\nu/2+it_\psi}} \frac{\Gamma(s+\nu/2+it_\psi)\Gamma(s+\nu/2-it_\psi)}{\Gamma(s+\nu/2+1/2)}  \\
   \cdot F\left(s+\nu/2+it_\psi,1/2+it_\psi;s+\nu/2+1/2;\frac{t n^2 +\Delta-4a|m|}{t n^2 +\Delta+4a|m|}\right),
\end{multline}
where $4am=t n^2-\Delta$ and $n\in \mathbb{Z}_{\geq0}$, $m\in\mathbb{Z}$.
Expanding $F$ in a Taylor series about zero in its last variable gives
\begin{equation}\label{eqn:F-1}
  F\left(s+\nu/2+it_\psi,1/2+it_\psi;s+\nu/2+1/2;
  \frac{t n^2 +\Delta-4a|m|}{t n^2 +\Delta+4a|m|}\right)
  = 1 + O\left(\frac{\Delta}{t n^2+\Delta}\right),
\end{equation}
uniformly for $s \in \mathbb C$ with $\tmop{Re}(s)>\vartheta$.
Let
\begin{multline}\label{eqn:R(s)}
  R(s) = \mathop{\sum_{n\geq1}}_{t n^2\equiv \Delta (4a)} \lambda_\psi\Big(\frac{t n^2-\Delta}{4a}\Big) (2-\delta_{n,0}) \chi(n) n^\nu
  \frac{(2|tn^2-\Delta|)^{it_\psi}}
  {(tn^2 +\Delta+|tn^2-\Delta|)^{s+\nu/2+it_\psi}}
  \\
  \cdot \left( 1 - F\left(s+\nu/2+it_\psi,1/2+it_\psi;s+\nu/2+1/2;
  \frac{t n^2 +\Delta-|t n^2-\Delta|}{t n^2 +\Delta+|t n^2-\Delta|}\right) \right).
\end{multline}
Thus by \eqref{eqn:I=sum} and \eqref{eqn:B=}, we have
\begin{align*}
  I 
  & =    \varrho_\psi (16\pi a)^{1/2} \frac{\Gamma(s+\nu/2+it_\psi)\Gamma(s+\nu/2-it_\psi)}{\Gamma(s+\nu/2+1/2)}
  \\
  & \hskip 10pt \cdot \mathop{\sum_{n\geq0}}_{t n^2\equiv \Delta (4a)} \lambda_\psi\Big(\frac{t n^2-\Delta}{4a}\Big) (2-\delta_{n,0}) \chi(n) n^\nu
  \frac{(4\pi|tn^2-\Delta|)^{it_\psi}}
  {(2\pi tn^2 +2\pi \Delta+2\pi |tn^2-\Delta|)^{s+\nu/2+it_\psi}}
  \\
  & \hskip 10pt \cdot F\left(s+\nu/2+it_\psi,1/2+it_\psi;s+\nu/2+1/2;
  \frac{t n^2 +\Delta-|t n^2-\Delta|}{t n^2 +\Delta+|t n^2-\Delta|}\right)
  .
\end{align*}
Note that $F(\alpha,\beta;\gamma;0)=1$.  By \eqref{eqn:Dchi} and \eqref{eqn:R(s)}, we prove the first claim.

Now we deal with $R(s)$. We first assume $t\ll \Delta$.
By \eqref{eqn:F-1} and \eqref{eqn:R(s)} we have
\begin{multline*}
  R(s) \ll \sum_{1\leq n\leq \sqrt{\Delta/t}} \left(\frac{\Delta}{a}\right)^{\vartheta}  n^\nu
  \frac{1}{\Delta^{\Re{s}+\nu/2}}
  + \sum_{ n\geq \sqrt{\Delta/t}} \left(\frac{tn^2}{a}\right)^{\vartheta}  n^\nu
  \frac{1}{(tn^2)^{\Re{s}+\nu/2}}  \frac{\Delta}{t n^2}
  \\
  \ll \left(\frac{\Delta}{t}\right)^{1/2+\nu/2} \left(\frac{\Delta}{a}\right)^{\vartheta}
  \frac{1}{\Delta^{\Re{s}+\nu/2}}
  +  \left(\frac{t}{a}\right)^{\vartheta}
  \frac{1}{t^{\Re{s}+\nu/2}}  \frac{\Delta}{t} \sum_{ n\geq \sqrt{\Delta/t}} \frac{1}{n^{2+2\Re(s)-2\vartheta}} \\
   \ll t^{-1/2-\nu/2} a^{-\vartheta} \Delta^{1/2+\vartheta-\Re(s)},
\end{multline*}
if $\Re(s)\geq -1/2+\vartheta+\varepsilon$. If $\Delta\ll t$, then we have
\begin{multline*}
  R(s) \ll \sum_{ n\geq 1} \left( \frac{tn^2}{a} \right)^{\vartheta}  n^\nu
  \frac{1}{(tn^2)^{\Re{s}+\nu/2}}  \frac{\Delta}{t n^2}
  \ll \left(\frac{t}{a}\right)^{\vartheta}
  \frac{1}{t^{\Re{s}+\nu/2}}  \frac{\Delta}{t} \sum_{ n\geq 1} \frac{1}{n^{2+2\Re(s)-2\vartheta}} \\
   \ll t^{-1-\nu/2+\vartheta-\Re(s)} a^{-\vartheta} \Delta
   \ll t^{-1/2-\nu/2} a^{-\vartheta} \Delta^{1/2+\vartheta-\Re(s)},
\end{multline*}
if $\Re(s)\geq -1/2+\vartheta+\varepsilon$. This completes the proof.
\end{proof}

Now we are ready establish analytic estimates for $D_{\psi,\chi,t}(s,\Delta)$ by using the spectral decomposition of $P_{\Delta,\kappa,\chi_\nu}(\cdot;s+1/4)$ in $\textbf{H}_{\kappa}(M,\chi_\nu)$.
We will prove the following result.

\begin{lemma}\label{lemma:D(s)==}
  The function $(s-\tfrac12)D_{\psi,\chi,t}(s,\Delta)$ has an analytic continuation to the half-plane $\Re(s) > \tfrac14+\tfrac{\vartheta}{2}$. Moreover, in this region, we have that
  \begin{multline*}
  D_{\psi,\chi,t}(s,\Delta) =
  \frac{(2\pi)^{s+\nu/2} (4\pi \Delta)^{3/4-s} \Gamma(s+\nu/2+1/2)\Gamma(s-1/2)}{\varrho_\psi (16\pi a)^{1/2} \Gamma(s+\nu/2+it_\psi)\Gamma(s+\nu/2-it_\psi)}  \\
  \cdot
  \sum_{f_j \in \mathbf{H}_{\kappa}(M,\chi_\nu, i/4)}
      \overline{\rho_j(\Delta)}
     \langle f_j , \overline{\Psi} \theta_{\chi,t} \rangle
  + O_\psi\Big( M^A \Delta^{B} (1+|\Im(s)|)^C \Big).
  \end{multline*}
  Hence
  \begin{multline*}
    \underset{s=1/2}{\Res} \; D_{\psi,\chi,t}(s,\Delta) \\
    = \frac{2^{\nu/2-1} \pi^{\nu/2+1/4} \Delta^{1/4} \Gamma(1+\nu/2)}{\varrho_\psi a^{1/2} \Gamma(1/2+\nu/2+it_\psi)\Gamma(1/2+\nu/2-it_\psi)}
  \sum_{f_j \in \mathbf{H}_{\kappa}(M,\chi_\nu, i/4)}
      \overline{\rho_j(\Delta)}
     \langle f_j , \overline{\Psi} \theta_{\chi,t} \rangle,
  \end{multline*}
  and uniformly for $1/4+\vartheta/2+\varepsilon \le \Re(s) \le 1/2 $ we have $D_{\psi,\chi,t}(s,\Delta) = O\left( M^A \Delta^{B} (1+|\Im(s)|)^C \right)$.
\end{lemma}

\begin{proof}
  Spectrally expanding the inner product via Parseval's formula, we get
  \[
    \langle P_{\Delta,\kappa,\chi_\nu}(\cdot;s+1/4), \overline{\Psi} \theta_{\chi,t} \rangle
    = \sum_{j} \langle P_{\Delta,\kappa,\chi_\nu}(\cdot;s+1/4), f_j \rangle \langle f_j , \overline{\Psi} \theta_{\chi,t} \rangle + \textrm{cont,}
  \]
  where $f_j = f_{j,\kappa}$ is as in Section \ref{sec:autom_half} with $\kappa=1/2+\nu$, and ``cont'' denotes the contribution from the continuous spectrum which can be bounded in the same way as for the Maass forms. In fact, this is easier since we can use unfolding to deal with the triple product.

  By unfolding and applying \eqref{eqn:f}, we get that
  \begin{align*}
   \langle P_{\Delta,\kappa,\chi_\nu}& (\cdot;s+1/4),  f_j \rangle
     = \int_{0}^{\infty} \int_{0}^{1} \overline{f_j(z)} e^{-2\pi \Delta y} y^{s+1/4} e(\Delta x) \frac{\dd x\dd y}{y^2} \\
  & = \int_{0}^{\infty} \int_{0}^{1}
   \Big( \overline{\rho_{j}(0, y)} + \sum_{n \not=0} \overline{\rho_{j}(n)} W_{\sgn(n)\frac{1}{4}, \overline{it_j}}(4\pi |n|y)  e(-n x) \Big) e^{-2\pi \Delta y} y^{s+1/4} e(\Delta x) \frac{\dd x\dd y}{y^2} \\
  & =    \overline{\rho_j(\Delta)}
  \int_{0}^{\infty} W_{\frac{1}{4},\overline{it_j}}(4\pi \Delta y)   e^{-2\pi \Delta y} y^{s-3/4}  \frac{\dd y}{y} \\
  & =  \overline{\rho_j(\Delta)}  (4\pi \Delta)^{3/4-s} \frac{\Gamma(s-1/4-it_j)\Gamma(s-1/4+it_j)}{\Gamma(s)},
  \end{align*}
  if $\Re(s)>1/4+\vartheta/2$. Here we have used \cite[eq. (7.621.11)]{GR}. We have also used the fact that $t_j$ is real or purely imaginary so that $\{it_j,-it_j\}=\{\overline{it_j},-\overline{it_j}\}$.
  If $t_j\neq i/4$, then $|\Im(t_j)|\leq \vartheta/2$, and hence $\langle P_{\Delta,\kappa,\chi_\nu}(\cdot;s+1/4),  f_j \rangle$ is holomorphic in $\Re(s) > 1/4+\vartheta/2$. By Lemma \ref{lemma:triple}, we have that
  $\langle f_j , \overline{\Psi} \theta_{\chi,t} \rangle \ll M^A (1+|t_j|)^C e^{-\frac{\pi}{2}|t_j|}$. By Stirling's formula, we have that
  \[
    \sum_{f_j \notin \textbf{H}_{\kappa}(M,\chi_\nu, i/4)} \langle P_{\Delta,\kappa,\chi_\nu}(\cdot;s+1/4), f_j \rangle \langle f_j , \overline{\Psi} \theta_{\chi,t} \rangle
    \ll M^A (1+|\Im(s)|)^{C} \Delta^{B} e^{-\frac{\pi}{2}|\Im(s)|}.
  \]
  If $f_j\in \textbf{H}_{\kappa}(M,\chi_\nu, i/4)$, then $\langle P_{\Delta,\kappa,\chi_\nu}(\cdot;s+1/4),  f_j \rangle$ has a simple pole at $s=1/2$ and $\langle f_j , \overline{\Psi} \theta_{\chi,t} \rangle$ only depends on $\psi,a,b,c$. The contribution to $\langle P_{\Delta,\kappa,\chi_\nu}(\cdot;s+1/4), \overline{\Psi} \theta_{\chi_1,a^2} \rangle$ is
  \[
     \sum_{f_j \in \textbf{H}_{\kappa}(M,\chi_\nu, i/4)}
      \overline{\rho_j(\Delta)}  (4\pi \Delta)^{3/4-s} \Gamma(s-1/2)
     \langle f_j , \overline{\Psi} \theta_{\chi,t} \rangle.
  \]
  By Lemma \ref{lemma:D2inprod}, we have that
  \begin{multline*}
  D_{\psi,\chi,t}(s,\Delta) =
  \frac{(2\pi)^s (4\pi \Delta)^{3/4-s} \Gamma(s+\nu/2+1/2)\Gamma(s-1/2)}{\varrho_\psi (16\pi a)^{1/2} \Gamma(s+\nu/2+it_\psi)\Gamma(s+\nu/2-it_\psi)}
  \sum_{f_j \in \textbf{H}_{\kappa}(M,\chi_\nu, i/4)}
      \overline{\rho_j(\Delta)}
     \langle f_j , \overline{\Psi} \theta_{\chi,t} \rangle  \\
  + O\Big(\Big| \frac{(2\pi)^s \Gamma(s+\nu/2+1/2)}{\Gamma(s+\nu/2+it_\psi)\Gamma(s+\nu/2-it_\psi)}\Big| M^A \Delta^{B} (1+|\Im(s)|)^C e^{-\frac{\pi}{2}|\Im(s)|} \Big).
  \end{multline*}
  This completes the proof.
\end{proof}

\subsection{Proof of Theorem \ref{prop:sum_quad} (i)} \label{subsec:Proof1}
Apply Lemma \ref{lemma:S=int(D)} and shift the contour of integration to $\Re(s)=1/4+\vartheta/2+\varepsilon$, which is justified by \ref{lemma:D(s)==} using that $\widetilde{W}(s) \ll (1+\frac{|s|}{P})^{-A}$.  We potentially pick up a pole at $s=1/2$ and conclude that
\begin{align*}
  \mathcal{S}
  & = \frac{1}{2}
   \underset{s=1/2}{\Res}\; D_\psi(s,\Delta) \; \widetilde{W}(1/2) (8aY)^{1/2} \\
  & \hskip 30pt + O\left(\Big| \int_{(1/4+\vartheta/2+\varepsilon)} D_\psi(s,\Delta)  \widetilde{W}(s) (8aY)^{s} \dd s \Big|\right)
   + O\Big(\frac{P\Delta}{Y^{1/2-\vartheta}}\Big) \\
  & = C_{\psi,a,b,c}  \widetilde{W}(1/2) Y^{1/2}
  + O\left( M^A \Delta^B P^C Y^{1/4+\vartheta/2+\varepsilon} \right),
\end{align*}
where
\begin{multline}\label{eqn:C}
  C_{\psi,a,b,c} = \frac{(2a)^{1/2}}{\varphi(a')} \sum_{\substack{\chi \; (a') }} \bar\chi(b') 2^{\nu_\chi/2} d^{\nu_\chi} \\
  \cdot
  \frac{2^{\nu_\chi/2-1} \pi^{\nu_\chi/2+1/4} \Delta^{1/4} \Gamma(1+\nu_\chi/2)}{\varrho_\psi a^{1/2} \Gamma(1/2+\nu_\chi/2+it_\psi)\Gamma(1/2+\nu_\chi/2-it_\psi)}
  \sum_{f_j \in \mathbf{H}_{\kappa}(M,\chi_\nu, i/4)}
      \overline{\rho_j(\Delta)}
     \langle f_j , \overline{\Psi} \theta_{\chi,d^2} \rangle \\
  = \frac{\pi^{1/4} \Delta^{1/4}}{2^{1/2}\varrho_\psi \varphi(a')}
  \sum_{\substack{\chi \; (a') }}
  \frac{ \pi^{\nu_\chi} d^{\nu_\chi} \bar\chi(b')}{ \Gamma(1/2+\nu_\chi/2+it_\psi)\Gamma(1/2+\nu_\chi/2-it_\psi)}
  \sum_{f_j \in \mathbf{H}_{\kappa}(M,\chi_\nu, i/4)}
      \overline{\rho_j(\Delta)}
     \langle f_j , \overline{\Psi} \theta_{\chi,d^2} \rangle.
\end{multline}
Here we used the fact that $2^{\nu_\chi}\pi^{-\nu_\chi/2} \Gamma(1+\nu_\chi/2)=1$. Recall that $d=\gcd(2a,b)$, $a'=2a/d$, $b'=b/d$. This proves the first claim.

\subsection{Proof of Theorem \ref{prop:sum_quad} (ii)} \label{subsec:Proof2}

In the case $a\mid N$ and $a\mid b$, we can show that the main term vanishes unconditionally.
We only need to prove $\langle f_j , \overline{\Psi} \theta_{\chi_1,a^2} \rangle = 0 $ for all $f_j \in \textbf{H}_{1/2}(M, i/4)$ where $M=4aN$.
By \cite{Serre-Stark}, the space $\textbf{H}_{1/2}(M, i/4)$ is spanned by the theta series
\begin{equation}\label{eqn:FE_theta}
  \theta_{\omega,t}(z)
  = y^{1/4} \sum_{n\in\mathbb{Z}} \omega(n) e(tn^2 z)
  =  \sum_{n\geq0} (2-\delta_{n,0})  \omega(n) y^{1/4} e(tn^2 z) ,
\end{equation}
where $\omega$ is an even primitive Dirichlet character of conductor $r$ which induces $\chi_{t}$, and $r^2 t \mid a N$. Here $\chi_t$ is the primitive character associated to the field extension $\mathbb{Q}(\sqrt{t})/\mathbb{Q}$.
%
By Lemma \ref{lemma:res}, we have
\begin{equation}\label{eqn:tripe=Res}
  \langle \overline{\Psi} \theta_{\chi_1,a^2}, \theta_{\omega,t}\rangle
   = c \Res_{s=3/4} \langle \overline{\Psi} E_{\infty,1/2}(\cdot;s) , \theta_{\omega,t} \rangle,
\end{equation}
where $c$ is some constant depending on $M$. By unfolding, we have
\begin{align*}
  \langle \overline{\Psi} E_{\infty,1/2}(\cdot;s) , \theta_{\omega,t} \rangle
  & = \int_{\Gamma_0(M) \backslash \mathbb{H}} \overline{\Psi(z)} \overline{\theta_{\omega,t}(z)}   \sum_{\gamma \in \Gamma_{\infty}\backslash \Gamma_0(M)}  J(\gamma, z)^{-1} \Im(\gamma z)^{s} \frac{\dd x\dd y}{y^2} \\
  & = \int_{\Gamma_\infty \backslash \mathbb{H}}
  \overline{\Psi(z)} \overline{\theta_{\omega,t}(z)}  y^{s} \frac{\dd x\dd y}{y^2}.
\end{align*}
By the Fourier expansions \eqref{eqn:FE_Psi} and \eqref{eqn:FE_theta}, we have
\begin{align*}
  \langle \overline{\Psi} E_{\infty,1/2}(\cdot;s) , \theta_{\omega,t} \rangle
  & = 2 \sum_{\substack{n\geq1 \\ 4a\mid t n^2 }} \frac{\overline{\lambda_\psi(-tn^2/4a)}\bar{\omega}(n) } {(tn^2/4a)^{1/2}} \int_{0}^{\infty} W_{0, \overline{it_\psi}}(4\pi tn^2 y)
  e^{-2\pi tn^2 y}  y^{s - 3/4} \frac{\dd y}{y} \\
  & = 2 (4\pi t)^{3/4-s} (4a)^{1/2} t^{-1/2}  \sum_{\substack{n\geq1 \\ 4a\mid t n^2 }} \frac{\overline{\lambda_\psi(-tn^2/4a)}\bar\omega(n) } {n^{2s-1/2}} \mathscr{I}(s),
\end{align*}
where
$$
  \mathscr{I}(s):= \int_{0}^{\infty} W_{0, \overline{it_\psi}}(y)
  e^{-y/2}  y^{s - 3/4} \frac{\dd y}{y}.
$$
By \cite[eq. (7.621.11)]{GR}, we have
\[
  \mathscr{I}(s) = \frac{\Gamma(s-1/4+\overline{it_\psi}) \Gamma(s-1/4-\overline{it_\psi})}{\Gamma(s+1/4)},
\]
if $\Re(s-1/4\pm \overline{it_\psi})>0$. In particular, $\mathscr{I}(s)$ is holomorphic at $s=3/4$.
Now we consider the Dirichlet series.
Let $d=(4a,t)$ and $a'=4a/d$, $t'=t/d$. Write $a' =a_1 a_2^2$ with $a_1$ square-free. Then we have $a_1a_2 \mid n$. Write $n'=n/a_1a_2$, then we have
\begin{align*}
  \sum_{\substack{n\geq1 \\ 4a\mid t n^2 }} \frac{\overline{\lambda_\psi(-tn^2/4a)}\bar\omega(n) } {n^{2s-1/2}} & =
  \frac{\lambda_\psi(-1) \bar\omega(a_1a_2)}{(a_1a_2)^{2 s-1/2}} \sum_{\substack{n'\geq1  }} \frac{\overline{\lambda_\psi(t'a_1 (n')^2)}\bar\omega(n') } {(n')^{2s-1/2}} \\
   &  =  \frac{\lambda_\psi(-1) \bar\omega(a_1a_2)}{(a_1a_2)^{2 s-1/2}}
  \prod_{p}  \sum_{\ell=0}^{\infty} \bar\omega(p^\ell) \frac{\overline{\lambda_\psi(p^{2\ell+r_{p}})}}{p^{\ell(2 s-1/2)}},
\end{align*}
where $r_p\geq0$ such that $p^{r_p} \parallel t' a_1 $. For $p\nmid t' a_1$, the local Euler factor is the same as those for $\Sym^2 \psi \times \bar\omega$. Hence we have that
\begin{align*}
  \sum_{\substack{n\geq1 \\ 4a\mid t n^2 }} \frac{\overline{\lambda_\psi(-tn^2/4a)}\bar\omega(n) } {n^{2s-1/2}}
   &  =  \frac{\lambda_\psi(-1) \bar\omega(a_1a_2)}{(a_1a_2)^{2 s-1/2}}
   L_{t'a_1}(2s-1/2,\Sym^2 \psi \times \bar\omega) H(2s-1/2),
\end{align*}
where
\[
  H(s) = \prod_{p \mid t' a_1} \sum_{\ell=0}^{\infty} \bar\omega(p^\ell) \frac{\overline{\lambda_\psi(p^{2\ell+r_{p}})}}{p^{\ell(2 s-1/2)}},
\]
which is absolutely convergent for $\Re(s)>1/2$.
Since $\psi$ is not dihedral, the symmetric square lift  $\Sym^2\psi$ is a $\GL_3$ cuspidal automorphic form. Hence  $L(s,\Sym^2 \psi \times \omega)$ is an entire function. Thus we know that
\[
  \Res_{s=3/4} \langle \overline{\Psi} E_{\infty,1/2}(\cdot;s) , \theta_{\omega,t} \rangle = 0.
\]
Together with \eqref{eqn:tripe=Res}, we complete the proof of the second claim.

\subsection{Proof of Theorem \ref{prop:sum_quad} (iii)} \label{subsec:Proof3}

To prove Theorem \ref{prop:sum_quad} (iii), we first show the following lemma.

\begin{lemma} \label{lem:nair-tenenbaum}
Let $\psi$ be a Hecke-Maass cusp form of level $N$ with trivial nebentypus. Assume GRC. Then for any integers $a,b,c$ such that $Q(X)=aX^2+bX+c$ is irreducible over $\mathbb Q[X]$ we have that
\[
\frac{1}{x}\sum_{n \le x} |\lambda_{\psi}(Q(n))| \ll \frac{1}{(\log x)^{1/18}},
\]
where the implied constant depends on $\psi,a,b,c$.
\end{lemma}

\begin{remark}
  If $Q(X)$ is reducible over $\mathbb Q[X]$ then our argument can be modified to show that assuming GRC
  \begin{equation} \label{eq:splitbd}
  \frac{1}{x}\sum_{n \le x} |\lambda_{\psi}(Q(n))| \ll \frac{1}{(\log x)^{\delta}}
  \end{equation}
for some $\delta>0$.
  However, in this case it is possible that the methods of Holowinsky \cite{holowinsky2009} could be used to unconditionally establish such an estimate.
\end{remark}

\begin{proof}
By assumption $|\lambda_{\psi}(n)| \le \tau(n)$. Hence, it follows from Henriot
 \cite[Theorem 4]{henriot} that
\begin{equation} \label{eq:nair-tenenbaum}
\sum_{n \le x} |\lambda_{\psi}(Q(n))|  \ll x \prod_{p \le x}\left(1+\frac{\varrho_Q(p)(|\lambda_f(p)|-1)}{p} \right),
\end{equation}
where for $n \in \mathbb N$, $\varrho_Q(n)=\#\{ a \pmod n : Q(a) \equiv 0 \pmod n\}$ and the implied constant depends on $Q$. To estimate the Euler product above we use the inequality
$|y| \le \frac{1}{18}(8+11y^2-y^4)$, which holds for $|y|\le 2$\footnote{This is motivated by work of Elliott--Moreno--Shahidi \cite{ems}, see also the work of Holowinsky \cite[Eq'ns (60)-(63)]{holowinsky2009}. In particular, we have not optimized our argument.}. Hence,
\begin{equation}\label{eq:majorant}
\sum_{p \le x} \frac{\varrho_Q(p)|\lambda_f(p)|}{p}
\le \frac{1}{18}\left( 8\sum_{p \le x} \frac{\varrho_Q(p)}{p}+11\sum_{p \le x} \frac{\varrho_Q(p)\lambda_f(p)^2}{p}-\sum_{p \le x} \frac{\varrho_Q(p)\lambda_f(p)^4}{p} \right).
\end{equation}
Write $\Delta=b^2-4ac$.
For $(a,p)=1$ it follows that
\begin{equation} \label{eq:quadratic}
\varrho_Q(p)=1+\chi_{\Delta}(p)
\end{equation}
where $\chi_{\Delta}(p)$ is the Legendre symbol.
Each of the functions $L(s,\tmop{sym}^2 \psi \otimes \chi_{\Delta})$ and $L(s,\tmop{sym}^4 \psi \otimes \chi_{\Delta})$ are analytic and non-zero in the region  $\{ s \in \mathbb C :\tmop{Re}(s)>1-c/\log(|\tmop{Im}(s)|+2)\}$ for sufficiently small $c>0$, which depends on $\Delta$ and $\psi$. This follows from the arguments given in Blomer et. al. \cite[Section 2.3.4-2.4]{BFKMMS}, and in particular uses a result of Kim-Shahidi \cite[Theorem 3.3.7]{kim-shahidi}. Also, by the Hecke relations
\begin{equation} \notag
11\lambda_f(p)^2-\lambda_f(p)^4=9+8\lambda_f(p^2)-\lambda_f(p^4).
\end{equation}
Since $Q$ is irreducible we have $\Delta \neq \square$, so $\chi_{\Delta}$ is non-principal.
Thus, using \eqref{eq:quadratic} along with the previous observations gives
\[
8\sum_{p \le x} \frac{\varrho_Q(p)}{p}+11\sum_{p \le x} \frac{\varrho_Q(p)\lambda_f(p)^2}{p}-\sum_{p \le x} \frac{\varrho_Q(p)\lambda_f(p)^4}{p}=17 \log \log x+O(1),
\]
where the implied constant depends on $a,b,c,\psi$.  Applying this along with \eqref{eq:majorant} and \eqref{eq:quadratic} in \eqref{eq:nair-tenenbaum} completes the proof.
\end{proof}

\begin{proof}[Proof of Theorem \ref{prop:sum_quad} (iii)]
 First we consider the case where $Q$ is irreducible over $\mathbb Q[X]$. Combining
 Theorem \ref{prop:sum_quad} (i) and Lemma \ref{lem:nair-tenenbaum} we get  for each $a,b,c \in \mathbb Z$ such that $b^2-4ac>0$ is not a perfect square that
\[
C_{\psi,a,b,c}=\lim_{x \rightarrow \infty} \frac{1}{x} \sum_{n \le x} \lambda_{\psi}(Q(n))=0.
\]
The case that $Q$ is reducible is similar, only in place of Lemma \ref{lem:nair-tenenbaum} we use \eqref{eq:splitbd}.
\end{proof}

\section{The twisted first moment}\label{sec:twisted1moment}

Let $D=p_1p_2$ where $p_1 \equiv p_2 \equiv 3 \pmod 4$ and $p_1 \neq p_2$. Also,
 let $\phi$ be a Schwartz function with compact support in $[\tfrac12, 2]$ such that $\phi^{(j)}(x) \ll P^j$, where $P \ge 1$ is a large parameter. Note that
 \begin{equation} \label{eq:decay}
 \tilde \phi(s):= \int_0^{\infty} \phi(y) y^{s-1} \, \dd y \ll \frac{P^A}{1+|s|^A}
 \end{equation}
 for any $A \ge 1$, by repeated integration by parts.
 Throughout this section $\psi$ denotes an even Hecke--Maass cuspidal newform on $\Gamma_0(D)$ with trivial nebentypus. Additionally, for $k \in \mathbb N$ let $\phi_{2k}$ be as in \eqref{eq:dihedraldef}. Also, let $\eta_{\psi}(D)$ denote the $W_D$-eigenvalue of $\psi$, where $W_D$
is the Atkin-Lehner operator. By Propositions A.1 and A.2 of \cite{KMV}, we have $\eta_{\psi}(D) \in \{-1,+1\}$. If $\eta_{\psi}(D)=-1$ then $L(\tfrac12, \psi \times \phi_{2k})=0$, which follows from the functional equation i.e. \eqref{eq:functional-eqn}.


The first main result of this section is the following estimate for a first moment.

\begin{proposition} \label{prop:moment1}
Suppose $\eta_{\psi}(D)=1$. Then there exists $A_0>0$ such that
\[
\sum_{k \in \mathbb Z} L(\tfrac12, \psi \times \phi_{2k})  \, \phi\left( \frac{k}{K} \right)= \tilde{\phi}(1) \cdot C_{D,\psi} \cdot K +O( P^{A_0} K^{\frac12+\vartheta+\varepsilon}),
\]
where the implied constant depends at most on $ \psi,D$. Here
\begin{equation} \label{eq:constdef}
C_{D,\psi}= 2  \cdot \frac{L(1, \chi_{D}) }{\zeta_D(2)}L(1, \tmop{sym}^2 \psi)\left(1+ \frac{\lambda_\psi(p_1)}{\sqrt{p_1}}+\frac{\lambda_\psi(p_2)}{\sqrt{p_2}} +\frac{\lambda_\psi(D)}{\sqrt{D}} \right)
\end{equation}
where 
$\zeta_D(s)= \zeta(s) \prod_{p|D}(1-p^{-s}) $.
\end{proposition}

Assuming GRC holds for $\psi$ we are also able to compute a twisted first moment. Given $(\beta) \subset \mathcal O_{\mathbb Q(\sqrt{D})}$ with $\beta=M+N\omega_D$ let $n_{\beta}= \frac{1}{(M,N)^2} |N(\beta)|$. The number $n_{\beta}$ is independent of the choice of generator of $(\beta)$, which we will justify later on (see Remark \ref{rem:1}). Also, let $h(\cdot)$ be the multiplicative function with
\begin{equation} \label{eq:hdef}
h(n)=\sum_{N((\beta))=n} \frac{\vartheta(n_{\beta})}{\sqrt{n_{\beta}}}
\end{equation}
where $\vartheta(\cdot)$ is the multiplicative function defined as follows
\begin{equation}\label{eqn:theta}
  \vartheta(p^b) = \left( \lambda_\psi(p^b) - \frac{\chi_D^0(p) \lambda_\psi(p^{b-2})}{p} \right) \left(1+\frac{\chi_D^0(p)}{p}\right)^{-1},
\end{equation}
for $p$ prime and $b\geq1$. Here $\lambda_\psi(p^{-1})=0$ and $\chi_D^0$ is the principal character modulo $D$.

\begin{proposition} \label{prop:moment}
Assume GRC. Suppose $\eta_{\psi}(D)=1$.
Then there exists $A_0>0$ such that for $n \in \mathbb N$
\[
\sum_{k \in \mathbb Z} L(\tfrac12, \psi \times \phi_{2k}) \cdot \lambda_{2k}(n) \, \phi\left( \frac{k}{K} \right)= \tilde{\phi}(1) \cdot C_{D,\psi} \cdot h\left( \frac{n}{(n,D)} \right) \cdot K +O((Pn)^{A_0} \cdot K^{\frac12+\vartheta+\varepsilon}),
\]
where $C_{D,\psi}$ is as in Proposition \ref{prop:moment1}.
\end{proposition}

Using the twisted first moment we can quickly deduce the following corollary.



\begin{corollary}\label{cor:weighted1moment}
Assume GRC. Suppose $\eta_{\psi}(D)=1$. Also, suppose $P \le K^{\delta}$ for some $\delta>0$ sufficiently small.
Then for any $A \ge 1$ we have that
\[
\sum_{k \in \mathbb Z } \frac{L(\tfrac12, \psi \times \phi_{2k})}{L(1, \phi_{2k})^2} \phi\left( \frac{k}{K}\right)
 = \tilde{\phi}(1) \cdot C_{D,\psi}' \cdot C_{D,\psi} \cdot K+O\bigg( \frac{K}{(\log K)^A}\bigg)
\]
where
\begin{multline}\label{eqn:C_Dpsi}
C_{D,\psi}'= \prod_{(p,D)=1}\left(1-\frac{2\vartheta(p)r_D(p)}{p^{3/2}} +\frac{3\chi_D(p)+h(p^2)}{p^3}+\frac{2\vartheta(p)r_D(p) \chi_D(p)}{p^{5/2}}+\frac{1}{p^5} \right) \\
 \times \prod_{p|D}\left(1-\frac{2}{p}+\frac{1}{p^2}  \right),
\end{multline}
where $r_D(p)=\sum_{N((\beta))=p}1$.
\end{corollary}

\subsection{Preliminary lemmas} For $\xi>0$ let
\[
W_s(\xi)= \frac{1}{2\pi i} \int_{(c)} \left(\frac{D^{3/2}}{ \xi k^2}  \right)^w L(2w+2s, \chi_D) \frac{\gamma(s+w, \psi \times \phi_{2k})}{\gamma(s, \psi \times \phi_{2k})} \frac{e^{w^2}}{w} \, dw,
\]
where $c>0$ and $\gamma(s,\psi \times \phi_{2k})$ is as in \eqref{eq:gamma-factor}.
\begin{lemma} \label{lem:AFE}
For $0 \le \tmop{Re}(s) \le 1$ we have that
\[
L(s, \psi \times \phi_{2k})=\sum_{ n \ge 1} \frac{\lambda_{2k}(n) \lambda_{\psi}(n)}{n^s} W_s\left(\frac{n}{k^2} \right)+\eta_\psi(D) \sum_{n \ge 1} \frac{\lambda_{2k}(n) \lambda_{\psi}(n)}{n^{1-s}} W_{1-s}\left( \frac{n}{k^2}\right).
\]
Write $W_{1/2}=W$, we have that
\[
\xi^j W^{(j)}\left( \xi \right) \ll \frac{1}{1+|\xi|^A} \qquad \text{ and } \qquad W(\xi)=L(1, \chi_D)+O\left( \xi^{\frac12-\vartheta-\varepsilon} \right).
\]
\end{lemma}
\begin{proof}
The first claim is the well-known formula for the approximate functional equation. For example see Theorem 5.3 of \cite{iwaniec2004analytic}.

To establish the claimed bounds for $W$, first shift the contour of integration to $\tmop{Re}(w)=A$ and observe that by Stirling's formula $k^{-2\tmop{Re}(w)} |\frac{\gamma(w+\frac12, \psi \times \phi_{2k})}{\gamma(\frac12, \psi \times \phi_{2k})}|\ll 1$. Note that $|\tmop{Im} t_{\psi}| \le \vartheta$. Hence, shifting contours to $\tmop{Re}(w)=-\frac12+\vartheta+\varepsilon$ and collecting a pole at $w=0$ with residue $L(1, \chi_D)$ we obtain the second claimed bound.
\end{proof}
For $\xi>0$ let
\begin{equation} \label{eq:Fdef}
F\left(\xi;K, N \right)=\phi\left( \frac{\xi}{K}\right) W\left( \frac{N}{\xi^2} \right) \quad \text{ and } \quad
\widehat F\left(\lambda;K, N \right)= \int_{\mathbb R} F\left(\xi;K, N \right) e(-\lambda \xi) \, \dd \xi.
\end{equation}
In particular, note that
\begin{equation} \label{eq:Fest}
\begin{split}
 \widehat F(\lambda_1;K,N)=&\widehat F(\lambda_2;K,N)+O(K|\lambda_1-\lambda_2|), \\
 \widehat F(\lambda;K,N_1)=&\widehat F(\lambda;K,N_2)+O\left(\frac{|N_1-N_2|}{K}\right)
 \end{split}
\end{equation}
for $N_1,N_2 \gg 1$,
which we will use later. We will also require the following additional estimates for $\widehat F$.
\begin{lemma} \label{lem:IBP1}
We have that
\[
|\widehat F\left(\lambda;K, N \right)| \ll K \min\left\{ \left(\frac{K^2}{N} \right)^A , \left(\frac{P}{|\lambda K|}\right)^A\right\}
\]
for any $A \ge 1$.
\end{lemma}
\begin{proof}
By Lemma \ref{lem:AFE}, $|W(N/\xi^2)| \ll (\xi^2/N)^A$, hence
\[
|\widehat F\left(\lambda;K, N \right)| \ll K \left(\frac{K^2}{N} \right)^A.
\]
To get the second bound, observe for any integer $A \ge 1$ and $\lambda \neq 0$ that by making a change of variables then integrating by parts gives
\begin{equation} \label{eq:IBPeq}
\widehat F(\lambda;K,N) =\frac{K}{(-2\pi i \lambda K)^A} \int_{\mathbb R} e(-\lambda  K \xi) \cdot  \frac{d^A}{d\xi^A} F(K \xi; K,N) \, \dd \xi.
\end{equation}
To bound $\frac{d^A}{d\xi^A} F(K \xi; K,N)$ we note that $\phi$ has compact support and for any $A_1,A_2 \ge 0$
\[
\left(\frac{N}{K^2} \right)^{A_1}  W^{(A_1)}\left( \frac{N}{K^2 \xi^2} \right) \phi^{(A_2)}(\xi) \ll P^{A_2}.
\]
Hence, we get that $
\frac{d^A}{d\xi^A} F(K \xi; K,N) \ll P^A$ and using this in \eqref{eq:IBPeq} completes the proof.
\end{proof}

\subsection{Estimates for diagonal and off-diagonal terms}
In this section we estimate the sums which appear
in the first moment computation.
 For $\tmop{Re}(s)\ge 1/4$, let $\mathcal L(s,\cdot )$ be the multiplicative function given by
 \begin{equation} \label{eq:ldef}
\mathcal L(s,p^b)= \frac{ \sum_{j \ge 0} \frac{\lambda_{\psi}(p^{b+2j})}{p^{js}}}{ \sum_{j \ge 0} \frac{\lambda_{\psi}(p^{2j})}{p^{js}}}.
 \end{equation}
 We first give the following  local computation.
\begin{lemma}\label{lemma:local}
  If $p\nmid D$ then we have
  \begin{equation}\label{eqn:L(s,p^b)}
    \mathcal{L}(s,p^b)
    = \left( \lambda_\psi(p^b) - \frac{\lambda_\psi(p^{b-2})}{p^s} \right) \left(1+\frac{1}{p^s}\right)^{-1}.
  \end{equation}
  Here we use the notation $\lambda_\psi(p^{-1})=0$.

\end{lemma}

Note that if $p\mid D$, then $\mathcal{L}(s,p^b) = \lambda_{\psi}(p^b)$.
Note that we have defined the multiplicative function $\vartheta(p^b):=\mathcal{L}(1,p^b)$.

\begin{proof}

Assume $p\nmid D$.  To prove \eqref{eqn:L(s,p^b)}, we will use the Hecke relation
\begin{equation}\label{eqn:HL}
  \lambda_\psi(p^{\alpha+\beta}) = \lambda_\psi(p^\alpha) \lambda_\psi(p^\beta) - \lambda_\psi(p^{\alpha-1}) \lambda_\psi(p^{\beta-1}), \quad \textrm{if } \alpha,\beta\geq1.
\end{equation}
We first consider the case $b=1$. Since $\lambda_\psi(p^{2j+1}) = \lambda_\psi(p^{2j}) \lambda_\psi(p) - \lambda_\psi(p^{2j-1})$, we have
\begin{align*}
  \sum_{j\geq0} \frac{\lambda_\psi(p^{1+2j})}{p^{js}}
  & = \lambda_\psi(p) \sum_{j\geq0} \frac{\lambda_\psi(p^{2j})}{p^{js}} - \sum_{j\geq1} \frac{\lambda_\psi(p^{2j-1})}{p^{js}} \\
  & = \lambda_\psi(p) \sum_{j\geq0} \frac{\lambda_\psi(p^{2j})}{p^{js}} - \frac{1}{p^s} \sum_{j\geq0} \frac{\lambda_\psi(p^{2j+1})}{p^{js}}.
\end{align*}
Hence,
\begin{align*}
  \sum_{j\geq0} \frac{\lambda_\psi(p^{1+2j})}{p^{js}}
  & = \lambda_\psi(p) \left(1+\frac{1}{p^s}\right)^{-1} \sum_{j\geq0} \frac{\lambda_\psi(p^{2j})}{p^{js}} .
\end{align*}
Now we consider the case $b\geq2$. By \eqref{eqn:HL} with $\alpha=2j$ and $\beta=b$, we have that
\begin{align*}
  \sum_{j\geq0} \frac{\lambda_\psi(p^{b+2j})}{p^{js}}
  & = \lambda_\psi(p^b) \sum_{j\geq0} \frac{\lambda_\psi(p^{2j})}{p^{js}} - \lambda_\psi(p^{b-1})\sum_{j\geq1} \frac{\lambda_\psi(p^{2j-1})}{p^{js}} \\
  & = \lambda_\psi(p^b) \sum_{j\geq0} \frac{\lambda_\psi(p^{2j})}{p^{js}} - \frac{\lambda_\psi(p^{b-1})}{p^s} \sum_{j\geq0} \frac{\lambda_\psi(p^{2j+1})}{p^{js}} \\
  & = \left( \lambda_\psi(p^b) - \frac{\lambda_\psi(p^{b-1})}{p^s} \lambda_\psi(p) \left(1+\frac{1}{p^s}\right)^{-1} \right) \sum_{j\geq0} \frac{\lambda_\psi(p^{2j})}{p^{js}}.
\end{align*}
By \eqref{eqn:HL} again, we prove \eqref{eqn:L(s,p^b)}.
\end{proof}

 \begin{lemma} \label{lem:diagonal}
Let $a \in \mathbb N$. Then
for $1 \le \Lambda_2 \le K$
\begin{multline*}
  \sum_{n \ge 1} \frac{\lambda_{\psi}(an^2)}{n} \widehat F(0; K, \Lambda_2 n^2)  \\
  =\vartheta(a) \cdot \frac{K}{\zeta_D(2)}\tilde\phi(1)L(1, \tmop{sym}^2 \psi)L(1, \chi_D)+O\left( a^{\vartheta+\varepsilon} \cdot \left( \frac{K}{\sqrt{\Lambda_2}}\right)^{\frac12+\varepsilon}\right),
\end{multline*}
 where $\vartheta(a)$ is defined as in \eqref{eqn:theta}.
 \end{lemma}
 \begin{proof}
 Using the bound $\lambda_{\psi}(a) \ll a^{\vartheta+\varepsilon}$ we get that for $\tmop{Re}(s) \ge 1/4$
 \begin{equation} \label{eq:easyLbd}
     \mathcal L(s,a) \ll a^{\vartheta+\varepsilon}.
 \end{equation}

 For sake of brevity, write $\gamma(s)=\gamma(s,\psi \times \phi_{2k})$. By the definition of $\widehat F$ we get that
 \[
\begin{split}
\sum_{n \ge 1} &\frac{\lambda_{\psi}(an^2)}{n} \widehat F(0; K, \Lambda_2 n^2)\\
&=K \frac{1}{2\pi i} \int_{(c)}\tilde\phi(2s+1) \left( \frac{D^{3/2} K^2}{ \Lambda_2 k^2} \right)^s L(2s+1,\chi_D) \frac{\gamma(s+\tfrac12)}{\gamma(\tfrac12)} \, \sum_{n \ge 1} \frac{\lambda_{\psi}(an^2)}{n^{2s+1}}  \, \frac{e^{s^2}}{s}\,  \dd s \\
&=K \frac{1}{2\pi i} \int_{(c)}\tilde\phi(2s+1) \left( \frac{D^{3/2} K^2}{ \Lambda_2 k^2} \right)^s L(2s+1,\chi_D) \times \\
& \qquad \qquad \qquad \qquad \times \frac{\gamma(s+\tfrac12)}{\gamma(\tfrac12)} \, \frac{\mathcal L(2s+1,a)}{\zeta_D(4s+2)} L(2s+1, \tmop{sym}^2 \psi)\, \frac{e^{s^2}}{s}\,  \dd s.
\end{split}
 \]
 The integrand is analytic in the region $\tmop{Re}(s) \ge -\tfrac14+\varepsilon$ except for a simple pole at $s=0$ and in this region bounded by $\ll \left(\frac{K^2}{\Lambda_2}\right)^{\tmop{Re}(s)} \cdot \frac{|\mathcal L(2s+1,r)|}{|s|} e^{\tmop{Re}(s^2)}$, for $A$ sufficiently large. This bound follows upon using Stirling's formula (the implicit constant depends on $t_{\psi}$ and $D$).
 Shifting contours to $\tmop{Re}(s)=-\tfrac14+\varepsilon$, collecting a simple pole at $s=0$, and using the preceding bound along with \eqref{eq:easyLbd} we get that the RHS equals
 \[
 K \tilde\phi(1) L(1, \chi_D) \frac{\mathcal L (1,a)}{\zeta_D(2)}L(1, \tmop{sym}^2 \psi)+O\left(a^{\vartheta+\varepsilon} \cdot \Lambda_2^{-1/4+\varepsilon}\cdot K^{\frac12+\varepsilon}  \right).
 \]
 This completes the proof.
 \end{proof}

Let
 \[
 Q(X,Y)=(X+Y \omega_D) (X+Y \tilde\omega_D).
 \]
Also, given $\gamma=\begin{pmatrix}
 a & b \\
 c& d\end{pmatrix}\in \tmop{ GL}_2(\mathbb Z)$ write
 \[
 Q^{\gamma}(X,Y)=Q(aX+cY, bX+dY).
 \]
 Also, let $\lVert \gamma \rVert_1=\max\{ |a|+|c|, |b|+|d|\}$.



 Using the results of Section \ref{sec:nonsplit} leads to the following estimate.

 \begin{lemma} \label{lem:off}
 Let $\gamma \in \tmop{SL}_2(\mathbb Z)$ and $h \in \mathbb Z$ with $ h \neq 0$. Let $\Lambda \in \mathbb R$ with $\Lambda \neq 0$
Write $Q^{\gamma}(X,h)=a_{\gamma} X^2+b_{\gamma,h} X+c_{\gamma,h}$.
Suppose that $a_{\gamma}$ divides $b_{\gamma,h}$ or GRC holds.
 Then there exists $A_0>0$ such that
 \begin{equation} \label{eq:lemmaoffdiag}
 \begin{split}
 \bigg|\sum_{\substack{r \ge 1}}  \frac{\lambda_{\psi}(Q^{\gamma}(r,h))}{r} \widehat F\left(\Lambda \cdot \frac{h}{r}; K, a_{\gamma} r^2 \right)
\bigg| \ll (P|h|(|\Lambda|+|\Lambda|^{-1}) \lVert \gamma \rVert_1))^{A_0} K^{\frac12+\vartheta+\varepsilon}. \end{split}
 \end{equation}
  \end{lemma}

\begin{proof}

We may assume that $|\Lambda|^{\pm1},P,|h|, \lVert \gamma \rVert_1 \le K^{\delta}$ for some $\delta$ sufficiently small,
since otherwise the result is trivial.
For sake of brevity, we write $Q(r)=Q^{\gamma}(r,h)$, note that the discriminant of $Q$ equals $\Delta=Dh^2$.
By Lemma \ref{lem:IBP1} it follows that for $|h| \ge 1$ the above sum is effectively restricted to $r$ with $\frac{ |\Lambda|K^{1-\varepsilon}}{P
}\le r \le  K^{1+\varepsilon}$. Also, using \eqref{eq:Fest}
we get that
\begin{equation} \label{eq:unityclean}
\begin{split}
& \sum_{\frac{|\Lambda|K^{1-\varepsilon}}{P
}\le r \le  K^{1+\varepsilon}}
\frac{ \lambda_{\psi}(Q(r))}{r} \widehat F\left(\Lambda \cdot \frac{h}{r}; K, a_{\gamma} r^2 \right)\\
&= \sqrt{a_{\gamma}}\sum_{\frac{|\Lambda|K^{1-\varepsilon}}{P
}\le r \le  K^{1+\varepsilon}}
\frac{ \lambda_{\psi}(Q(r))}{\sqrt{Q(r)}} \widehat F\left(\Lambda \cdot \frac{h \sqrt{a_{\gamma}}}{\sqrt{Q(r)}}; K, Q(r) \right)+O(K^{1/2}).
\end{split}
\end{equation}

We now introduce a smooth partition of unity.
Let $U$ be a non-negative, smooth function on with compact support on $[1,2]$ with $U^{(j)}(x) \ll 1$ such that
\[
\sum_{M \in \mathcal S} U\left(\frac{x}{M} \right)=1, \qquad \forall x \in \mathbb R_{>0},
\]
for some $\mathcal S \subset \mathbb R_{>0}$ which satisfies $\# \{ M \in \mathcal S : X^{-1} \le M \le X \} \ll \log(1+X)$, for $X \ge 1$.
Applying the above smooth partition of unity and making a change of variables we get that
\[
\begin{split}
&\sum_{\frac{|\Lambda|K^{1-\varepsilon}}{P
}\le r \le  K^{1+\varepsilon}}
\frac{ \lambda_{\psi}(Q(r))}{\sqrt{Q(r)}} \widehat F\left(\Lambda \cdot \frac{h \sqrt{a_{\gamma}}}{\sqrt{Q(r)}}; K, Q(r) \right)\\
&= \sum_{\sqrt{\frac{|\Lambda|}{2PK^{\varepsilon}}} \le M \le 2  K^{\varepsilon}}\sum_{r \ge 1}
 \lambda_{\psi}(Q(r)) \cdot \frac{1}{\sqrt{Q(r)}} \widehat F\left( \Lambda \cdot \frac{h \sqrt{a_{\gamma}}}{\sqrt{Q(r)}}; K, Q(r) \right)U\left( \frac{Q(r)}{K^2 M}\right).
\end{split}
\]
The inner sum equals
\[
 \int_{\mathbb R} \phi(\xi)
\sum_{r \ge 1} \lambda_{\psi}(Q(r)) \cdot \frac{K}{\sqrt{Q(r)}}
W\left( \frac{Q(r)}{K^2 \xi^2} \right) e\left( \frac{-\Lambda h \sqrt{a_{\gamma}} K \xi }{\sqrt{Q(r)}}\right) U\left(\frac{Q(r)}{K^2M}\right) \, \dd \xi.
\]
Applying Theorem \ref{lemma:S=int(D)} with $Y=K^2M$ and smooth function  $G(x)=\frac{1}{\sqrt{Mx}}W(\frac{M x}{\xi^2})U(x)e(-\Lambda h \frac{\sqrt{a_{\gamma}} \xi}{\sqrt{M x}})$ it follows that there exists $A_0>0$ such that the above sum is
\[
\ll  (P |h| \lVert \gamma \rVert_1 (|\Lambda|+|\Lambda|^{-1}))^{A_0} K^{\frac12+\vartheta +\varepsilon}.
\]
Using this bound in \eqref{eq:unityclean} completes the proof.
\end{proof}

Given $(\alpha),(\beta) \subset \mathcal O_{\mathbb Q(\sqrt{D})}$ we also need estimates for the quantity $\ell -\frac{\log|\frac{\alpha \cdot \beta}{ (\widetilde \alpha \cdot \widetilde \beta)}|}{\log \epsilon_D}$, for $\ell=0,1$.
Define the angle of a nonzero element $\gamma\in \mathcal O_{\mathbb Q(\sqrt{D})}$ as $\theta_\gamma=\log|\gamma/\tilde{\gamma}|$. Here and throughout $\epsilon_D=x+y\omega_D>1$ is the fundamental unit. Since $0<\tilde{\varepsilon}_D<1$, we have $x\geq2$ and $y\geq1$.

\begin{lemma} \label{lem:algebra} Let $(\alpha), (\beta) \subset \mathcal O_{\mathbb Q(\sqrt{D})}$, with $\alpha=m+n \omega_D, \beta=M+N \omega_D$ and $m,n, M,N \in \mathbb Z$.
Also,
define
\begin{align*}
C_1=& \frac{-(M,N) \sqrt{D}}{\beta \log \epsilon_D}, \qquad
C_2 =\frac{-(M,N)}{\beta \log \epsilon_D} \\
C_3 =&\frac{(M,N)}{ \beta \log \epsilon_D} \cdot ((x-1,x+y \omega_D \cdot \tilde\omega_D)+(y,1+x) \omega_D ), \\
C_4= &\frac{-(M,N)}{ \beta \log \epsilon_D} \cdot ((1+x,y \omega_D \cdot \tilde\omega_D+x)+(y,1-x) \omega_D).
\end{align*}
Let
\begin{align*}
    a_1=&\frac{N}{(M,N)}, b_1=- \frac{M+N}{(M,N)},  a_2=\frac{2M+N}{(M,N)}, b_2=-\frac{M+N-2N \omega_D \cdot \tilde\omega_D}{(M,N)},  \\
    a_3=&\frac{My-N(1+x)}{(M,N)(y,1+x)}, b_3= \frac{N \omega_D \cdot \tilde\omega_Dy+(M+N)(1+x)}{(M,N)(y,1+x)}, \\  a_4=&\frac{N(1-x)+My}{(M,N)(1-x,y)}, b_4= \frac{N \omega_D \cdot \tilde\omega_Dy+(M+N)(x-1)}{(M,N)(1-x,y)}.
\end{align*}
If $|\theta_{\alpha \cdot \beta}| \le \frac{1}{10}$ then
\[
\frac{-\theta_{\alpha \cdot \beta}}{\log \epsilon_D}=
\begin{cases}
\frac{C_{1}}{\alpha} \cdot \left(a_1 \cdot m-b_1 \cdot n \right)  +O\left(\left(\frac{(\widetilde \alpha \cdot  \widetilde \beta)}{\alpha \cdot \beta}-1\right)^2 \right) &  \text{ if } N(\alpha \cdot \beta)>0, \\
\frac{C_{2}}{\alpha}  \cdot \left(a_2 \cdot m-b_2 \cdot n \right)+O\left(\left(\frac{(\widetilde\alpha \cdot \widetilde \beta)}{\alpha \cdot \beta}+1\right)^2 \right) &  \text{ if } N(\alpha \cdot \beta)<0.
\end{cases}
\]
If $|1-\frac{\theta_{\alpha \cdot \beta}}{\log \epsilon_D}| \le \tfrac{1}{10}$ and $N(\epsilon_D)=1$ then
\[
1-\frac{\theta_{\alpha \cdot \beta}}{\log \epsilon_D}=
\begin{cases}
\frac{C_{3}}{\alpha}  \cdot \left(a_3 \cdot m-b_3 \cdot n \right) +O\left(\left( \epsilon_D\frac{(\widetilde \alpha \cdot \widetilde \beta)}{\alpha \cdot \beta}-1\right)^2 \right) &  \text{ if } N(\alpha \cdot \beta)>0, \\
\frac{C_{4}}{\alpha}  \cdot (a_{4} \cdot m-b_4 \cdot n)+O\left(\left(\epsilon_D\frac{(\widetilde\alpha \cdot \widetilde\beta)}{\alpha \cdot \beta}+1\right)^2 \right) &  \text{ if } N(\alpha \cdot \beta)<0.
\end{cases}
\]
\end{lemma}
\begin{remark}
Note for $j=1,2,3,4$ that $C_j,a_j,b_j$ do not depend on $\alpha$ and $C_j \ll 1$, $a_j,b_j \ll |M|+|N|$ where the implied constants depend at most on $D$.
\end{remark}
\begin{proof}
First suppose that $N(\alpha \cdot \beta)>0$ and write $\gamma=\alpha\cdot \beta=r+s \omega_D$, where $r=Mm-Nn \omega_D \cdot \tilde \omega_D$, $s=Nm+Mn+Nn$. Then for $\ell=0,1$ we get
\[
\begin{split}
\ell-\frac{\theta_{\gamma}}{\log \epsilon_D}
=& \frac{1}{\log \epsilon_D}\log\left(1+\left(\frac{\tilde \gamma}{\gamma} \epsilon_D^{\ell} -1\right) \right) \\
=&\frac{1}{\gamma \log \epsilon_D} \left( \tilde\gamma \epsilon_D^{\ell}-\gamma \right)+O\left(\left|\frac{\tilde \gamma}{\gamma}\epsilon_D^{\ell}-1\right|^2 \right).
\end{split}
\]
Hence, if $\ell=0$ the claim follows since $\tilde\gamma-\gamma=s(\tilde\omega_D-\omega_D)$.
For $\ell=1$
we suppose $N(\epsilon_D)=1$. Observe that for any $A_1,A_2,B_1,B_2 \in \mathbb Z$ with $\frac{A_1}{B_1}=\frac{A_2}{B_2}$ and  $A_1A_2>0$ that for $r,s \in \mathbb Z$
\[
A_1 r+B_1 s+(A_2 r+B_2 s) \omega_D= ((A_1,B_1)+(A_2,B_2)\omega_D) \cdot  \left(\frac{A_2}{(A_2,B_2)} r+\frac{B_2}{(A_2,B_2)} s\right).
\]
Since $\frac{y}{1+x}={\frac{1-x}{x+y \omega_D \cdot \tilde\omega_D}}$ using the above observation yields
\[
\begin{split}
\tilde\gamma \epsilon_D-\gamma
=&((x-1)r+s(x+y \omega_D \cdot \tilde\omega_D))+(ry-s(1+x)) \omega_D \\
=&( (x-1,x+y \omega_D \cdot \tilde\omega_D)+(y,1+x) \omega_D ) \cdot \left( r \cdot \frac{y}{(y,1+x)}-s \cdot \frac{1+x}{(y,1+x)}\right),
\end{split}
\]
which completes the proof for the case $\ell=1$, $N(\alpha \cdot \beta)>0$.

For $N(\alpha \cdot \beta)<0$ we get that
\[
\begin{split}
\ell-\frac{\theta_{\gamma }}{\log \epsilon_D}
=& \frac{1}{\log \epsilon_D}\log\left(1-\left(\frac{\tilde\gamma}{\gamma} \epsilon_D^{\ell} +1\right) \right) \\
=& \frac{-1}{\gamma \log \epsilon_D} \left( \tilde\gamma \epsilon_D^{\ell}+\gamma \right)+O\left(\left|\frac{\tilde\gamma}{\gamma}\epsilon_D^{\ell}+1\right|^2 \right),
\end{split}
\]
and if $\ell=0$ we are done since $\tilde\gamma+\gamma=2r+s$.
For $\ell =1$ and recalling that $N(\epsilon_D)=1$, we argue as before to get
\[
\begin{split}
\epsilon_D\tilde\gamma+\gamma=&
(r(1+x)+s(y \omega_D \cdot \tilde\omega_D+x))+(ry+s(1-x))\omega_D \\
=&((1+x,y \omega_D \cdot \tilde\omega_D+x)+(y,1-x) \omega_D) \cdot \left(r \cdot  \frac{y}{(y,1-x)}-s \cdot \frac{(x-1)}{(y,1-x)} \right),
\end{split}
\]
which completes the proof.
\end{proof}

The numbers $a_j,b_j$, $j=1,2,3,4$, arise in the calculation of the constant $C_{D,\psi}$ appearing in Proposition \ref{prop:moment} and are analyzed in the next two lemmas.
\begin{lemma} \label{lem:factor}
We have that $(1+x,y)(1-x,y)=y$ and we can define $p_1$ and $p_2$ as
\[
Q\left( \frac{1+x}{(1+x,y)}, \frac{y}{(1+x,y)} \right)=p_1 \qquad \text{ and } \qquad Q\left( \frac{x-1}{(1-x,y)}, \frac{y}{(1-x,y)} \right)=-p_2.
\]
\end{lemma}

\begin{remark}
A simple consequence of the preceding lemma, which we will use later is that
\begin{equation} \label{eq:noramanujan}
p_1 |  \frac{2y}{(x+1,y)} \omega_D \tilde\omega_D+\frac{x+1}{(x+1,y)} \qquad
\text{and} \qquad
p_2 |  \frac{2y}{(x-1,y)}\omega_D \tilde\omega_D+\frac{x-1}{(x-1,y)}.
\end{equation}
To see this first write $g_1=(x+1,y)$, $g_2=(x-1,y)$, $r= \frac{2y}{g_1} \omega_D \tilde{\omega}_D+\frac{x+1}{g_1}$, $s=\frac{2y}{g_2}\omega_D \tilde{\omega}_D+\frac{x-1}{g_2}$, $u=2x+2+y$, $v=2x-2+y$ and
note that $g_2us=-D(x+1)y$, $g_1vr=-Dy(x-1)$. The lemma gives that $u=g_1^2p_1$, $v=g_2^2p_2$. Hence, combining formulas we have that $sg_1=-(x+1)p_2$, $rg_2=-(x-1)p_1$ so that $s=-\frac{x+1}{g_1} \cdot p_2$ and $r=-\frac{x-1}{g_2}p_1$.
\end{remark}
\begin{proof}
 For $p \neq 2$, if $p^a || y$ then
\[
0=1-N(\epsilon_D) \equiv 1- x^2 \pmod{p^a};
\]
so either $p^a | x+1$ or $p^a|x-1$. Hence, $p^a||(x+1,y)(x-1,y)$. If $p=2$ and $2^a||y$ with $a \ge 1$ then $x$ is odd and $(x+1,x-1)=2$. As before, $2^a| (x+1)(x-1)$ so $2^{a-1}||(x+1,y)$ or $2^{a-1}||(x-1,y)$. It follows that $2^a||(x+1,y)(x-1,y)$. Hence, $(1+x,y)(1-x,y)=y$ as claimed.

Observe that since $N(\epsilon_D)=1$ we have
\begin{equation}\notag
\begin{split}
&Q\left( \frac{1+x}{(1+x,y)}, \frac{y}{(1+x,y)} \right)Q\left( \frac{x-1}{(1-x,y)}, \frac{y}{(1-x,y)} \right)\\
&=\frac{1}{(1+x,y)^2(1-x,y)^2} \cdot (2x+y+2)(2-2x-y) \\
&=\frac{-1}{(1+x,y)^2(1-x,y)^2} \cdot D y^2=-D.
\end{split}
\end{equation}
Also each factor above on the LHS is not equal to $\pm 1$, since $\epsilon_D=x+y \omega_D$ is the fundamental unit.
\end{proof}

\begin{lemma} \label{lem:diagterms}
Let $(\beta) \subset \mathcal O_{\mathbb Q(\sqrt{D})}$ and write $\beta=M+N \omega_D=(M,N)(M_1+N_1 \omega_D)$. Also, let $n_{\beta}=Q(M_1,N_1)$.
For $j=1,2,3,4$, let $a_j,b_j$ be as defined in Lemma \ref{lem:algebra}.
Then
\[
Q(b_1,a_1)=n_{\beta}, Q(b_2,a_2)=-D n_{\beta}, Q(b_3,a_3)=p_1 n_{\beta}, Q(b_4,a_4)=-p_2 n_{\beta}.
\]
Additionally, if $(N(\beta),D)=1$ then $(a_1,b_1)=(a_2,b_2)=(a_3,b_3)=(a_4,b_4)=1$.
\end{lemma}

\begin{proof}
The first claim follows from a slightly tedious yet direct and elementary computation upon using Lemma \ref{lem:factor}.

To prove the second claim first note that since $(N,M+N)=(M,N)$ we get $(a_1,b_1)=1$. Next, observe that since $(N(\beta),D)=1$ and $(a_2,b_2)^2|D n_{\beta}$ we must have $((a_2,b_2),D)=1$, since $D$ is square-free. So $(a_2,b_2)|a_2+2b_2=-D \cdot \frac{N}{(M,N)}$ which implies $(a_2,b_2)|\frac{M}{(M,N)}$. Hence, $(a_2,b_2)=1$. To prove $(a_3,b_3)=1$ notice $(a_3,b_3)^2 | p_1 n_{\beta}$ so as before $p_1 \nmid (a_3,b_3)$. Hence, using Lemma \ref{lem:factor} we get $(a_3,b_3)|\frac{(1+x)}{(x+1,y)}a_3-\frac{y}{(y,x+1)}b_3=-p_1 \frac{N}{(M,N)}$ so $(a_3,b_3)|\frac{N}{(M,N)} $ and it follows $(a_3,b_3)| \frac{M}{(M,N)}$ so $(a_3,b_3)=1$. Finally, just as before $p_2 \nmid (a_4,b_4)$ so using Lemma \ref{lem:factor} we get $(a_4,b_4) |\frac{(1-x)}{(x-1,y)}a_4+\frac{y}{(x-1,y)}b_4=-p_2 \frac{N}{(M,N)}$ and we conclude $(a_4,b_4)=1$.
\end{proof}

\subsection{Proof of Proposition \ref{prop:moment}}
Recall that we assumed the class number of $\mathcal O_{\mathbb Q(\sqrt{D})}$ equals $1$.
Let $I=[-\frac13,2\log \epsilon_D-\frac13)$.
Given $\mathfrak a \subset \mathcal O_{\mathbb Q(\sqrt{D})}$, let $\theta_{\mathfrak a}=I \cap \{ \log |\frac{\alpha}{\tilde\alpha}| : (\alpha)=\mathfrak a\}$ be the angle of $\mathfrak{a}$. Note that
\[
\Xi_k(\mathfrak a)=e\left(k \cdot \frac{\theta_{\mathfrak a}}{2 \log \epsilon_D} \right).
\]

Let $\mathfrak b=(\beta) \in \mathcal O_{\mathbb Q(\sqrt{D})}$ and write $\beta=M+N\omega_D=(M,N)(M_1+N_1 \omega_D)$.  Let $n_{\beta}=Q(M_1,N_1)$ and $\epsilon = \tmop{sgn}(N(\beta))$.
 We will prove the following result, which implies Proposition \ref{prop:moment}.

\begin{proposition} \label{prop:moment2} Assume the GRC. Suppose $\eta_{\psi}(D)=1$. Also,
let $(\beta)\subset \mathcal O_{\mathbb Q(\sqrt{D})}$ with $(N(\beta),D)=1$. Then there exists $A_0 >0$ such that
\[
\sum_{k \in \mathbb Z} L(\tfrac12, \psi \times \phi_{2k}) \cdot  \Xi_{2k}((\beta)) \, \phi\left( \frac{k}{K} \right)= \tilde{\phi}(1) \cdot C_{D,\psi} \cdot  \frac{\vartheta(n_{\beta})}{\sqrt{n_{\beta}}} \cdot K +O((P|N(\beta)|)^{A_0} \cdot K^{\frac12+\vartheta+\varepsilon}),
\]
where $C_{D,\psi}$ is as in \eqref{eq:constdef}.
\end{proposition}

\begin{remark} \label{rem:1}
The number $n_{\beta}$ is independent of choice of generator of $(\beta)$, which can be seen by an elementary argument as follows.
Let $\alpha=\varepsilon^n \beta$ and write $\alpha=a+b\omega_D$, $\varepsilon^n=x_n+y_n\omega_D$.
First note that $$a=Mx_n-Ny_n \omega_D \tilde{\omega}_D, b=N(x_n+y)+My_n.$$
Write $g=(a,b) $. Clearly $(M,N)|g$. Also
\begin{equation*} 
\begin{split}
    Mx_ny_n \equiv& Ny_n^2 \omega_D \tilde{\omega}_D \pmod g \\
     -Mx_ny_n \equiv& N(x_n^2+x_ny_n) \pmod g.
\end{split}
\end{equation*}
Hence, $g|N$. This implies $g|Mx_n$ and $g|My_n$, so that $g|M$ since $(x_n,y_n)=1$. Thus $(M,N)=(a,b)$.
\end{remark}

\begin{remark}\label{remark:prop}
  If $p$
is ramified in $\mathcal O_{\mathbb Q(\sqrt{D})}$ with $(\pi^2)=(p)$ then $\Xi_k(\pi)=1$, since $(\pi)=(\tilde\pi)$. Hence, for any $(\beta)\subset \mathcal O_{\mathbb Q(\sqrt{D})}$ there exist $\beta_1,\beta_2$ with $\beta_1 \cdot \beta_2=\beta$ and $(N(\beta_1),D)=1$ and $\Xi_k((\beta))=\Xi_k((\beta_1))$. Hence Proposition \ref{prop:moment} follows from Proposition \ref{prop:moment2}.
\end{remark}

We will now proceed to prove Proposition \ref{prop:moment2}, let $F$ and $\widehat F$ be as defined in \eqref{eq:Fdef}.
Using Lemma \ref{lem:AFE}, applying Poisson summation we get that
\begin{multline} \label{eq:moment1}
\sum_{k \in \mathbb Z} L(\tfrac12, \psi \times \phi_{2 k}) \cdot \Xi_{2k}(\mathfrak b) \, \phi\left( \frac{k}{K}\right)
= 2 \sum_{r \ge 1} \frac{\lambda_{\psi}(r)}{\sqrt{r}} \sum_{\substack{\mathfrak a \in \mathcal O_{\mathbb Q(\sqrt{D})} \\ N(\mathfrak a)=r}} \sum_{k \in \mathbb Z}  e\left( k \frac{\theta_{\mathfrak a \cdot \mathfrak b}}{\log \epsilon_D}\right) F\left(k;K,r\right)\\
=2\sum_{\ell=0,1} \sum_{r \ge 1} \frac{\lambda_{\psi}(r)}{\sqrt{r}}  \sum_{\substack{\mathfrak a \in \mathcal O_{\mathbb Q(\sqrt{D})} \\ N(\mathfrak a)=r}} \widehat F\left(\ell-\frac{\theta_{\mathfrak a \cdot \mathfrak b}}{\log \epsilon_D};K,r \right)+O\left( K^{-10}\right),
\end{multline}
where in the last step we used Lemma \ref{lem:IBP1} to bound the contribution from the terms with $\ell \neq 0,1$ (here we use the bound $|\ell -\frac{\theta_{\mathfrak a \cdot \mathfrak b}}{\log \epsilon_D}| \gg |\ell|$). Given $\alpha=m+n\omega_D$, we let $\theta(m,n)=\log |\frac{\alpha \beta}{\widetilde{(\alpha \beta)}}|$. For each $\mathfrak a \in \mathcal O_{\mathbb Q(\sqrt{D})}$ there exists a unique $\alpha =m+n\omega_D>0$ such that $\mathfrak a=(\alpha)$ and $\theta(m,n) \in I$.
For other $\alpha>0$ with $(\alpha)=\mathfrak a$ and $\theta(m,n) \notin I$ we have $|\ell-\frac{\theta(m,n)}{\log \epsilon_D}| \gg 1$, for $\ell=0,1$. Hence, using Lemma \ref{lem:IBP1} and noting that $N(\mathfrak a)=|Q(m,n)|$ the RHS of \eqref{eq:moment1} equals
\begin{equation} \label{eq:expand}
2 \sum_{\ell=0,1}\sum_{\substack{m,n \in \mathbb Z  \\ m+n \omega_D>0 }} \frac{\lambda_{\psi}(Q(m,n))}{\sqrt{|Q(m,n)|}} \widehat F\left(\ell- \frac{\theta(m,n)}{\log \epsilon_D}; K, |Q(m,n)| \right)+O\left( K^{-10}\right),
\end{equation}
where we use the convention $\lambda_{\psi}(-n)=\lambda_{\psi}(n)$ (note that $\psi$ is even).

We first require the following technical estimate. Let $\epsilon=\tmop{sgn}(N(\beta))$.
\begin{lemma} \label{lem:cleaned}
Let $(\beta)  \in \mathcal O_{\mathbb Q(\sqrt{D})}$ with $\beta=M+N \omega_D$.
For each $j=1,2,3,4$ let $a_j$, $b_j$ and $C_j$ be as in Lemma \ref{lem:algebra}. Also, let $\overline{a_j}$ and $\overline{b_j}$ be integers such that $a_j \overline{a_j}-b_j\overline{b_j}=1$ and $|\overline a_j| \le |b_j|, |\overline b_j | \le |a_j|$. Let
\[
\gamma_j=\begin{pmatrix}
b_j & a_j \\
-\overline{a_j} & -\overline{b_j}
\end{pmatrix}.
\]
Also, let $\ell_1=\ell_2=0$ and $\ell_3=\ell_4=1$ and $\widetilde C_j=-C_j/(b_j+a_j \omega_D)$.
Then for each $j=1,2,3,4$
\begin{equation} \label{eq:cleaned}
\begin{split}
&\sum_{\substack{m,n \in \mathbb Z \\ (-1)^{j-1} \cdot \epsilon \cdot Q(m,n) >0  \\ m+n \omega_D>0 }} \frac{\lambda_{\psi}(Q(m,n))}{\sqrt{|Q(m,n)|}} \widehat F\left(\ell_j- \frac{\theta(m,n)}{\log \epsilon_D}; K, |Q(m,n)| \right)\\
&= \sum_{\substack{r,h \in \mathbb Z, \, r \neq 0 \\ (-1)^{j-1} \cdot \epsilon \cdot Q^{\gamma_j}(r,h) >0 \\ (b_j+a_j \omega_D)r>(\overline{a_j}+\overline{b_j} \omega_D)h}}
\frac{\lambda_{\psi}(Q^{\gamma_j}(r,h))}{\sqrt{|Q(b_j,a_j)|} \cdot |r|}
\widehat F\left(\widetilde C_j \cdot \frac{h}{r}; K, |Q(b_j,a_j)|r^2 \right)+O\left( (P|\beta|\lVert \gamma_j \rVert_1)^{10} \cdot K^{1/2+\varepsilon}\right).
\end{split}
\end{equation}
\end{lemma}

\begin{proof}
Assume $N(\beta)>0$ (the case $N(\beta)<0$ is similar).
By Lemma \ref{lem:IBP1} we can restrict the range of summation in the sum on the LHS of \eqref{eq:cleaned} to $m,n$ with $|\ell_j- \frac{\theta(m,n)}{\log \epsilon_D}| \le P K^{-1+\varepsilon}$, $j=1,2,3,4$,
at the cost of an error of size $O(K^{-10})$.
Hence, applying Lemmas \ref{lem:IBP1} and \ref{lem:algebra} along with \eqref{eq:Fest} we get that the LHS of \eqref{eq:cleaned} equals
\begin{equation} \label{eq:step1}
\sum_{\substack{m,n \in \mathbb Z \\ (-1)^{j-1}Q(m,n) >0  \\ m+n \omega_D>0 }} \frac{\lambda_{\psi}(Q(m,n))}{\sqrt{|Q(m,n)|}} \widehat F\left(\frac{C_j}{m+n \omega_D} \cdot (a_j\cdot m-b_j \cdot n); K, |Q(m,n)| \right)+O(P^2 K^{1/2}).
\end{equation}
Since $\gamma_j \in \tmop{SL}_2(\mathbb Z)$, for each $(m,n) \in \mathbb Z^2$ there exists a unique $(r,h)\in \mathbb Z^2$ with
$\begin{pmatrix}
m & n \\
\end{pmatrix}=\begin{pmatrix}
r& h \\
\end{pmatrix} \cdot \gamma_j$.
So making the the change of variables $m=b_jr-\overline{a_j} h$ and $n=a_j r -\overline{b_j} h$ the above sum equals
\begin{equation} \label{eq:step2}
\sum_{\substack{r,h \in \mathbb Z, \, r \neq 0 \\ (-1)^{j-1}Q^{\gamma_j}(r,h) >0  \\ (b_j+a_j \omega_D)r>(\overline{a_j}+\overline{b_j} \omega_D)h }} \frac{\lambda_{\psi}(Q^{\gamma_j}(r,h))}{\sqrt{|Q^{\gamma_j}(r,h)|}} \widehat F\left(\widetilde {C}_j' \cdot \frac{h}{r}; K, |Q^{\gamma_j}(r,h)| \right)+O(K^{-10})
\end{equation}
where $\widetilde C_j'=-C_j/((b_j+a_j \omega_D)-(\overline{a_j}+\overline{b_j} \omega_D)h/r)$ and we have bounded the contribution from the sum over $h$ with $r=0$ by using Lemma \ref{lem:IBP1}.

By the rapid decay of $\widehat F$ established in Lemma \ref{lem:IBP1} the sum in \eqref{eq:step2} effectively restricted to $|h| \le K^{\varepsilon} |\beta| P \lVert \gamma_j \rVert^2$
and $ \frac{K^{1-\varepsilon}}{\lVert \gamma_j \rVert^2 |\beta| P}\le r \le K^{1+\varepsilon}$.
Also, for $|r| \ge |h|$ we have $\widetilde C_j'=\widetilde C_j+O(\lVert \gamma_j \rVert_1 |\frac{h}{r}|)$ and $Q^{\gamma_j}(r,h)=Q(b_j,a_j)r^2+O(\lVert \gamma_j \rVert_1^2 |hr|)$. Using these estimates
and Lemma \ref{lem:IBP1} along with \eqref{eq:Fest}
we get that \eqref{eq:step2} equals the RHS of \eqref{eq:cleaned}
 plus an error term of size
\[
\begin{split}
\ll  K^{1+\varepsilon} \cdot \lVert \gamma_j \rVert_1^{3} \cdot \sum_{1 \le |h| \le K^{\varepsilon} |\beta| P \lVert \gamma_j \rVert^2} |h| \sum_{\frac{K^{1-\varepsilon}}{\lVert \gamma_j \rVert^2 |\beta| P}\le r \le K^{1+\varepsilon}} \frac{1}{|r|^{3/4}} \cdot   \frac{1}{|r|} \ll (P|\beta|\lVert \gamma_j \rVert_1)^{10} \cdot K^{1/2}.
\end{split}
\]
This completes the proof.
\end{proof}

In the sum on the RHS of \eqref{eq:expand} split the sum over $m,n$ into two sums depending on whether $\epsilon \cdot Q(m,n)>0$ (i.e. whether $N(\alpha \cdot \beta)>0$).
Combining \eqref{eq:expand} and Lemma \ref{lem:cleaned}
it follows that the RHS of \eqref{eq:moment1}
equals
\begin{multline} \label{eq:keysum}
2 \sum_{j=1}^4 \sum_{\substack{r,h \in \mathbb Z \\ (-1)^{j-1} \cdot \epsilon \cdot Q^{\gamma_j}(r,h) >0 \\ (b_j+a_j \omega_D)r>(\overline{a_j}+\overline{b_j} \omega_D)h}}
\frac{\lambda_{\psi}(Q^{\gamma_j}(r,h))}{\sqrt{|Q(b_j,a_j)|} \cdot |r|}
\widehat F\left(\widetilde C_j \cdot \frac{h}{r}; K, |Q(b_j,a_j)|r^2 \right) \\
+O\left(  \max_{1 \le j \le 4} (P|\beta|\lVert \gamma_j \rVert_1)^{10} \cdot K^{1/2+\varepsilon} \right).
\end{multline}

\subsubsection{Diagonal terms}
The main terms in Proposition \ref{prop:moment} arise from the terms with $h=0$ in the inner sum in \eqref{eq:keysum}. By Lemma \ref{lem:diagterms}, for each $j=1,2,3,4$ we get $(-1)^{j-1} \cdot \epsilon \cdot Q^{\gamma_j}(r,0) =(-1)^{j-1}\cdot \epsilon \cdot Q(b_j,a_j)>0$. Applying Lemma \ref{lem:diagonal} the $h=0$ term in the inner sum in \eqref{eq:keysum} equals (replacing $r$ with $-r$ if necessary)
\begin{multline} \label{eq:evaluated}
\frac{1}{\sqrt{|Q(b_j,a_j)|}}\sum_{ r \ge 1} \frac{\lambda_{\psi}(Q(b_j, a_j) r)}{r} \widehat F\left(0; K, |Q(b_j,a_j)|r^2 \right) \\
 = \frac{\vartheta(|Q(b_j,a_j)|)}{\sqrt{|Q(b_j,a_j)|}} \cdot \frac{1}{\zeta_D(2)}\tilde\phi(1)L(1, \tmop{sym}^2 \psi)L(1, \chi_D) \cdot K +O\left(K^{1/2+\varepsilon} \right).
\end{multline}

\subsubsection{Off-diagonal terms}\label{subsec:offdiagonal}

We next consider the terms with $h \neq 0$ in the inner sum in \eqref{eq:keysum}. Using Lemma \ref{lem:IBP1} for each $j=1,2,3,4$ the sum is effectively restricted to $1 \le |h| \le P |\beta| \lVert \gamma_j \rVert_1^2 K^{\varepsilon}$. Also, for $|r| \ge K^{1/2+\varepsilon} \lVert \gamma_j \rVert$, $j=1,2,3,4$, the condition $(-1)^{j-1} \cdot \epsilon \cdot Q^{\gamma_j}(r,h)>0$ is satisfied by Lemma \ref{lem:diagterms}, since $(-1)^{j-1}\cdot \epsilon \cdot Q^{\gamma_j}(r,h)=(-1)^{j-1} \cdot \epsilon \cdot Q(b_j,a_j)r^2+O(|h r| \lVert \gamma_j \rVert_1^2)$.
Let $\epsilon_j=\tmop{sgn}(b_j+a_j \omega_D)$. Hence, using Lemma \ref{lem:IBP1} we can bound the $h \neq 0$ terms in \eqref{eq:keysum} by
\begin{equation} \label{eq:offest1}
\ll \frac{1}{\sqrt{|Q(b_j,a_j)|}} \sum_{1 \le |h| \le P \lVert \gamma_j \rVert_1^2 |\beta| K^{\varepsilon}}
\bigg|\sum_{r \ge 1} \frac{\lambda_{\psi}(Q^{\gamma_j}(\epsilon_j r,h))}{r} \widehat F\left(\widetilde C_j \cdot \frac{\epsilon_j \cdot h}{r}; K, |Q(b_j,a_j)|r^2 \right) \bigg|+K^{-10}.
\end{equation}
Let $Q^{\gamma}(X,h)=a_{\gamma}X^2+b_{\gamma,h}X+c_{\gamma,h}$.
For $\beta=1$, we have that $a_{\gamma_j}|b_{\gamma_j,h}$ by \eqref{eq:noramanujan}, $j=3,4$ and this also holds for $j=1,2$ by inspection. It follows that if $\beta=1$ or assuming GRC if $\beta \neq 1$ we can apply Lemma \ref{lem:off} (replacing $\gamma_j$ with $\begin{pmatrix} -b_j & -a_j \\ -\overline{a_j} & -\overline{b_j} \end{pmatrix}$ if $\epsilon_j=-1$). Hence, the inner sum in \eqref{eq:offest1} is
\begin{equation} \notag
\ll (P(|h|+\lVert \gamma_j \rVert^2 |\beta|) \lVert \gamma_j \rVert_1)^{A_0}  K^{\frac12+\vartheta+\varepsilon},
\end{equation}
unconditionally for $\beta=1$ and assuming GRC for $\psi$ for $(\beta) \subset \mathcal O_{\mathbb Q(\sqrt{D})}$. Therefore, the sum over $h$ in \eqref{eq:offest1} is
\begin{equation} \label{eq:offest2}
\ll (P |\beta| \lVert \gamma_j \rVert_1)^{A_0} K^{\frac12+\vartheta+\varepsilon}
\end{equation}

\subsubsection{Completion of the proof of Propositions \ref{prop:moment1}, \ref{prop:moment}, \ref{prop:moment2}} \label{sec:offdiag}

We now choose the generator of $(\beta)$ so that $\theta_{\beta} \in [0,2\log \epsilon_D)$ and consequently $N((\beta)) \asymp M^2+N^2$, so $|\beta| \lVert \gamma_j \rVert_1 \ll N((\beta))$.
In \eqref{eq:keysum} we apply the estimates \eqref{eq:evaluated} along with \eqref{eq:offest1} and \eqref{eq:offest2}. In the resulting formula we then apply Lemma \ref{lem:diagterms} to evaluate $Q(b_j,a_j)$, which completes the proof of Propositions \ref{prop:moment1} and \ref{prop:moment2}. Finally, recall that Proposition \ref{prop:moment} follows from Proposition \ref{prop:moment2} (see Remark \ref{remark:prop}).\qed

\subsection{Proof of Corollary \ref{cor:weighted1moment}}

We first give a short Dirichlet polynomial approximation to $\frac{1}{L(1,\phi_{2k})^2}$.
Coleman has established a zero free region for Hecke $L$-functions which is analogous to Vinogradov--Korobov's result for the Riemann zeta-function. More precisely, \cite[Theorem 2]{coleman} shows that $L(s, \phi_{2k}) \neq 0$ for
\begin{equation} \label{eq:zerofree}
\tmop{Re}(s) \ge 1- \frac{c}{(\log (k+|\tmop{Im}(s)|+1))^{2/3} (\log \log (k+|\tmop{Im}(s)|+1))^{1/3}}
\end{equation}
where $c>0$ is an absolute constant. Consequently,
in this region
\begin{equation} \label{eq:l1bd}
\frac{1}{L(s,\phi_{2k})} \ll (\log (k+|\tmop{Im}(s)|))^{2/3} (\log \log (k+|\tmop{Im}(s)|))^{1/3}
\end{equation}
(see \cite[Lemma 11]{huang2019}).
These estimates allow us to
quickly derive a short Dirichlet polynomial approximation of $1/L(1,\phi_{2k})^2$, by a contour integration argument (see \cite[Lemma 3]{huang2018quantum} for a similar result).

\begin{lemma} \label{lem:zerofree}
For $k \ge 3$ and $x \ge \exp((\log k)^{2/3}(\log \log k)^2)$ we have that
\[
\frac{1}{L(1, \phi_{2k})^2}=\sum_{n \le x} \frac{(\mu_{2k} \ast \mu_{2k})(n)}{n}+O\bigg( \exp\bigg(\frac{-c_0\log x}{\log(k+x)^{2/3} \log \log (k+x)^{1/3}} \bigg)  \bigg),
\]
where $c_0>0$ is an absolute constant and $\mu_{2k}$ is the multiplicative function with
\begin{equation} \label{eq:mudef}
\begin{split}
\mu_{2k}(p)=&-\lambda_{2k}(p) \\
\mu_{2k}(p^2)=&\chi_D(p)
\end{split}
\end{equation}
and $\mu_{2k}(p^j)=0$ if $j \ge 3$.
\end{lemma}

\begin{remark}
For $n=r^2s$ where $s$ is square-free it follows that
\begin{equation} \label{eq:muformula}
\mu_{2k}(n)= \chi_D(r)\mu^2(r) \mu(s) \lambda_{2k}(s) 1_{(r,s)=1}.
\end{equation}
\end{remark}

\begin{proof}
Using Perron's formula, it follows that for $x>0$ with $x \notin \mathbb Z$
\[
\sum_{n \le x} \frac{(\mu_{2k} \ast \mu_{2k})(n)}{n} =\frac{1}{2\pi i} \int_{1-iT}^{1+iT} \frac{1}{L(s+1,\phi_{2k})^2} \frac{x^s}{s} \,ds +O\left(\frac{x^2}{T} \right).
\]
where $T>3$.
Let $\delta=c/(\log(k+T)^{2/3} \log \log (k+T)^{1/3})$.
We now shift the contour of integration to the linear path $-\delta-iT$ to $-\delta+iT$, which is justified since $1/L(1+s,\phi_{2k})^2$ is analytic in this region by \eqref{eq:zerofree}. Collecting a simple pole at $s=0$ with residue $1/L(1,\phi_{2k})^2$ and using \eqref{eq:l1bd} to bound the horizontal contours as $O( (\log (k+T))^2 x/T)$ it follows that
\[
\sum_{n \le x} \frac{(\mu_{2k} \ast \mu_{2k})(n)}{n}=\frac{1}{L(1,\phi_{2k})^2} +\frac{1}{2\pi i} \int_{-\delta-iT}^{-\delta+iT} \frac{1}{L(1+s,\phi_{2k})^2} \frac{x^s}{s} \, ds+O\left(\frac{x^2 (\log(k+T))^2}{T} \right).
\]
Applying \eqref{eq:l1bd} the integral above is
\[
\ll (\log (k+T))^{3} \exp\left( \frac{-c \log x}{(\log(k+T)^{2/3} \log \log (k+T)^{1/3})}\right).
\]
Choosing $T=k+x^3$ completes the proof for $x \notin \mathbb Z$. If $x \in \mathbb Z$ the $n=x$ term in the sum is absorbed by the error term, so the result holds in this case as well.
\end{proof}

\begin{proof}[Proof of Corollary \ref{cor:weighted1moment}]
Applying Lemma \ref{lem:zerofree}  with $x=\exp((\log K)^{3/4})$ we have for $K < k \le 2K$ that
\begin{equation} \label{eq:approxl1}
\frac{1}{L(1,\phi_{2k})^2}= \sum_{r \le x} \frac{(\mu_{2k}\ast\mu_{2k})(r)}{r}+O(e^{-(\log K)^{1/15}}),
\end{equation}
where $\mu_{2k}(r)$ is as defined in \eqref{eq:mudef}. Note that $\mu_{2k}(r)$ is multiplicative, $D=p_1p_2$ and for $e|D$, $(\mu_{2k}\ast \mu_{2k})(e)=\sum_{ab=e}\mu(a)\mu(b)=:(\mu \ast \mu)(e)$. By these observations and noting that $D$ is squarefree
\begin{equation} \label{eq:decomp}
\sum_{r \le x} \frac{(\mu_{2k}\ast\mu_{2k})(r)}{r}=\sum_{e|D} \frac{(\mu \ast \mu)(e)}{e} \sum_{\substack{s \le x/e \\ (s,D)=1 }} \frac{(\mu_{2k}\ast\mu_{2k})(s)}{s}.
\end{equation}

Let $$g(n)=\sum_{\substack{ab=n \\ a=r_1^2 s_1, b=r_2^2s_2 \\ (r_1,s_1)=(r_2,s_2)=1}} \chi_D(r_1r_2) \mu^2(r_1)\mu^2(r_2) \mu(s_1)\mu(s_2) \sum_{d|(s_1,s_2)} h\left(\frac{s_1s_2}{d^2}\right) \chi_D(d).$$ Note that both $g$ is a multiplicative function and that
\begin{align} \label{eq:gbd}
g(p)=&-2h(p),  &g(p^2)=&3 \chi_D(p)+h(p^2), \\ g(p^3)=&-2 \chi_D(p) h(p),   &g(p^4)=&\chi_D(p)^2 , \notag
\end{align}
$g(p^j)=0$ for $j \ge 5$. Also for $(p,D)=1$ we have $h(p)=r_D(p)\frac{\vartheta(p)}{\sqrt{p}}$.
Using the Hecke relation
\[
\lambda_{2k}(a)\lambda_{2k}(b)=\sum_{d|(a,b)} \lambda_{2k}(ab/d^2) \chi_D(d)
\]
it follows from Proposition \ref{prop:moment} along with \eqref{eq:muformula} that
\begin{equation} \label{eq:weight1}
\begin{split}
&\sum_{k \in \mathbb Z}
\sum_{\substack{s \le x/e \\ (s,D)=1 }} \frac{(\mu_{2k}\ast\mu_{2k})(s)}{s}  \cdot L(\tfrac12, \psi \times \phi_{2k}) \, \phi\left( \frac{k}{K}  \right)\\
& =
\sum_{\substack{s \le x/e \\ (s,D)=1}}\frac{1}{s} \sum_{\substack{ ab=s \\a=r_1^2 s_1, b=r_2^2s_2  \\ (r_1,s_1)=(r_2,s_2)=1}} \chi_D(r_1r_2) \mu^2(r_1)\mu^2(r_2) \mu(s_1)\mu(s_2) \sum_{d|(s_1,s_2)} \chi_D(d)
\\
& \hskip 90pt \cdot \sum_{k \in \mathbb Z} \lambda_{2k}\left(\frac{s_1s_2}{d^2}\right)  L(\tfrac12, \psi \times \phi_{2k}) \phi\left( \frac{k}{K}  \right)\\
&  = \tilde{\phi}(1) \cdot C_{D,\psi} \cdot K \cdot  \sum_{\substack{s \le x/e \\ (s,D)=1 }} \frac{g(s)}{s}+O\left(P^{A_0}K^{\frac12+\vartheta+\varepsilon} \right).
\end{split}
\end{equation}
Using \eqref{eq:gbd}, the sum on the RHS equals
\begin{equation} \label{eq:weight2}
\prod_{(p,D)=1}\left(1-\frac{2\vartheta(p)r_D(p)}{p^{3/2}}+\frac{3\chi_D(p)+h(p^2)}{p^3}+\frac{2\vartheta(p)r_D(p) \chi_D(p)}{p^{5/2}}+\frac{1}{p^5} \right)+O(x^{-1/4}).
\end{equation}
Combine \eqref{eq:approxl1}, \eqref{eq:decomp}, \eqref{eq:weight1} and \eqref{eq:weight2} and use Proposition \ref{prop:moment} to estimate the error term which arises from applying \eqref{eq:approxl1} to complete the proof (we know that $L(\tfrac12, \psi \times \phi_{2k}) \ge 0$ by Lapid \cite{Lapid}).
\end{proof}

\section{Estimate of the quantum variance}\label{sec:variance}
In this section, we prove Theorem \ref{thm:QV} using the first moment computed in Section \ref{sec:twisted1moment}
along with the Watson--Ichino formula, which we will state below.

\subsection{The \textit{L}-functions} Let $\Xi_k$ be a Hecke Gr\"{o}ssencharacter $\tmop{mod} 1$ of frequency $k\geq1$.
Define its Hecke $L$-function
\begin{equation}\label{eqn:L-DS}
  L(s,\phi_k) := \sum_{0\neq\mathfrak{a}\subset \mathcal{O}} \frac{\Xi_k(\mathfrak{a})}{\N(\mathfrak{a})^s} = \prod_{\mathfrak{p}} \left(1-\frac{\Xi_k(\mathfrak{p})}{\N(\mathfrak{p})^s}\right)^{-1},
  \quad \textrm{$\Re(s)>1$}.
\end{equation}
We have
\[
  L(s,\phi_k) = \sum_{n=1}^{\infty} \frac{\lambda_k(n)}{n^s},
  \quad \mathrm{with}\
  \lambda_k(n)=\sum_{\N(\mathfrak{a})=n} \Xi_k(\mathfrak{a}),
  \quad \textrm{for $\Re(s)>1$}.
\]
It is easy to see that $\lambda_k(n)$ are real.
The complete $L$-function is defined by
\[
  \Lambda(s,\phi_k)=D^{s/2}\gamma(s,\phi_k)L(s,\phi_k)
  \quad \textrm{ with } \quad
  \gamma(s,\phi_k) = \pi^{-s} \Gamma\left(\frac{s+it_k}{2}\right)\Gamma\left(\frac{s-it_k}{2}\right).
\]

Let $D>0$ be an odd fundamental discriminant and $D_1|D$.
We now turn to the $L$-functions for $\psi$, a Hecke--Maass cuspidal newform of level $D_1$. Assume $\psi$ is even. We have that
\begin{align*}
  \Lambda(s,\psi) & = D_1^{s/2} \gamma(s,\psi) L(s,\psi), \\
  \Lambda(s,\psi\times \chi_D) & = D^s \gamma(s,\psi\times \chi_D) L(s,\psi\times \chi_D), \\
  \Lambda(s,\operatorname{sym}^2 \psi ) & = D_1^s \gamma(s,\operatorname{sym}^2 \psi ) L(s,\operatorname{sym}^2 \psi ),
\end{align*}
where
\begin{align*}
  \gamma(s,\psi) &=
  \gamma(s,\psi\times \chi_D) = \pi^{-s} \Gamma\left(\frac{s+it_\psi}{2}\right)\Gamma\left(\frac{s-it_\psi}{2}\right) , \\
  \gamma(s,\operatorname{sym}^2 \psi ) & = \pi^{-3s/2} \Gamma\left(\frac{s+2it_\psi}{2}\right)\Gamma\left(\frac{s}{2}\right)\Gamma\left(\frac{s-2it_\psi}{2}\right).
\end{align*}
Recall that $\chi_D$ is an even real primitive Dirichlet character modulo $D$. We have
\[
  \Lambda(s,\chi_D) = D^{s/2} \gamma(s,\chi_D) L(s,\chi_D),
  \quad \textrm{ with } \quad
  \gamma(s,\chi_D) = \pi^{-s/2} \Gamma(s/2).
\]
We also need Rankin--Selberg $L$-functions.
By \cite[\S A.2 \& \S A.3]{Humphries-Khan}, we know the conductor and root number of $\psi\times \phi_{2k}$ are given by
\[
  \mathfrak{q}(\psi\times \phi_{2k}) = D^2 D_1, \quad
  \epsilon(\psi\times \phi_{2k}) = \eta_\psi(D_1),
\]
and the $L$-function of $\psi\times \phi_{2k}$ is given by
\[
  L(s,\psi\times \phi_{2k}) = L(2s,\chi_D) \sum_{n=1}^{\infty} \frac{\lambda_\psi(n) \lambda_{2k}(n)}{n^s}.
\]
Hence, the functional equation is
\begin{equation} \label{eq:functional-eqn}
  \Lambda(s,\psi\times \phi_{2k}) = (D^2D_1)^{s/2} \gamma(s,\psi\times \phi_{2k}) L(s,\psi\times \phi_{2k})
  = \eta_\psi(D_1) \Lambda(1-s,\psi\times \phi_{2k}),
\end{equation}
where for even $\psi$, we have
\begin{equation} \label{eq:gamma-factor}
  \gamma(s,\psi\times \phi_{2k}) = \pi^{-2s}
  \prod_{\pm_1}\prod_{\pm_2} \Gamma\left(\frac{s \pm_1 it_\psi \pm_2 i t_{2k}}{2} \right).
\end{equation}

\subsection{The Watson--Ichino formula}

By comparing Euler products, it is easy to see the following factorizations for $L$-functions of dihedral Maass forms, that is,
\[
  \Lambda(s,\operatorname{ad} \phi_k) = \Lambda(s,\chi_D) \Lambda(s,\phi_{2k}),
\]
and
\[
  \Lambda(s,\psi\times \operatorname{ad} \phi_k)  =  \Lambda(s,\psi\times \chi_D) \Lambda(s,\psi\times \phi_{2k}).
\]
Combining the above observation along with the Watson--Ichino formula due to Humprhies-Khan \cite[Proposition 1.16]{Humphries-Khan} we get the following result.
\begin{lemma}\label{lemma:WI}
  Let $D\equiv 1 \pmod {4}$ be a positive squarefree fundamental discriminant and let $\chi_D$ be the primitive quadratic character modulo $D$. Let $D_1\mid D$. Let $\mu_k(\cdot)$ be as in \eqref{eq:measdef}.
  Then for $\psi$ an $L^2$-normalized, even Hecke--Maass cuspidal newform of level $D_1$ with trivial nebentypus we have that
  \[
    |\mu_k(\psi)|^2 = \frac{1}{8 \sqrt{D}\nu(D/D_1)} \frac{\Lambda(\tfrac12,\psi)\Lambda(\tfrac12,\psi\times \chi_D) \Lambda(\tfrac12,\psi\times \phi_{2k})}{\Lambda(1,\operatorname{sym}^2 \psi )\Lambda(1,\chi_D)^2 \Lambda(1,\phi_{2k})^2 },
  \]
  where $\nu(n) = n\prod_{p\mid n}(1+p^{-1})$. If $\psi$ is odd, then $|\mu_k(\psi)|^2=0$.
\end{lemma}

\subsection{Proof of Theorem \ref{thm:QV}}

Let $\psi$ be an even Hecke--Maass cuspidal newform of level $D$ and trivial nebentypus.
By Lemma \ref{lemma:WI}, we have that
\begin{multline*}
  Q^{\mathrm{h}}(\psi,\psi;K;\Phi)
  =   \frac{1}{8 \sqrt{D}}  \frac{\Lambda(\tfrac12,\psi)\Lambda(\tfrac12,\psi\times \chi_D) }{\Lambda(1,\operatorname{sym}^2 \psi )\Lambda(1,\chi_D)^2}
  \sum_{k\in \mathbb{Z}} L(1,\phi_{2k})^2  \frac{ \Lambda(\tfrac12,\psi\times \phi_{2k})}{ \Lambda(1,\phi_{2k})^2 } \Phi\left(\frac{k}{K}\right) \\
   =   \frac{\pi^2 }{4 D^{2}}  \frac{|\Gamma(\tfrac14+it_\psi/2)|^4} {2\pi |\Gamma(\tfrac12+it_\psi)|^2}
   \frac{L(\tfrac12,\psi) L(\tfrac12,\psi\times \chi_D)}
   {L(1,\operatorname{sym}^2 \psi ) L(1,\chi_D)^2}
   \\  \cdot
  \sum_{k\in \mathbb{Z}}
   \frac{|\Gamma\left(\frac{\frac12 + it_\psi + i t_{2k}}{2} \right)|^2 |\Gamma\left(\frac{\frac12 - it_\psi + i t_{2k}}{2} \right)|^2} { |\Gamma(\frac{1+it_{2k}}{2})|^4 }
  L(\tfrac12,\psi\times \phi_{2k}) \Phi\left(\frac{k}{K}\right).
\end{multline*}
By Stirling's formula, we get that
\begin{equation*}
  Q^{\mathrm{h}}(\psi,\psi;K;\Phi)
  =   \frac{\pi \log \epsilon_D}{4 D^{2}}   V(\psi)
   \frac{L(\tfrac12,\psi) L(\tfrac12,\psi\times \chi_D)}
   {L(1,\operatorname{sym}^2 \psi ) L(1,\chi_D)^2}
  \sum_{k\in \mathbb{Z}}
  L(1/2,\psi\times \phi_{2k}) \frac{1}{k}\Phi\left(\frac{k}{K}\right) + O(1/K).
\end{equation*}
By Proposition \ref{prop:moment1} with $\phi(y)=\Phi(y)/y$, we prove the first claim in Theorem \ref{thm:QV}. The proof of the second claim in Theorem \ref{thm:QV} follows by a similar argument, only now we assume GRC so that we may apply Corollary \ref{cor:weighted1moment}.  \qed

\section{Bound for the covariance} \label{sec:covariance}

The purpose of this section is to show that $Q(\psi_1,\psi_2;K;\Phi)\rightarrow 0$ as $K \rightarrow \infty$ conditionally under  GRH, thereby proving Theorem \ref{thm:cov}. This follows from Lemma \ref{lemma:WI} together with the following result and noting that under GRH we have $L(1,\phi_{2k})^{-1} \ll \log \log k$.

\begin{proposition} \label{prop:moments}
Assume GRH. Let $n \ge 1$. Also, let $\psi_1, \ldots, \psi_n$ be pairwise orthogonal Hecke--Maass cuspidal newforms on $\Gamma_0(D)$ with trivial nebentypus. Then for any real numbers $\ell_1, \cdots, \ell_n >0$ we have that
\[
\sum_{K < k \le 2K} L(\tfrac12, \psi_1 \times \phi_{2k}
)^{\ell_1} \cdots L(\tfrac12, \psi_n \times \phi_{2k})^{\ell_n} \ll K \cdot (\log K)^{\frac{\ell_1(\ell_1-1)}{2}+\cdots +\frac{\ell_n(\ell_n-1)}{2}+\varepsilon}.
\]
\end{proposition}

\begin{remark} \label{rem:expectedvalue}
  Assume GRH. We have $1/L(1,\phi_{2k})\ll \log\log k$. By \eqref{eqn:EV}, Lemma \ref{lemma:WI} and Proposition \ref{prop:moments}, we get
  \begin{align*}
    \mathbb{E}(\psi;K) & \ll  \frac{1}{K} \sum_{k\asymp K}  |\mu_k(\psi)|  \ll \frac{\log\log K}{K^{3/2}} \sum_{k\asymp K} L(1/2,\psi\times \phi_{2k})^{1/2}
    \ll K^{-1/2} (\log K)^{-1/8+\varepsilon}.
  \end{align*}
\end{remark}

\begin{remark} In the proof of Proposition \ref{prop:moments} we assume that GRH holds for the $L$-functions $L(s,\psi_j \times \phi_{2k})$ and $L(s,\tmop{sym}^2 \psi_j\times \phi_{2k})$ for all $j=1,\ldots,n$ and $K < k \le 2K$.
\end{remark}

Let $\alpha_{\psi},\beta_{\psi}$ and $\alpha_{2k}, \beta_{2k}$ denote the Satake parameters for $\psi$ and $\phi_{2k}$, respectively. We have that
\begin{itemize}
    \item $\alpha_{2k}(p)=\Xi_{2k}(\mathfrak p)$, $\beta_{2k}(p)=\overline{\alpha_{2k}(p)}$,  if $\chi_D(p)=1$ where $\mathfrak p$ is a prime in $\mathcal O_{\mathbb Q(\sqrt{D})}$ which lies above $p$;
    \item $\alpha_{2k}(p)=1$, $\beta_{2k}(p)=-1$, if $\chi_D(p)=-1$;
    \item  $\alpha_{2k}(p)=\Xi_{2k}(\mathfrak p)$, $\beta_{2k}(p)=0$, if $p|D$;
\end{itemize}
(see Appendix A of \cite{Humphries-Khan}).
Additionally, for $(p, D)=1$, let
\[
\Lambda_{\psi \times \phi_{2k}}(p^n)=(\alpha_{\psi}(p)^n+\beta_{\psi}(p)^n)(\alpha_{2k}(p)^n +\beta_{2k}(p)^n).
\]
In particular, we have for $(p,D)=1$ that
\begin{equation} \label{eq:squares}
\Lambda_{\psi \times \phi_{2k}}(p^2)=(\lambda_{\psi}(p^2)-1)(\lambda_{4k}(p)+1-\chi_D(p)).
\end{equation}

\begin{lemma} \label{lem:chandee}
Assume GRH. Let $\psi$ be a Hecke--Maass cuspidal newform on $\Gamma_0(D)$ with trivial nebentypus. Then for $x>10$
\[
\log
L(\tfrac12, \psi \times \phi_{2k}) \le \sum_{\substack{p^n \le x \\ p \nmid D}} \frac{\Lambda_{\psi\times \phi_{2k}}(p^n)}{np^{n(\frac12+\frac{1}{\log x})}}  \frac{\log \frac{x}{p^n}}{\log x}+O\left( \frac{\log K}{\log x} +1\right)
\]
where the implied constant depends on $D$ and $\psi$.
\end{lemma}
\begin{proof}
This follows from \cite[Theorem 2.1]{chandee}.
\end{proof}

\begin{lemma} \label{lem:somesums}
 Assume GRH. For $x \ge 2$ we have
 \[
 \sum_{p \le x} \frac{\lambda_{\psi}(p^2) \lambda_{4k}(p)}{p}=O(\log \log \log k)
 \]
 and
 \[
 \sum_{p \le x} \frac{\lambda_{4k}(p)}{p}=O(\log \log \log k).
 \]
\end{lemma}
\begin{proof}
We will only establish the first bound, since the second follows from a similar yet simpler argument. Note that $L(s, \tmop{sym}^2 \psi \times \phi_{4k})$ has an analytic continuation  to the complex plain and satisfies a functional equation.  Assuming GRH for $L(s, \tmop{sym}^2 \psi \times \phi_{4k})$, it follows that $\log L(s, \tmop{sym}^2 \psi \times \phi_{4k}) $ is analytic in the region $\tmop{Re}(s) \ge \frac12+\frac{1}{\log x}$. Moreover, by repeating a classical argument of Littlewood (see Titchmarsh \cite[(14.2.2)]{Titchmarsh}) we have the bound $ |\log L(s, \tmop{sym}^2 \psi \times \phi_{4k})| \ll (\tmop{Re}(s)-\frac12)^{-1}\log (k+|\tmop{Im}(s)|)$ in this region. By Perron's formula, we have for $x \ge 2$ that
\[
\begin{split}
 \sum_{p \le x} \frac{\lambda_{\psi}(p^2) \lambda_{4k}(p)}{p} =&\frac{1}{2\pi i} \int_{1-ix\log (k+x)}^{1+i x \log (k+x)} \log L(s+1, \tmop{sym}^2 \psi \times \phi_{4k}) \,  x^s \frac{\dd s}{s} +O(1).
 \end{split}
\]
Shifting contours to $\tmop{Re}(s)=-1/2+1/\log x$ we collect a simple pole at $s=0$ with residue $\log L(1,\tmop{sym}^2 \psi \times \phi_{4k})$. The upper horizontal contour is bounded by
\[
\begin{split}
 \ll & \frac{1}{x \log (k+x)} \int_{-\frac12+\frac{1}{\log x}+ix\log(k+x)}^{1+ix\log(k+x)} |\log L(1,\tmop{sym}^2 \psi \times \phi_{4k})| |x^s| |\dd s| \\
 \ll & \frac{\log x \log (k+x\log(k+x))}{x \log (k+x)} \int_{-1/2}^1 x^{u} \, \dd u \ll 1
 \end{split}
\]
and the lower horizontal contour is also $O(1)$. Hence we have for $x \ge 2$ that
\[
 \sum_{p \le x} \frac{\lambda_{\psi}(p^2) \lambda_{4k}(p)}{p}=\log L(1,\tmop{sym}^2 \psi \times \phi_{4k})+O\left(\frac{\log x}{x^{1/2}} \int_{-x \log(k+x)}^{x\log(k+x)} \frac{\log(k+u)}{1+|u|} \, \dd u \right).
\]
Applying the above estimate twice we have for $z \ge (\log k)^3$  that
\begin{equation} \label{eq:symsquarebd1}
\bigg|\sum_{(\log k)^3 < p \le z} \frac{\lambda_{\psi}(p^2) \lambda_{4k}(p) }{p} \bigg|\ll 1.
\end{equation}
Next, using the bound $|\lambda_{4k}(p)| \le 2$, along with the elementary inequality $2|ab| \le |a|^2+|b|^2$, we have for $y \le (\log k)^3$ that
\begin{equation} \label{eq:symbsqaurebd2}
\bigg|\sum_{p \le y } \frac{\lambda_{\psi}(p^2) \lambda_{4k}(p)}{p}\bigg| \ll \log \log \log k+\sum_{p \le y } \frac{\lambda_{\psi}(p^2)^2}{p}.
\end{equation}
Using the Hecke relations, $\lambda_{\psi}(p^2)^2=1+\lambda_{\psi}(p^2)+\lambda_{\psi}(p^4)$. Also,
\begin{equation} \label{eq:symbsqaurebd3}
\bigg|
\sum_{p \le y} \frac{\lambda_{\psi}(p^2)}{p} \bigg| \ll 1 \qquad \text{and} \qquad \bigg|\sum_{p \le y} \frac{\lambda_{\psi}(p^4)}{p}\bigg| \ll 1 ,
\end{equation}
(see Blomer et. al. \cite[Section 2.3.4-2.4]{BFKMMS}). Combining \eqref{eq:symsquarebd1},\eqref{eq:symbsqaurebd2} and \eqref{eq:symbsqaurebd3} completes the proof.
\end{proof}

 \begin{lemma} \label{lem:moments}
 Let $r\in \mathbb N$. Then for $x \le K^{1/(10 r)}$ and real numbers $a_p$ with $a_p \ll p^{1/2-\delta}$ for some $\delta>0$, we have that
 \[
 \frac1K \sum_{K < k \le 2K} \Bigg( \sum_{\substack{p \le x \\ p \nmid D}} \frac{a_p \lambda_{2k}(p)}{p^{1/2}}  \Bigg)^{2r} \ll \frac{(2 r)!}{2^r r!} \Bigg( 2 \sum_{\substack{p \le x \\ \chi_D(p)=1}}\frac{a_p^2}{p} \Bigg)^r.
 \]
 \end{lemma}

 \begin{proof}
 For $p$ with $\chi_D(p)=1$ we have
 \[
 \lambda_{2k}(p)=\Xi_{2k}(\mathfrak p) + \overline{\Xi_{2k}(\mathfrak p)}.
 \]
 Also let $a_n=\prod_{p^j || n} a_p^j$. Note $\lambda_{2k}(p)=0$ if $\chi_D(p)=-1$.
 Let $F$ be a Schwartz function which majorizes $1_{[1,2]}$ with $\widehat F$ compactly supported.
 We have
 \begin{equation} \label{eq:expand2}
 \begin{split}
 &\sum_{k \in \mathbb Z}
 \left( \sum_{p \le x} \frac{a_p \lambda_{2k}(p)}{p^{1/2}}  \right)^{2r} F\left(\frac{k}{K} \right)\\
 &=\sum_{n} \frac{a_n}{n^{1/2}} \sum_{\substack{\mathfrak p_1, \ldots, \mathfrak p_r \subset \mathcal O_{\mathbb Q(\sqrt{D})} \\ N ( \mathfrak p_1 \cdots \mathfrak p_{2r})=n \\ N ( \mathfrak p_j) \le x, j=1, \ldots, 2r \\ \chi_{D}(N \, \mathfrak p_j)=1, j=1, \ldots, 2r } }  \sum_{k \in \mathbb Z} \prod_{j=1}^{2r} (\Xi_{2k}(\mathfrak p_j)+\Xi_{2k}(\widetilde {\mathfrak p_j})) F\left(\frac{k}{K} \right),
 \end{split}
 \end{equation}
 where the summation is over prime ideals.
Next, write $n=q_1^{e_1} \cdots q_s^{e_s}$ with $N (\mathfrak q_j)=q_j$ and $q_1,\ldots,q_s$ are distinct primes. The innermost sum equals
 \begin{equation} \label{eq:harmonics}
 \sum_{0 \le f_1 \le e_1} \cdots  \sum_{0 \le f_s \le e_s} \binom{e_1}{f_1} \cdots \binom{e_s}{f_s} \sum_{k \in \mathbb Z}  \Xi_{2k}( \mathfrak q_1^{f_1} \widetilde{\mathfrak q_1}^{e_1-f_1} \cdots \mathfrak q_s^{f_s}  \widetilde{\mathfrak q_s}^{e_s-f_s} )  F\left(\frac{k}{K} \right).
\end{equation}
Let $(\alpha)=\mathfrak q_1^{f_1} \widetilde{\mathfrak q_1}^{e_1-f_1} \cdots \mathfrak q_s^{f_s}  \widetilde{\mathfrak q_s}^{e_s-f_s}$. Also, let $\theta_{\alpha}=\log | \frac{\alpha}{\widetilde{\alpha}} |$ with $\theta_{\alpha} \in [0,2 \log \epsilon_D)$.
Applying Poisson summation we have that
\begin{equation} \label{eq:primepoisson}
\begin{split}
\sum_{k \in \mathbb Z} \Xi_{2k}((\alpha)) F\left(\frac{k}{K} \right)=&\sum_{k \in \mathbb Z} e\left( \frac{k \theta_{\alpha}}{\log \epsilon_D}\right) F\left(\frac{k}{K} \right) \\
=& K \sum_{\ell\in \mathbb Z} \widehat F\left( K \left( \ell-\frac{\theta_{\alpha}}{\log \epsilon_D}\right)\right).
\end{split}
\end{equation}
 Since $\widehat F$
 has compact support
 and $ N((\alpha))< K^{1/5}$ it follows that the sum above equals zero unless $\alpha=\widetilde \alpha$ or $|\frac{\alpha}{\widetilde \alpha}|=\epsilon_D$.
Hence, $\alpha=u\widetilde{\alpha}$ for some unit $u$. Therefore, since the primes $\mathfrak q_j$ are distinct the inner sum in \eqref{eq:harmonics} is zero unless $f_j=e_j/2$ for each $j=1, \ldots, s$ in which case the sum on the RHS of \eqref{eq:primepoisson} equals $K \widehat F(0)$.
  Let $g$ be the multiplicative function with $g(p^a)=0$ if $a$ is odd and $g(p^a)=\binom{a}{a/2}$ if $a$ is even, and $\nu$ be defined by $\nu(p^a)=1/a!$.
  Thus, the LHS of \eqref{eq:expand2} equals
  \[
 (2r)! K \widehat F(0) \sum_{\substack{p|n \Rightarrow p \le x \text{ and } \chi_D(p)=1 \\ \Omega(n)=2r }} \frac{a_n}{\sqrt{n}} g(n) \nu(n) \ll \frac{(2r)!}{ r!}\, K \, \Bigg( \sum_{\substack{p \le x \\ \chi_D(p)=1}} \frac{a_p^2}{p}\Bigg)^r,
  \]
  where the last bound follows since the sum on the LHS is supported on squares.
  This completes the proof.
 \end{proof}

 \begin{lemma} \label{lem:selberg}
 Let $\psi_1, \psi_2$ be two distinct Hecke--Maass cusp forms. Then for $x \ge 2$
 \[
 \sum_{p \le x} \frac{\lambda_{\psi_1}(p)\lambda_{\psi_2}(p)}{p}=O(1).
 \]
 \end{lemma}
 \begin{proof}
 This follows from  \cite[Corollary 1.5]{Liu-Wang-Ye}.
 \end{proof}
Given a Hecke--Maass cuspform $\psi$ with trivial nebentypus and primitive character $\chi$
recall that $\psi \times \chi$ is a Hecke--Maass cusp form (see the proof of Theorem 7.4 of \cite{iwaniec1997topics}) and that $\tmop{sym}^2 \psi \times \chi$ is cuspidal and has a standard zero free region (\cite[Theorem 3.3.7]{kim-shahidi}, \cite[Section 2.3.4-2.4]{BFKMMS}). Using Lemma \ref{lem:selberg} and the Hecke relations we get that for $2 \le y \le x$, $\ell_1,\ldots,\ell_n>0$ and distinct Hecke--Maass cusp forms $\psi_1, \ldots,\psi_n$ that
 \begin{equation} \label{eq:variance}
 \begin{split}
 \sum_{\substack{y< p \le x \\ \chi_{D}(p)=1}} \frac{(\ell_1 \lambda_{\psi_1}(p) + \cdots + \ell_n \lambda_{\psi_n}(p))^2}{p}=& \frac12  \sum_{\substack{y< p \le x \\ p \nmid D}} \frac{(\ell_1 \lambda_{\psi_1}(p) + \cdots + \ell_n \lambda_{\psi_n}(p))^2}{p} (\chi_D(p)+1) \\
 =& \frac{1}{2} \sum_{j=1}^n \ell_j^2 \sum_{y< p \le x} \frac{\lambda_{\psi_j}(p)^2}{p}+O(1) \\
 =& \frac{1}{2} \left(  \sum_{j=1}^n \ell_j^2 \right)\log \frac{\log x}{\log y}+O(1).
\end{split}
 \end{equation}

 Before stating the next lemma let us introduce the following notation. For $2 \le y \le x$, let
 \[
 \mathcal P(k;x,y)=\sum_{\substack{p\le y \\ p \nmid D}}
 \frac{(\ell_1 \lambda_{\psi_1}(p)+ \cdots + \ell_n \lambda_{\psi_n}(p)) \lambda_{2k}(p)}{p^{\frac12+\frac{1}{\log x}}} \frac{\log x/p}{\log x}
 \]
 and
 $\mathcal A_K(V;x)=\# \{ K <k \le 2K :
 \mathcal P(k;x,x)> V\}$. Also, define
\[
 \sigma(K)^2= (\ell_1^2+\cdots + \ell_n^2) \log \log K.
\]
 \begin{lemma} \label{lem:largedev}
 Let $C \ge 1$ be fixed and $\epsilon>0$ be sufficiently small.
 In the above notation, we have for
 $\sqrt{\log \log K} \le V \le C \log K/\log \log K$ that
 \[
 \mathcal A_K(V; K^{\frac{1}{\epsilon V}})  \ll K \left(
 e^{-\frac{V^2}{2 \sigma(K)^2}(1-2\epsilon)}  + e^{-\frac{\epsilon}{11} \cdot V \log V}\right).
 \]

 \end{lemma}

 \begin{proof}
 We assume throughout that $\sqrt{\log \log K} \le V \le C \frac{\log K}{\log \log K}$.
 Set $x=K^{\frac{1}{\epsilon V}}$ and
 let $z=x^{1/\log \log K}$. Write $\mathcal P(k;x,x)=\mathcal P_1(k)+\mathcal P_2(k)$ where $\mathcal P_1(k)=\mathcal P(k;x,z)$. Also, let $ V_1=(1-\epsilon) V$ and $V_2=\epsilon V$. If $\mathcal P(k;x,x)>V$ then
 \[
 i) \quad \mathcal P_1(k) > V_1 \qquad \text{ or }  \qquad ii) \quad \mathcal P_2(k) > V_2.
 \]
 Using Lemma \ref{lem:moments} and \eqref{eq:variance} we have for $r \le \frac{\epsilon V}{10} \log \log K$ that the number of $K < k \le 2K$ for which $i)$ holds is bounded by
 \[
 \begin{split}
 \frac{1}{V_1^{2r}} \sum_{K < k \le 2K} \mathcal P_1(k)^{2r} \ll& K \frac{(2r)!}{V_1^{2r} 2^r r!} (\sigma(K)(1+o(1))^{2r} \\
 \ll& K \left(\frac{ 2r\sigma(K)^2 (1+o(1))}{V_1^2 e} \right)^r
 \end{split}
 \]
where in the second step we applied Stirling's formula.  In the range $V \le \frac{\epsilon}{10} \sigma^2(K) \log \log K$ we set $r= \lfloor \frac{V_1^2}{2 \sigma(K)^2}\rfloor$ and for larger $V$ we set $r=\lfloor \epsilon V/10 \rfloor$. Hence,
\begin{equation} \label{eq:cheby1}
 \# \{ K < k \le 2K : \mathcal P_1(k) > V_1\} \ll K\left( e^{-\frac{V^2}{2 \sigma(K)^2}(1-2 \epsilon)}+e^{-\frac{\epsilon}{11} \cdot V \log V}\right).
\end{equation}
 It remains to bound the number of $K < k \le 2K$ for which $ii)$ holds. Take $r=\lfloor \frac{\epsilon V}{10} \rfloor$.  As before, we use Lemma \ref{lem:moments} and \eqref{eq:variance} to bound this quantity by
 \begin{equation} \label{eq:cheby2}
 \begin{split}
 \frac{1}{V_2^2} \sum_{K < k \le 2K} \mathcal P_2(k)^{2r}
 \ll& K \frac{(2r)!}{2^r r!} \left( \frac{C}{V_2^2} \log \log \log K  \right)^r \\
 \ll& K \left(\frac{C'}{V_2^2} r \log \log \log K\right)^r  \ll e^{-\frac{\epsilon}{11}\cdot V \log V},
 \end{split}
 \end{equation}
 where $C,C'$ are constants that depend at most on $\epsilon$, $\ell_1,\ldots, \ell_n$, $\psi_1, \ldots, \psi_n$ and $D$. Combining \eqref{eq:cheby1} and \eqref{eq:cheby2} completes the proof.

 \end{proof}

 \begin{proof}[Proof of Proposition \ref{prop:moments}]
Using \eqref{eq:squares} and bounding the sum over $p^n \le x$ with $n \ge 3$ we get that
 \[
 \begin{split}
 \sum_{\substack{p^n \le x \\ p \nmid D}} \frac{\Lambda_{\psi\times \phi_{2k}}(p^n)}{np^{n(\frac12+\frac{1}{\log x})}}  \frac{\log \frac{x}{p^n}}{\log x}=&\sum_{\substack{p \le x \\ p \nmid D}} \frac{\lambda_{\psi}(p)\lambda_{2k}(p)}{p^{\frac12+\frac{1}{\log x}}}  \frac{\log \frac{x}{p}}{\log x}\\
 &+ \frac12 \sum_{\substack{p\le \sqrt{x} \\ p \nmid D}} \frac{(\lambda_{\psi}(p^2)-1)(\lambda_{4k}(p)+1-\chi_D(p))}{p^{1+\frac{2}{\log x}}}  \frac{\log \frac{x}{p^2}}{\log x}+O(1),
 \end{split}
 \]
Hence, using Lemma \ref{lem:somesums}
the second sum above equals
\begin{equation}\label{eq:square}
\begin{split}
-\frac12\log \log x+O(\log \log \log k).
\end{split}
\end{equation}

Let $\mu(K)=(-\frac12+\epsilon)(\ell_1+\cdots+\ell_n) \log \log K$. Also, define
\[
\mathcal L(K)=L(\tfrac12, \psi_1 \times \phi_{2k}
)^{\ell_1} \cdots L(\tfrac12, \psi_n \times \phi_{2k})^{\ell_n}\]
and $\mathcal B_K(V)=\# \{ K <k \le 2K :  \log \mathcal L(k)> V\}$.
 Clearly,
\[
\sum_{K < k \le 2K} \mathcal L(k) =- \int_{\mathbb R} e^V  \, \dd \mathcal B_K(V)=\int_{\mathbb R} e^V  \mathcal B_K(V) \, \dd V=e^{\mu(K)} \int_{\mathbb R} e^V \mathcal B_K(V+\mu(K)) \, \dd V.
\]
Note that $\log \mathcal L(k) \le C \log K/\log \log K$ for some $C>1$, which follows from using Lemma \ref{lem:chandee}. So we only need to consider $\sqrt{\log \log K} \le V \le C \frac{\log K}{\log \log K}$ in the integral above (for smaller $V$ use the trivial bound $\mathcal B_K(V) \le K$).
Let $x=K^{1/(\epsilon V)}$. For $\sqrt{\log \log K} \le V \le (\log \log K)^4$
\[
-\frac12 (\ell_1+\cdots+\ell_n) \log \log x+O(\log \log \log k) \le \mu(K).
\]
It follows from Lemma \ref{lem:chandee} and \eqref{eq:square} that
\[
\mathcal B_K(V+\mu(K)) \le \mathcal A_K(V(1-2\epsilon);x)
\]
provided $\sqrt{\log \log K} \le V \le (\log \log K)^4$. For $V> (\log \log K)^4$ the above inequality is also true since in this range $V+\mu(K)=V(1+o(1))$.  Hence, combining estimates and applying Lemma \ref{lem:largedev} there exists an absolute constant $C>0$ so that
\[
\begin{split}
\sum_{K < k \le 2K} \mathcal L(k) \ll & K e^{\mu(K)} \int_{\sqrt{\log \log K}}^{C\frac{\log K}{\log \log K}} e^V \left(  e^{-\frac{V^2}{2 \sigma(K)^2}(1-  \varepsilon)}  + e^{- \varepsilon \cdot V \log V} \right) \, \dd V \\
\ll &K (\log K)^{\varepsilon}  e^{\mu(K)+\frac{\sigma^2(K)}{2}} \ll K (\log K)^{\frac{\ell_1(\ell_1-1)}{2}+\cdots +\frac{\ell_n(\ell_n-1)}{2}+\varepsilon},
\end{split}
\]
where in the last step we used the identity
\[
\int_{\mathbb R} e^{-\frac{x^2}{2\sigma^2}+ x} \, \dd x=\sqrt{2 \pi} \,  \sigma \, e^{\sigma^2/2}.
\]
This completes the proof.
 \end{proof}

\appendix

\section{The triple product estimate}\label{appendix:Sarnak}

\subsection{Introdution}
For a Dirichlet character $\chi$ mod $r$ and an integer $t\geq1$ we have $\theta_{\chi,t}(z)=y^{1/4+\nu/2}\sum_{n\in\mathbb{Z}} \chi(n) n^\nu e(n^2 z)\in \mathbf{H}_{\kappa}( 4r^2 t ,\chi_\nu)$ with $\kappa=1/2+\nu$, $\chi_\nu(n)=\chi(n) \left(\frac{-1}{n}\right)^\nu$, and $\nu=0,1$ depending on $\chi(-1)=(-1)^\nu$.
Recall that $\Psi \in L^2_{\mathrm{cusp}}(\Gamma_{0}(4aN)\backslash \mathbb{H})$ as in \S\ref{subsec:wt0}.
We take $M=\lcm[4aN,4r^2t]$ and throughout this appendix we write $\Gamma=\Gamma_0(M)$.
Let $\{f_j\}$ be a complete orthonormal system of $\textbf{H}_{\kappa}(M,\chi_\nu)$ where each $f_{j}$ is an eigenfunction of $\Delta_{\kappa}$ with eigenvalue $\lambda_j = \frac{1}{4} + t_j^2$.
Both $\overline{\Psi} \theta_{\chi,t}$ and  $f_j$ are in $\mathbf{H}_{\kappa} (M,\chi_\nu)$. In this appendix, we will follow \cite{sarnak1994integrals} (see also \cite[Appendix 2]{biro2011relation}) to estimate the triple product $\langle f_j , \overline{\Psi} \theta_{\chi,t} \rangle$ in terms of $t_j$ as $t_j\rightarrow\infty$.
\begin{lemma}\label{lemma:triple}
  With the notation as above.
  Then there are constants $A,C$ such that
  \[
    \langle f_j , \overline{\Psi} \theta_{\chi,t} \rangle
    \ll M^A (1+|t_j|)^C e^{-\frac{\pi}{2} |t_j|},
  \]
  where the implied constant depends on $\psi$.
\end{lemma}

\begin{remark}
  We mainly use ideas from \cite{sarnak1994integrals}. However, in our case we need to bound $\theta_{\chi,t}(z)$ and cannot use the $L^\infty$-norm bound as in the proof of \cite[Lemma 1]{sarnak1994integrals}, since $\theta_{\chi,t}(z)$ may not be a cusp form. To get around this issue, we estimate the $L^2$-norm instead. Our method can also be used to give another proof of  \cite[Lemma A2]{biro2011relation}.
\end{remark}

\begin{remark}
  Together with the sup-norm bounds for $\psi$, we can prove the dependence on $\psi$ is polynomially on $t_\psi$ and $N$.
\end{remark}

 Before commencing with the proof let us introduce some further notation and record some well-known facts. For reference see Iwaniec \cite{iwaniec2002spectral} and Roelcke \cite{Roelcke}.
 Let
  \[
    h(\tau) = e^{-(\tau-T)^2}+ e^{-(\tau+T)^2},
  \]
  with a fixed large real $T>0$.
  Set
  \begin{equation}\label{eqn:g}
    g(r) = \frac{1}{2\pi} \int_{-\infty}^{+\infty} h(\tau) e^{ir\tau} \dd \tau = \frac{\cos(rT)}{\sqrt{\pi}} e^{-r^2/4}
  \end{equation}
  and
  \begin{equation}\label{eqn:k&q}
    k(u) = -\frac{1}{\pi} \int_{u}^{+\infty} \frac{\dd q(v)}{\sqrt{v-u}}, \quad
    q(v) = \frac{1}{2} g(2\log(\sqrt{v+1}+\sqrt{v})).
  \end{equation}
  Here $k(\cdot)$ is the Selberg transform of $h(\cdot)$ (see e.g. \cite[eq. (1.64)]{iwaniec2002spectral}).
  Let
  \[
	K(z,w) = \sum_{\gamma \in \Gamma} \chi_\nu(\gamma)
    J(\gamma,w)^{2\kappa} (( z, \gamma w))^{\kappa/2} k(z, \gamma w)
  \]
  be the automorphic kernel of weight $\kappa$ and character $\chi_\nu$,
  where
  \[
    k(z,w) = k\left(u(z,w)\right), \quad
    u(z,w) = \frac{|z-w|^2}{4\Im z \Im w}.
  \]
Here we have used the notation of \cite{patterson1975laplacian},
  \[
	((z,w)) = \frac{w-\bar{z}}{z-\bar{w}} . 
  \]
  Note $k(z,w)$ is a point pair invariant function, that is $k(\gamma z, \gamma w)=k(z,w)$ for all $\gamma \in \Gamma$.
  The spectral expansion (\cite{Roelcke}) of $K(z,w)$ is
  \begin{align*}\label{eqn:spectralExpansion}
    K(z,w) &=\sum_{j} h(t_j) f_j(z) \overline{f_j(w)}
    + \sum_{\mathfrak{a} \text{ cusp}  }
    \frac{1}{4\pi} \int_{\mathbb{R}} h(t) E_{\mathfrak{a},\kappa,\chi}^{(M)} (z, \tfrac12+ it) \overline{ E_{\mathfrak{a},\kappa,\chi}^{(M)} (w, \tfrac12+ it) } \dd t.
  \end{align*}
  Let us also introduce the Maass raising operator
  \[
    K_k := (z-\bar{z})\frac{\partial}{\partial z} + k = iy\frac{\partial}{\partial x}+ y \frac{\partial}{\partial y} +k.
  \]
For a Maass form $f$ of weight $\kappa$, define, for each $n \ge 0$,
  \begin{equation}\label{eqn:f_n}
    f_n(w) = \overline{\overline{f}_{-n}(w)}
    = \frac{1}{n!} K_{\kappa/2+n-1}\cdots K_{\kappa/2+1} K_{\kappa/2} f(w).
  \end{equation}
  Finally, recall that the geodesic polar coordinates of the point $w \in \mathbb H$ are given by
  \begin{equation} \label{eq:polar}
    \frac{z-w}{z-\bar{w}} = \tanh\left(\frac{r}{2}\right) e^{i\varphi}, \quad
    \rho = \tanh\left(\frac{r}{2}\right)
    \quad  \textrm{and} \quad
    u=\frac{\rho^2}{1-\rho^2}.
  \end{equation}

\subsection{Proof of Lemma \ref{lemma:triple} }
 Let
  \[
    \mathfrak{I}(w) := \int_{\Gamma\backslash \mathbb{H}} \Psi(z) \overline{\theta_{\chi,t}(z)} K(z,w) \dd\mu(z).
  \]
  Then by the Parseval identity we have
  \[
    \mathfrak{I}(w) = \sum_{j} h(t_j) \langle f_j, \overline{\Psi}\theta_{\chi,t} \rangle \overline{f_j(w)} + \mathrm{cont.}
  \]
  and
  \[
    \| \mathfrak{I} \|_2^2 = \sum_{j} | h(t_j) \langle f_j, \overline{\Psi}\theta_{\chi,t} \rangle |^2 + \mathrm{cont.}.
  \]
  Now it is  not hard to see that for $T\leq t_j \leq T+1$ we have
  \begin{equation}\label{eqn:3prod<<L2}
    | \langle f_j, \overline{\Psi}\theta_{\chi,t} \rangle |
    \ll \| \mathfrak{I} \|_2.
  \end{equation}

  We now consider $\mathfrak{I}(w)$.
  Since the quantity $((z,w))$ satisfies
  \[
    ((\gamma z, \gamma w)) = J(\gamma,z)^{4} J(\gamma,w)^{-4} ((z,w)),
  \]
  by unfolding, we have
  \[
    \mathfrak{I}(w) = 2 \int_{\mathbb{H}} \Psi(z) \overline{\theta_{\chi,t}(z) \left(\frac{z-\bar{w}}{w-\bar{z}}\right)^{\kappa/2}} k(z,w) \dd\mu(z).
  \]
  Using the geodesic polar coordinates of the point $w \in \mathbb H$, as in \eqref{eq:polar}, we get
  \begin{equation}\label{eqn:I2B}
    \mathfrak{I}(w) = 2 \int_{0}^{\infty} k(u) B(\rho) \dd u,
  \end{equation}
  where
  \[
    B(\rho) := \int_{0}^{2\pi} \Psi(z) \overline{\theta_{\chi,t}(z) \left(\frac{z-\bar{w}}{w-\bar{z}}\right)^{\kappa/2}} \dd \varphi.
  \]

  By \cite[Theorems 1.1 \& 1.2]{Fay}, we have
  \begin{equation}\label{eqn:theta-expansion}
    \theta_{\chi,t}(z) \left(\frac{z-\bar{w}}{w-\bar{z}}\right)^{\kappa/2}
    = \sum_{\ell=0}^{\infty} \rho^\ell
    \left(1-\rho^2\right)^{\kappa/2}
    (\theta_{\chi,t})_{\ell}(w) e^{i\ell \varphi}
  \end{equation}
  and
  \[
    \Psi(z) = \sum_{m\in\mathbb{Z}} \Psi_m(w) P_{1/2+it_\psi,0}^{m}(z,w) e^{im\varphi},
  \]
  where
  \[
    P_{s,0}^m (z,w) = \rho^{|m|}
    \left(1-\rho^2\right)^{s}
    F\left(s,s+|m|;1+|m|;\rho^2\right).
  \]
  Hence we can write
  \begin{equation}\label{eqn:B(rho)}
    B(\rho) = 2\pi \sum_{\ell=0}^{\infty}
    \Psi_\ell(w) \overline{(\theta_{\chi,t})_{\ell}(w)} \rho^{2\ell}
    \left(1-\rho^2\right)^{\kappa/2+1/2+it_\psi}
    F\left(\frac{1}{2}+it_\psi,\frac{1}{2}+it_\psi+\ell; 1+\ell;\rho^2\right).
  \end{equation}

  Similar to \cite[Lemma 2]{sarnak1994integrals}, by $\int_{1}^{2}\int_{0}^{2\pi} |\Psi(z)|^2 \dd \varphi \dd \xi \ll 1$ where the integral is on the geodesic annulus with center $w$ and radius $r$ with $\tanh^2(r/2) = 1-\xi/L$  we can prove
  \begin{equation}\label{eqn:sum_Psi<<}
    \sum_{0\leq \ell\leq L} (1+\ell) |\Psi_\ell(w)|^2 \ll L^2.
  \end{equation}
  Now we deal with the theta function $\theta_{\chi,t}(z)$. We consider the integral $\int_{1}^{2}\int_{0}^{2\pi} |\theta_{\chi,t}(z)|^2 \dd \varphi \dd \xi$. Write $w=u+iv$. Then for $z=x+iy$ in the geodesic annulus with center $w$ and radius $r$ with $\tanh^2(r/2) = 1-\xi/L\in[1-2/L,1-1/L]$, we have $1-4e^{-r}+O(e^{-2r})=1-\xi/L$ and $r=\log(L/4\xi)+O(\xi/L)$. Hence $v/L  \leq y \leq Lv$. Hence
  \[
    |\theta_{\chi,t}(z)| \leq y^{1/4} \sum_{n\in\mathbb{Z}} e^{-2\pi tn^2 y}
    \ll y^{1/4}+y^{-1/4}t^{-1/2}
    \ll L^{1/4} \big( v^{1/4}+v^{-1/4}t^{-1/2} \big),
  \]
  if $\chi$ is even, and
  \[
    |\theta_{\chi,t}(z)| \leq y^{3/4} \sum_{n\in\mathbb{Z}_{\neq0}} e^{-2\pi tn^2 y}
    \ll y^{1/4}t^{-1/2}
    \ll L^{1/4} v^{1/4}t^{-1/2} ,
  \]
  if $\chi$ is odd.
  Hence
  \[
    \int_{1}^{2}\int_{0}^{2\pi} |\theta_{\chi,t}(z)|^2 \dd \varphi \dd \xi
    \ll L^{1/2} \big( v^{1/4}+v^{-1/4} \big)^2.
  \]
  On the other hand, by \eqref{eqn:theta-expansion}, we have
  \[
    \sum_{\ell \leq L/100} |(\theta_{\chi,t})_\ell(w)|^2
    \int_{1}^{2} (1-\xi/L)^{\ell} (\xi/L)^{1/2} \dd \xi
    \ll \int_{1}^{2}\int_{0}^{2\pi} |\theta_{\chi,t}(z)|^2 \dd \varphi \dd \xi.
  \]
  So we get
  \begin{equation}\label{eqn:sum_theta<<}
    \sum_{\ell \leq L} |(\theta_{\chi,t})_\ell(w)|^2
    \ll L \big( v^{1/2}+v^{-1/2} \big).
  \end{equation}

  We claim that the function $B(\rho)$ extends to an even analytic function of $z=\rho$ in $|z|<1$ and satisfies $B(z)\ll (1-|z|)^{-1}( v^{1/4}+v^{-1/4})$.
  Indeed, by \cite[eq. (9.111)]{GR}, we have
  \begin{align*}
    & \Big|(1-z)^{1/2+it_\psi} F\left(\frac{1}{2}+it_\psi,\frac{1}{2}+it_\psi+\ell; 1+\ell;z\right)\Big| \\
     & = \Big|  \frac{(1-z)^{1/2+it_\psi} \Gamma(1+\ell)}{\Gamma(1/2+it_\psi)\Gamma(1/2+\ell-it_\psi)}
     \int_{0}^{1} x^{-1/2+it_\psi} (1-x)^{\ell-1/2-it_\psi} (1-xz)^{-1/2-\ell-it_\psi} \dd x \Big| \\
     & \ll (1+\ell)^{1/2} \int_{0}^{1} x^{-1/2} (1-x)^{\ell-1/2} |1-xz|^{-1/2-\ell} |1-z|^{1/2} \dd x.
  \end{align*}
  For $x\in[0,1]$ and $|z|<1$, we have $|1-xz|\geq 1-x$ and $|1-xz|^{-1}|1-z|\ll1$. Hence
  \begin{multline}\label{eqn:2F1<<}
    \Big|(1-z)^{1/2+it_\psi} F\left(\frac{1}{2}+it_\psi,\frac{1}{2}+it_\psi+\ell; 1+\ell;z\right)\Big|  \\
    \ll (1+\ell)^{1/2} \int_{0}^{1} x^{-1/2} (1-x)^{-1/2} |1-xz|^{-1/2} |1-z|^{1/2} \dd x
    \ll (1+\ell)^{1/2}.
  \end{multline}
  Since $(1-z)^{1/2+it_\psi} F\left(\frac{1}{2}+it_\psi,\frac{1}{2}+it_\psi+\ell; 1+\ell;z\right)$ is holomorphic in $|z|<1$, we have established the analytic continuation of $B(z)$ to $|z|<1$.
  Now, by \eqref{eqn:B(rho)}, \eqref{eqn:2F1<<}, we have
    \begin{multline*}
    B(z) \ll \sum_{\ell=0}^{\infty}
    |\Psi_\ell(w)| |(\theta_{\chi,t})_{\ell}(w)| |z|^{2\ell}
    |1-z^2|^{\frac{\kappa}{2}} (1+\ell)^{1/2}
    \\
    \ll \sum_{\ell=0}^{\infty}
    |\Psi_\ell(w)| |(\theta_{\chi,t})_{\ell}(w)| |z|^{2\ell}
     (1+\ell)^{1/2}.
  \end{multline*}
  Using summation by parts, we obtain
  \begin{equation*}
    B(z)\ll 1+ \lim_{X\rightarrow\infty} \int_{1}^{X} \Big(\sum_{\ell\leq x} |\Psi_\ell(w)| |(\theta_{\chi,t})_{\ell}(w)|
     (1+\ell)^{1/2}\Big) \dd  |z|^{2 x} .
  \end{equation*}
  By \eqref{eqn:sum_Psi<<}, \eqref{eqn:sum_theta<<}, and Cauchy, we get
  \begin{multline}\label{eqn:B(z)<<}
    B(z)
     \ll 1+ \lim_{X\rightarrow\infty} \int_{1}^{X} x^{3/2} \big( v^{1/4}+v^{-1/4}\big) \dd  |z|^{2 x}
     \ll 1+ \big( v^{1/4}+v^{-1/4}\big) \lim_{X\rightarrow\infty} \int_{1}^{X} x^{2} \dd  |z|^{2 x} \\
    \ll (\log|z|)^{-1} \big( v^{1/4}+v^{-1/4}\big)
    \ll (1-|z|)^{-1} \big( v^{1/4}+v^{-1/4}\big).
  \end{multline}

  By \eqref{eqn:I2B} and making a change of variable $\cosh(r)=1+2u$, we have
  \[
    \mathfrak{I}(w) 
    = \int_{0}^{\infty} k\Big(\frac{\cosh(r)-1}{2}\Big) B\Big(\tanh\Big(\frac{r}{2}\Big)\Big) \sinh(r) \dd r.
  \]
  By the first identity in \eqref{eqn:k&q} and making a change of variable $v=\frac{\cosh(z)-1}{2}$, we have
  \begin{align*}
    \mathfrak{I}(w) &
    = \int_{0}^{\infty} \Big( -\frac{\sqrt{2}}{\pi} \int_{r}^{\infty} \frac{\dd q(\frac{\cosh(z)-1}{2})}{\sqrt{\cosh(z)-\cosh(r)}} \Big) B\Big(\tanh\Big(\frac{r}{2}\Big)\Big) \sinh(r) \dd r .
  \end{align*}
  By the second identity in \eqref{eqn:k&q}, for $z>0$, we have
  \[
    q\left(\frac{\cosh(z)-1}{2}\right) = \frac{1}{2} g\left(2\log \left(\sqrt{\frac{e^z+e^{-z}+2}{4}}+\sqrt{\frac{e^z+e^{-z}-2}{4}}\right)\right)
    = \frac{1}{2} g(z).
  \]
  By \eqref{eqn:g}, we get
  \begin{align*}
    \mathfrak{I}(w) &
    = \int_{0}^{\infty} \Big( \frac{1}{\sqrt{2}\pi^{3/2}} \int_{r}^{\infty} \frac{T \sin(zT)+ \frac{z}{2} \cos(zT)}{\sqrt{\cosh(z)-\cosh(r)}} e^{-z^2/4} \dd z \Big) B\Big(\tanh\Big(\frac{r}{2}\Big)\Big) \sinh(r) \dd r \\
    & = \frac{1}{\sqrt{2}\pi^{3/2}} \int_{0}^{\infty}
    \Big( \int_{0}^{z} \frac{B\left(\tanh\left(\frac{r}{2}\right)\right) \sinh(r)}{\sqrt{\cosh(z)-\cosh(r)}} \dd r \Big)
    \Big(T \sin(zT)+\frac{z}{2}\cos(zT) \Big) e^{-z^2/4} \dd z.
  \end{align*}
  We know the function
  \begin{equation}\label{eqn:H}
    H(z)=\int_{0}^{z} \frac{B\left(\tanh\left(\frac{r}{2}\right)\right) \sinh(r)}{\sqrt{\cosh(z)-\cosh(r)}} \dd r
    = \int_{0}^{1} \frac{B\left(\tanh\left(\frac{\xi z}{2}\right)\right) \sinh(\xi z)}{\sqrt{2\frac{\sinh((z(1+\xi)/2))}{z}\cdot \frac{\sinh((z(1-\xi)/2))}{z}}} \dd \xi
  \end{equation}
  is holomorphic in $|\Im(z)|<\pi/2$ and is odd (cf. \cite[p. 259]{sarnak1994integrals}).
  Indeed, we have
  \[
    |\tanh(\frac{x+iy}{2})|^2 = 1-\frac{4e^{x}\cos y}{e^{2x}+1+2e^x \cos y}
    \leq 1-\frac{1}{4 T e^{|x|}},
  \]
  if $x\in\mathbb{R}$ and $y\in[-\pi/2+1/T,\pi/2-1/T]$.
  Hence
  \begin{align*}
     \mathfrak{I}(w) & = \frac{1}{(2\pi)^{3/2}} \int_{-\infty}^{\infty} H(z) e^{-z^2/4} \Big(T \sin(zT)+\frac{z}{2}\cos(zT)\Big) \dd z \\
     & =\frac{1}{(2\pi)^{3/2}} \int_{-\infty}^{\infty} H(z) e^{-z^2/4} \Big(\frac{1}{2i} T e^{izT}+\frac{z}{4}e^{izT}\Big) \dd z
     \\
     & \hskip 30pt + \frac{1}{(2\pi)^{3/2}} \int_{-\infty}^{\infty} H(z) e^{-z^2/4} \Big(\frac{-1}{2i} T e^{-izT}+\frac{z}{4}e^{-izT}\Big) \dd z.
  \end{align*}
  Now we shift the contour of integration from the real axis $\tmop{Im}(z)=0$ to the line $\Im(z)=\pi/2-1/T$ in the first integral and to the line $\Im(z)=-\pi/2+1/T$ in the second integral. This shift is justified by the rapid decay of $e^{-z^2/4}$.
  We will only bound the first integral, since second one can be handled similarly and satisfies the same bound.
  We first bound $H(z)$ with $z=x+i(\pi/2-1/T)$. For $\xi\in(0,1)$, we have
  \[
    \sinh(\xi z) \ll e^{|x|}, \quad
    \sinh((z(1\pm\xi)/2)) \gg (|x|+1/T)(1\pm\xi).
  \]
  By \eqref{eqn:B(z)<<} and \eqref{eqn:H}, we get
  \[
    H(x+i(\pi/2-1/T)) \ll T e^{2|x|} \frac{1+|x|}{1/T+|x|} \big( v^{1/4}+v^{-1/4}\big).
  \]
  Thus we obtain
  \begin{multline*}
    \mathfrak{I}(w) \ll
    \int_{\mathbb{R}} T e^{2|x|} \frac{1+|x|}{1/T+|x|} \big( v^{1/4}+v^{-1/4}\big) e^{-x^2/4} (T+|x|) e^{-\frac{\pi}{2}T} \dd x \\
    \ll
    T^2 (\log T) e^{-\frac{\pi}{2}T} \big(v^{1/4}+v^{-1/4}\big).
  \end{multline*}
  Hence $\|\mathfrak{I}\|_2 \ll T^C e^{-\frac{\pi}{2}T} M^A$. Here one may use \cite[Proposition 2.5]{iwaniec1997topics} for an explicit choice of the representatives of $\Gamma_0(M)\backslash \SL_2(\mathbb{Z})$ to compute the integral of $w$.
  Our lemma follows from \eqref{eqn:3prod<<L2} easily.
\qed

\section{Explicit residue of Eisenstein series}\label{appendix:Eisenstein}

Recall that
\begin{displaymath}
  E_{\mathfrak{a},\kappa}(z; s) = \sum_{\gamma \in \Gamma_{\mathfrak{a}}\backslash \Gamma_0(M)}  J(\sigma_{\mathfrak{a}}^{-1}\gamma, z)^{-2\kappa} \Im(\sigma_{\mathfrak{a}}^{-1}\gamma z)^s.
\end{displaymath}
We consider the case $\mathfrak{a}=\infty$ and $\kappa=1/2$.
In this appendix, we will prove the following lemma.
\begin{lemma}\label{lemma:res}
  Assume that  $M=2^{\beta_0} p_1^{\beta_1} p_2^{\beta_2}$ where $\beta_0\geq2$, $\beta_1\geq0$, $\beta_2\geq0$, and  $p_1\equiv p_2 \equiv 3\pmod{4}$ are two distinct primes. Write $p_0=2$. Then
  \begin{equation*}
  \Res\limits_{s=3/4} E_{\infty,1/2}(z; s)
  = \frac{\pi}{4 \zeta_{(M)}(1)\zeta_M(2)} \Big( \prod_{j=0}^{2} p_j^{-[\frac{\beta_j+1}{2}]} \Big) \theta_{\chi_1,t_M}(z),
  \end{equation*}
  where $t_M=2^{-2} \prod_{j=0}^{2} p_j^{2[\frac{\beta_j}{2}]}$.
  Here $[x]$ is the floor function, $\zeta_{(M)}(s)=\prod_{p\mid M} (1-p^{-s})^{-1}$, and $\zeta_{M}(s)=\prod_{p\nmid M} (1-p^{-s})^{-1}$ for $\Re(s)>1$.
\end{lemma}

\begin{proof}
Recall that we have the Fourier expansion (see \eqref{eqn:FE-ES})
\begin{multline} \label{eqn:E=sum}
  E_{\infty,1/2}(z; s)   = y^s + \frac{\pi 4^{1-s} e\left(-\frac{1}{8}\right) \Gamma(2s-1)}{\Gamma(s+1/4)\Gamma(s-1/4)} \phi(0, s)y^{1-s} \\
  + \sum_{n \not=0}
  \frac{\pi^s e\left(-\frac{1}{8}\right) |n|^{s-1}}{\Gamma(s+\sgn(n)\frac{1}{4})} \phi(n, s) W_{\sgn(n)\frac{1}{4}, s-\frac{1}{2}}(4\pi|n|y) e(nx),
\end{multline}
where
\begin{equation*}
  \phi(n, s)
  = \sum_{\substack{c\geq1 \\ M\mid c}} \sum_{\substack{1 \leq d \leq c\\ (d,c)=1}}
    \left(\frac{c}{d}\right) \epsilon_d \; e\left(\frac{nd}{c}\right) c^{-2s}.
\end{equation*}
Let $c=M' q$ where $(M',q)=1$ and $M \mid M'\mid M^{\infty}$, where the notation $M'|M^{\infty}$ means that $M'$ divides a sufficiently large power of $M$. By using quadratic reciprocity and the Chinese remainder theorem (see  \cite[Lemma 1]{sturm1980special}), we have
\begin{equation*}
  \sum_{\substack{1 \leq d \leq c\\ (d,c)=1}}
  \left(\frac{c}{d}\right) \epsilon_d \; e\left(\frac{nd}{c}\right) \\
  =
  \left(\frac{-1}{q}\right) \epsilon_q  \sum_{d=1}^{q} \left(\frac{d}{q}\right) e\left(\frac{nd}{q}\right)
  \sum_{d'=1}^{M'}  \epsilon_{d'} \left(\frac{M'}{d'}\right) e\left(\frac{nd'}{M'}\right).
\end{equation*}
Hence, we have
\begin{equation} \label{eqn:phi(n)=bc}
  \phi(n, s)
  =  b(n,s,\omega) c(n,s),
\end{equation}
where the character $\omega(\cdot) = \left(\frac{M}{\cdot}\right)$,
\[
  b(n,s,\omega) := \sum_{q=1}^{\infty} \left(\frac{-M}{q}\right) \epsilon_q \; \omega(q) q^{-2s} \sum_{d=1}^{q} \left(\frac{d}{q}\right) e\left(\frac{nd}{q}\right)
\]
and
\[
  c(n,s) := \sum_{\substack{M'\\ M\mid M'\mid M^{\infty}}}
  \left( \sum_{d=1}^{M'}  \epsilon_{d} \left(\frac{M'}{d}\right) e\left(\frac{nd}{M'}\right) \right) (M')^{-2s}.
\]
For $n=tm^2$ with a positive integer $m$ and square-free integer $t$, we have (see \cite[Proposition 1]{Shimura1975})
\begin{equation}\label{eqn:b(n)}
  b(n,s,\omega) = \frac{L_M(2s-1/2,\omega_1)}{L_M(4s-1,\omega_2)} \sum_{\substack{\ell_1,\ell_2\geq1 \\ (\ell_1\ell_2,M)=1,\; \ell_1\ell_2\mid m}} \mu(\ell_1) \omega_1(\ell_1) \omega_2(\ell_2) \ell_1^{-2s+1/2} \ell_2^{2-4s},
\end{equation}
where $\mu$ is the M\"{o}bius function, and (see \cite[eq. (3.9)]{Shimura1975})
\begin{equation}\label{eqn:b(0)}
  b(0,s,\omega) = \frac{L_M(4s-2,\omega_2)}{L_M(4s-1,\omega_2)}.
\end{equation}
Here $\omega_1$ and $\omega_2$ are primitive characters defined by
\begin{align*}
  \omega_1(\ell) & = \left(\frac{tM}{\ell}\right) \omega(\ell)
  = \left(\frac{t}{\ell}\right)  \qquad \textrm{for } (\ell,tM)=1, \\
  \omega_2(\ell) & = \omega(\ell)^2 = 1 \hskip 80pt \textrm{for } (\ell,M)=1,
\end{align*}
and $L_M(s,\omega) = \sum_{(\ell,M)=1} \omega(\ell) \ell^{-s}$.
Thus we have
\begin{equation} \label{eqn:Res_b}
  \Res_{s=3/4} b(0,s,\omega) = \frac{1}{4 \zeta_{(M)}(1) \zeta_{M}(2)}, \quad \textrm{and} \quad
  \Res_{s=3/4} b(n,s,\omega) = \frac{1}{2 \zeta_{(M)}(1) \zeta_{M}(2)},
\end{equation}
if $n=m^2$, and $\Res_{s=3/4} b(n,s,\omega)=0$ otherwise. Here we have used the fact
\[
  \sum_{\substack{\ell_1,\ell_2\geq1 \\ (\ell_1\ell_2,M)=1,\; \ell_1\ell_2\mid m}} \mu(\ell_1) \ell_1^{-1} \ell_2^{-1}
   = \sum_{\substack{\ell\geq1 \\ (\ell,M)=1,\; \ell\mid m}} \frac{1}{\ell}  \sum_{\ell_1\mid \ell}\mu(\ell_1) = 1.
\]

\medskip

Now we consider $c(n,s)$. Let $G_n(\chi')=\sum_{d=1}^{r} \chi'(d) e(dn/r)$ be the usual Gauss sum where $\chi'$ is a Dirichlet character mod $r$.
Note that
\[
  \epsilon_d = \frac{1+i}{2} \chi_{-4}^0(d) + \frac{1-i}{2} \chi_{-4}(d),
\]
where $\chi_{-4}=(\frac{-4}{\cdot})$ is the primitive quadratic character mod $4$.   We have  that
\[
  \sum_{d=1}^{M'}  \epsilon_{d} \left(\frac{M'}{d}\right) e\left(\frac{nd}{M'}\right)
  = \frac{1+i}{2} G_n\left(\left(\frac{M'}{\cdot}\right)\right) + \frac{1-i}{2} G_n\left(\chi_{-4}\left(\frac{M'}{\cdot}\right)\right).
\]
By the Gauss sums as in \cite[Lemma 2]{sturm1980special}, we know that $c(n,s)$ is a finite Dirichlet series if $n\neq0$.
We have
\[
  c(n,s) = \sum_{\substack{M'\\ M\mid M'\mid M^{\infty}}}
  \left( \frac{1+i}{2} G_n\left(\left(\frac{M'}{\cdot}\right)\right) + \frac{1-i}{2} G_n\left(\chi_{-4}\left(\frac{M'}{\cdot}\right)\right) \right) (M')^{-2s}.
\]

Recall that $M=2^{\beta_0} p_1^{\beta_1} p_2^{\beta_2}$ where $\beta_0\geq2$, $\beta_1\geq0$, $\beta_2\geq0$, and  $p_1\equiv p_2 \equiv 3\pmod{4}$ are two distinct primes.
Let $M'=2^{k_0}p_1^{k_1} p_2^{k_2}$ with $k_0\geq \beta_0$, $k_1\geq \beta_1$ and $k_2\geq \beta_2$.  Define $M_0'=p_1^{k_1}p_2^{k_2}$ and $\overline{M_0'}$ the inverse of $M_0'$ mod $2^{k_0}$, and similarly $M_1'=2^{k_0} p_2^{k_2}$, $\overline{M_1'} M_1'\equiv 1 \pmod{p_1^{k_1}}$, and $M_2'=2^{k_0} p_1^{k_1}$ and $\overline{M_2'} M_2'\equiv 1 \pmod{p_2^{k_2}}$. By the Chinese remainder theorem we have
      \begin{align*}
        G_n\left(\left(\frac{M'}{\cdot}\right)\right) & = G_{n\overline{M_0'}}\left(\left(\frac{(-1)^{k_1+k_2}2^{k_0}}{\cdot}\right)\right) G_{n\overline{M_1'}}\left(\left(\frac{(-p_1)^{k_1}}{\cdot}\right)\right) G_{n\overline{M_2'}}\left(\left(\frac{(-p_2)^{k_2}}{\cdot}\right)\right) \\
        & =
        G_{n\overline{M_0'}}\left( \chi_{-4}^{k_1+k_2} \chi_{8}^{k_0}  \chi_{2^{k_0}}^0 \right) G_{n\overline{M_1'}}\left(\chi_{-p_1}^{k_1} \chi_{p_1^{k_1}}^0\right) G_{n\overline{M_2'}}\left(\chi_{-p_2}^{k_2} \chi_{p_2^{k_2}}^0\right) \\
        & =  \chi_{-4}^{k_1+k_2}(p_1^{k_1}p_2^{k_2}) \chi_{8}^{k_0} (p_1^{k_1}p_2^{k_2}) \chi_{-p_1}^{k_1}(2^{k_0} p_2^{k_2}) \chi_{-p_2}^{k_2}(2^{k_0} p_1^{k_1}) \\
        & \hskip 60pt \cdot
        G_{n}\left( \chi_{-4}^{k_1+k_2} \chi_{8}^{k_0}  \chi_{2^{k_0}}^0 \right) G_{n}\left(\chi_{-p_1}^{k_1} \chi_{p_1^{k_1}}^0\right) G_{n}\left(\chi_{-p_2}^{k_2} \chi_{p_2^{k_2}}^0\right) ,
      \end{align*}
      and
      \begin{align*}
        G_n\left(\left(\frac{-M'}{\cdot}\right)\right) & = G_{n\overline{M_0'}}\left( \chi_{-4}^{k_1+k_2+1} \chi_{8}^{k_0}  \chi_{2^{k_0}}^0 \right) G_{n\overline{M_1'}}\left(\chi_{-p_1}^{k_1} \chi_{p_1^{k_1}}^0\right) G_{n\overline{M_2'}}\left(\chi_{-p_2}^{k_2} \chi_{p_2^{k_2}}^0\right) \\
        & =  \chi_{-4}^{k_1+k_2+1}(p_1^{k_1}p_2^{k_2}) \chi_{8}^{k_0} (p_1^{k_1}p_2^{k_2}) \chi_{-p_1}^{k_1}(2^{k_0} p_2^{k_2}) \chi_{-p_2}^{k_2}(2^{k_0} p_1^{k_1}) \\
        & \hskip 60pt \cdot
        G_{n}\left( \chi_{-4}^{k_1+k_2+1} \chi_{8}^{k_0}  \chi_{2^{k_0}}^0 \right) G_{n}\left(\chi_{-p_1}^{k_1} \chi_{p_1^{k_1}}^0\right) G_{n}\left(\chi_{-p_2}^{k_2} \chi_{p_2^{k_2}}^0\right) .
      \end{align*}
      Here $\chi_1 =1$, $\chi_{-4}= \left(\frac{-4}{\cdot}\right)$, $\chi_{8}= \left(\frac{2}{\cdot}\right)$ and $\chi_{-8}= \left(\frac{-2}{\cdot}\right)$, and $\chi_{-p}=\left(\frac{-p}{\cdot}\right)$ if $p\equiv 3\pmod{4}$.
      We understand $\chi_{-4}^{k_1+k_2} \chi_{8}^{k_0}  \chi_{2^{k_0}}^0$ as a character modulo $2^{k_0}$. Similarly for other characters in the Gauss sums.
      Recall that $G_1(\chi_1)=1$, $G_1(\chi_{-4})=2i$, $G_1(\chi_{8})=2\sqrt{2}$ and $G_1(\chi_{-8})=2\sqrt{2}i$.
       Assume $n=2^{2\alpha_0} p_1^{2\alpha_1} p_2^{2\alpha_2} n_0$, where $(n_0,2p_1p_2)=1$ and $n_0$ is a square. Then we have
       \begin{align}
         G_{n}\left( \chi_{8}^{k_0}  \chi_{2^{k_0}}^0 \right) & = \left\{
         \begin{array}{ll}
           \varphi(2^{k_0}), & \textrm{if $2\mid k_0$ and $k_0 \leq 2\alpha_0$,} \\
           2\sqrt{2} \cdot 2^{2\alpha_0}, & \textrm{if $k_0 = 2\alpha_0 + 3$,} \\
           0 , & \textrm{otherwise},
         \end{array} \right.
         \label{eqn:G8} \\
         G_{n}\left( \chi_{-4} \chi_{8}^{k_0}  \chi_{2^{k_0}}^0 \right) & = \left\{
         \begin{array}{ll}
           2i \cdot 2^{2\alpha_0}, & \textrm{if $k_0=2\alpha_0+2$,} \\
           2\sqrt{2}i \cdot 2^{2\alpha_0}, & \textrm{if $k_0 = 2\alpha_0 + 3$,} \\
           0 , & \textrm{otherwise},
         \end{array} \right.
         \label{eqn:G-8} \\
         G_{n}\left(\chi_{-p_j}^{k_j} \chi_{p_j^{k_j}}^0\right) & = \left\{
         \begin{array}{ll}
           \varphi(p_j^{k_j}), & \textrm{if $2\mid k_j \leq 2\alpha_j$,} \\
           i\sqrt{p_j} \cdot p_j^{2\alpha_j} 
           , & \textrm{if $k_j = 2\alpha_j + 1$,} \\
           0 , & \textrm{otherwise},
         \end{array} \right. \label{eqn:Gp}
       \end{align}
       for $j=1,2$.
       Here we used the fact $G_1\left(\chi_{-p}\right)=i\sqrt{p}$ if $p\equiv 3\pmod{4}$. Hence
       \begin{align*}
         c(n,s) & =\frac{1+i}{2} \sum_{k_0\geq\beta_0} \sum_{k_1\geq\beta_1} \sum_{k_2\geq\beta_2} \frac{\chi_{-4}^{k_1+k_2}(p_1^{k_1}p_2^{k_2}) \chi_{8}^{k_0} (p_1^{k_1}p_2^{k_2}) \chi_{-p_1}^{k_1}(2^{k_0} p_2^{k_2}) \chi_{-p_2}^{k_2}(2^{k_0} p_1^{k_1}) } {(2^{k_0}p_1^{k_1}p_2^{k_2})^{2s}}
         \\
         & \hskip 90pt \cdot G_{n}\left( \chi_{-4}^{k_1+k_2} \chi_{8}^{k_0}  \chi_{2^{k_0}}^0 \right) G_{n}\left(\chi_{-p_1}^{k_1} \chi_{p_1^{k_1}}^0\right) G_{n}\left(\chi_{-p_2}^{k_2} \chi_{p_2^{k_2}}^0\right)
         \\
         & \hskip -10pt + \frac{1-i}{2} \sum_{k_0\geq\beta_0} \sum_{k_1\geq\beta_1} \sum_{k_2\geq\beta_2}
         \frac{ \chi_{-4}^{k_1+k_2+1}(p_1^{k_1}p_2^{k_2}) \chi_{8}^{k_0} (p_1^{k_1}p_2^{k_2}) \chi_{-p_1}^{k_1}(2^{k_0} p_2^{k_2}) \chi_{-p_2}^{k_2}(2^{k_0} p_1^{k_1})
        }{(2^{k_0}p_1^{k_1}p_2^{k_2})^{2s}} \\
        & \hskip 90pt \cdot
        G_{n}\left( \chi_{-4}^{k_1+k_2+1} \chi_{8}^{k_0}  \chi_{2^{k_0}}^0 \right) G_{n}\left(\chi_{-p_1}^{k_1} \chi_{p_1^{k_1}}^0\right) G_{n}\left(\chi_{-p_2}^{k_2} \chi_{p_2^{k_2}}^0\right) .
       \end{align*}
       If $2\alpha_0+3< \beta_0$, then $G_{n}\left( \chi_{8}^{k_0}  \chi_{2^{k_0}}^0 \right)=G_{n}\left( \chi_{-4} \chi_{8}^{k_0}  \chi_{2^{k_0}}^0  \right) =0$ for all $k_0\geq \beta_0$, and hence $c(n,s)=0$.
       Similarly, we have $c(n,s)=0$ if $2\alpha_j+1<\beta_j$ for either $j=1$ or $2$.
       Now we assume
       \begin{equation}
         \alpha_0\geq \frac{\beta_0-3}{2}, \quad
         \textrm{$\alpha_j\geq \frac{\beta_j-1}{2}$ for both $j=1$ and $2$}.
       \end{equation}

       We split the above sums to four cases depending on the parity of $k_1$ and $k_2$. We have
       \begin{equation}\label{eqn:c=cjj}
         c(n,s) = c_{11}(n,s)+c_{12}(n,s) +c_{21}(n,s)+c_{22}(n,s),
       \end{equation}
       where
by \eqref{eqn:G8} and \eqref{eqn:G-8} we have
\begin{align*} 
         c_{11}(n,s ) & = \frac{1+i}{2}\sqrt{p_1p_2} p_1^{2\alpha_1} p_2^{2\alpha_2}
         \sum_{\substack{\beta_0\leq k_0\leq 2\alpha_0+2 \\ 2\mid k_0}} \frac{1} {(2^{k_0}p_1^{2\alpha_1+1}p_2^{2\alpha_2+1})^{2s}}
         \varphi(2^{k_0})
         \\
         & \hskip 10pt
         +(1+i)\sqrt{p_1p_2} p_1^{2\alpha_1} p_2^{2\alpha_2}
         \frac{1} {(2^{2\alpha_0+3}p_1^{2\alpha_1+1}p_2^{2\alpha_2+1})^{2s}}
         2^{2\alpha_0+3/2} \\
         & = \frac{1+i}{4}
         \frac{p_1^{-1/2}}{(p_1^{2\alpha_1+1})^{2s-1}}
         \frac{p_2^{-1/2}}{(p_2^{2\alpha_2+1})^{2s-1}}
         \Big( \sum_{\substack{\beta_0\leq k_0\leq 2\alpha_0+2 \\ 2\mid k_0}} \frac{1} {(2^{k_0})^{2s-1}}
         +
         \frac{ 2^{1/2} } {(2^{2\alpha_0+3})^{2s-1}}\Big),
       \end{align*}
       \begin{align*} 
         c_{12}(n,s ) & =
         \frac{1+i}{2}
         \sum_{\substack{ \beta_0\leq k_0 \leq 2\alpha_0+2 \\ 2\mid k_0}} \sum_{\substack{ \beta_2\leq k_2 \leq 2\alpha_2 \\ 2\mid k_2}}  \frac{1}{(2^{k_0}p_1^{2\alpha_1+1}p_2^{k_2})^{2s}}
        \varphi(2^{k_0})   \sqrt{p_1} \cdot p_1^{2\alpha_1} \varphi(p_2^{k_2})
         \\
         & \hskip 10pt + (1+i)
           \sum_{\substack{ \beta_2\leq k_2 \leq 2\alpha_2 \\ 2\mid k_2}} \frac{1} {(2^{2\alpha_0 + 3}p_1^{2\alpha_1+1}p_2^{k_2})^{2s}}
          2\sqrt{2} \cdot 2^{2\alpha_0} \sqrt{p_1} \cdot p_1^{2\alpha_1} \varphi(p_2^{k_2}) \\
         & =
         \frac{1+i}{4}
         \frac{p_1^{-1/2}}{(p_1^{2\alpha_1+1})^{2s-1}}
         \sum_{\substack{ \beta_2\leq k_2 \leq 2\alpha_2 \\ 2\mid k_2}}  \frac{1-\frac{1}{p_2}}{(p_2^{k_2})^{2s-1}} \Big( \sum_{\substack{\beta_0\leq k_0\leq 2\alpha_0+2 \\ 2\mid k_0}} \frac{1} {(2^{k_0})^{2s-1}}
         +
         \frac{ 2^{1/2} } {(2^{2\alpha_0+3})^{2s-1}}\Big),
       \end{align*}
  \begin{align*}  
    c_{21}(n,s ) & =\frac{1+i}{2}
    \sum_{\substack{ \beta_0 \leq k_0 \leq 2\alpha_0+2 \\ 2\mid k_0}} \sum_{\substack{ \beta_1 \leq k_1 \leq 2\alpha_1 \\ 2\mid k_1}} \frac{1} {(2^{k_0}p_1^{k_1}p_2^{2\alpha_2+1})^{2s}}
    \varphi(2^{k_0}) \varphi(p_1^{k_1}) \sqrt{p_2} \cdot p_2^{2\alpha_2}
         \\
    & \hskip 10pt + (1+i)
     \sum_{\substack{ \beta_1 \leq k_1 \leq 2\alpha_1 \\ 2\mid k_1}} \frac{1} {(2^{2\alpha_0+3}p_1^{k_1}p_2^{2\alpha_2+1})^{2s}}
    2\sqrt{2} \cdot 2^{2\alpha_0} \varphi(p_1^{k_1}) \sqrt{p_2} \cdot p_2^{2\alpha_2} \\
    & =
         \frac{1+i}{4}
         \sum_{\substack{ \beta_1 \leq k_1 \leq 2\alpha_1 \\ 2\mid k_1}}  \frac{1-\frac{1}{p_1}}{(p_1^{k_1})^{2s-1}}
         \frac{p_2^{-1/2}}{(p_2^{2\alpha_2+1})^{2s-1}}
          \Big( \sum_{\substack{\beta_0 \leq k_0\leq 2\alpha_0+2 \\ 2\mid k_0}} \frac{1} {(2^{k_0})^{2s-1}}
         +
         \frac{ 2^{1/2} } {(2^{2\alpha_0+3})^{2s-1}}\Big),
  \end{align*}
       \begin{align*}  
         c_{22}(n,s ) & =\frac{1+i}{2}
          \sum_{\substack{ \beta_0 \leq k_0 \leq 2\alpha_0+2 \\ 2\mid k_0}} \sum_{\substack{ \beta_1 \leq k_1 \leq 2\alpha_1 \\ 2\mid k_1}}
         \sum_{\substack{ \beta_2 \leq k_2 \leq 2\alpha_2 \\ 2\mid k_2}} \frac{1} {(2^{k_0}p_1^{k_1}p_2^{k_2})^{2s}}
         \varphi(2^{k_0}) \varphi(p_1^{k_1})\varphi(p_2^{k_2})
         \\
          & \hskip 10pt + (1+i)
         \sum_{\substack{ \beta_1 \leq k_1 \leq 2\alpha_1 \\ 2\mid k_1}}
         \sum_{\substack{ \beta_2 \leq k_2 \leq 2\alpha_2 \\ 2\mid k_2}} \frac{1} {(2^{2\alpha_0 + 3}p_1^{k_1}p_2^{k_2})^{2s}}
         2\sqrt{2} \cdot 2^{2\alpha_0} \varphi(p_1^{k_1})\varphi(p_2^{k_2}) \\
         & \hskip -30pt =\frac{1+i}{4}
         \sum_{\substack{ \beta_1 \leq k_1 \leq 2\alpha_1 \\ 2\mid k_1}}
         \frac{1-\frac{1}{p_1}} {(p_1^{k_1})^{2s-1}}
         \sum_{\substack{ \beta_2 \leq k_2 \leq 2\alpha_2 \\ 2\mid k_2}}
         \frac{1-\frac{1}{p_2}} {(p_2^{k_2})^{2s-1}}
         \Big( \sum_{\substack{\beta_0 \leq k_0\leq 2\alpha_0+2 \\ 2\mid k_0}} \frac{1} {(2^{k_0})^{2s-1}}
         +
         \frac{ 2^{1/2} } {(2^{2\alpha_0+3})^{2s-1}}\Big).
       \end{align*}
Thus we have (for $m=2^{\alpha_0} p_1^{\alpha_1} p_2^{\alpha_2}m_0$ with $\alpha_0\geq1$)
\begin{align*}\label{eqn:c(ns)}
  c(m^2,s )
  & = \frac{1+i}{4}  \Big( \sum_{\substack{\beta_0 \leq k_0\leq 2\alpha_0+2 \\ 2\mid k_0}} \frac{1} {(2^{k_0})^{2s-1}}
         +
         \frac{ 2^{1/2} } {(2^{2\alpha_0+3})^{2s-1}}\Big) \\
  & \hskip 90pt  \cdot \prod_{j=1}^{2}\Big( \sum_{\substack{ \beta_j\leq k_j \leq 2\alpha_j \\ 2\mid k_j}} \frac{1-\frac{1}{p_j}} {(p_j^{k_j})^{2s-1}} +
  \frac{p_j^{-1/2}}{(p_j^{2\alpha_j+1})^{2s-1}} \Big).
\end{align*}
Hence
\begin{equation}\label{eqn:cn3/4}
  c(m^2,3/4 ) = \frac{1+i}{2}  \frac{1}{2^{[\frac{\beta_0+1}{2}]} p_1^{[\frac{\beta_1+1}{2}]} p_2{[\frac{\beta_2+1}{2}]}}  .
\end{equation}

If $n=0$ then $G_n\left(\chi_{-4}\left(\frac{M'}{\cdot}\right)\right)=0$ since $\chi_{-4}\left(\frac{M'}{\cdot}\right)$ is odd. Hence
\begin{equation}\label{eqn:c(0)}
  c(0,s ) = \frac{1+i}{2} \sum_{\substack{M'\\ M\mid M'\mid M^{\infty}}}
  \left( \sum_{d=1}^{M'} \left(\frac{M'}{d}\right) \right)  (M')^{-2s}.
\end{equation}
Assume $M= \prod_{p\mid M} p^{\nu_p(M)}$ where $\nu_p(M)$ is defined as $p^{\nu_p(M)} \parallel M$. By the orthogonality of characters, we have $\sum_{d=1}^{M'}  \left(\frac{M'}{d}\right) = 0$ if $M'$ is not a square, and $\sum_{d=1}^{M'} \left(\frac{M'}{d}\right)=\varphi(M')= \varphi(M) M'/M$ if $M'$ is a square. Thus
\[
  c(0,s ) = \frac{1+i}{2} \prod_{\substack{p\mid M\\ 2\nmid \nu_p(M)}}  \sum_{k=0}^{\infty} \frac{\varphi(p^{\nu_p(M)+1+2k})}{p^{(\nu_p(M)+1+2k)2s}}
  \prod_{\substack{p\mid M\\ 2\mid \nu_p(M)}}  \sum_{k=0}^{\infty} \frac{\varphi(p^{\nu_p(M)+2k})}{p^{(\nu_p(M)+2k)2s}}.
\]
If  $M=2^{\beta_0} p_1^{\beta_1} p_2^{\beta_2}$, then we have
\begin{align*}
  c(0,s ) & = \frac{1+i}{2}
  \sum_{k_0=0}^{\infty} \frac{\varphi(2^{2[\frac{\beta_0+1}{2}]+2k_0})}{2^{(2[\frac{\beta_0+1}{2}]+2k_0)2s}} \sum_{k_1=0}^{\infty} \frac{\varphi(p_1^{2[\frac{\beta_1+1}{2}]+2k_1})}{p_1^{(2[\frac{\beta_1+1}{2}]+2k_1)2s}} \sum_{k_2=0}^{\infty} \frac{\varphi(p_2^{2[\frac{\beta_2+1}{2}]+2k_2})}{p_2^{(2[\frac{\beta_2+1}{2}]+2k_2)2s}}.
\end{align*}
Thus  ($p_0=2$)
\begin{equation}\label{eqn:c03/4}
  c(0,3/4 )  = \frac{1+i}{2} \prod_{j=0}^{2} \frac{1}{p_j^{[\frac{\beta_j+1}{2}]}}
   = e(1/8) 2^{-1/2} \prod_{j=0}^{2} \frac{1}{p_j^{[\frac{\beta_j+1}{2}]}}.
\end{equation}

Now we are ready to compute the residue of $E(z; s)$ at $s=3/4$.
Note that we have (\cite[eq. (9.222.1)]{GR})
\[
  W_{\frac{1}{4}, \frac{1}{4}}(4\pi m^2 y)
  = \frac{(4\pi m^2 y)^{3/4} e^{-2\pi m^2 y}}{\sqrt{\pi}} \int_{0}^{\infty} e^{-(4\pi m^2 y)t} t^{-1/2} \dd t
  = (4\pi m^2 y)^{1/4} e^{-2\pi m^2 y}.
\]
By \eqref{eqn:E=sum}, \eqref{eqn:phi(n)=bc}, \eqref{eqn:Res_b}, \eqref{eqn:cn3/4}, \eqref{eqn:c03/4}, we have
\begin{multline} \label{eqn:Res}
  \Res\limits_{s=3/4} E_{\infty,1/2}(z; s)   =
   \frac{\pi}{4\zeta_{(M)}(1)\zeta_M(2)}
   \Big( \prod_{j=0}^{2} p_j^{-[\frac{\beta_j+1}{2}]} \Big) y^{1/4} \\
  + \sum_{m\geq1}
   \frac{\pi^{3/4} e\left(-\frac{1}{8}\right) m^{-1/2}}{2\zeta_{(M)}(1)\zeta_M(2)} c(m^2,3/4 )
  W_{\frac{1}{4}, \frac{1}{4}}(4\pi m^2y) e(m^2 x) \\
  = \frac{\pi}{4 \zeta_{(M)}(1)\zeta_M(2)} \Big( \prod_{j=0}^{2} p_j^{-[\frac{\beta_j+1}{2}]} \Big) \theta_{\chi_1,t_M}(z),
\end{multline}
where $t_M=2^{2[\frac{\beta_0-2}{2}]}\prod_{j=1}^{2} p_j^{2[\frac{\beta_j}{2}]}$.
\end{proof}

\end{document}